\documentclass[11pt,twoside]{article}
\usepackage{bbm}
\usepackage{bm}
\usepackage[numbers,sort&compress]{natbib}
\usepackage{psfrag}
\usepackage{ifpdf}
\usepackage{multicol,graphics}
\usepackage{subfigure}
\usepackage[dvips]{graphicx}
\usepackage{color}% 支持彩色
\usepackage{epsfig}% eps图像
\usepackage{epstopdf}
\usepackage{float}
\usepackage{pgfplots}
\usetikzlibrary{positioning} %为了实现相对位置的设定
\usepackage{xcolor} %为了实现不同的颜色
\usepackage{dsfont}
\usepackage{enumitem}
\usepackage{indentfirst}
\usepackage[below]{placeins}
\usepackage[bookmarks=true,bookmarksopen=true,breaklinks=true,colorlinks=true,linkcolor=blue,
anchorcolor=blue,citecolor=blue,filecolor=blue,menucolor=blue,urlcolor=blue]{hyperref}
\usepackage{amsmath,amssymb,amsthm,mathrsfs,mathtools}
\newcommand{\norm}[1]{\displaystyle \left\| #1 \right\|}
\DeclareMathOperator{\tr}{tr}
\allowdisplaybreaks[3]
\newtheorem{theorem}{Theorem}[section]
\newtheorem{lemma}[theorem]{Lemma}

\newtheorem{proposition}[theorem]{Proposition}
\newtheorem{claim}[theorem]{Claim}
\newtheorem{definition}{Definition}[section]

\newtheorem{remark}[theorem]{Remark}

\newtheorem{step}{Step}
\newcounter{RomanNumber}

\numberwithin{equation}{section}
\newcommand\blfootnote[1]{%
\begingroup
\renewcommand\thefootnote{}\footnote{#1}%
\addtocounter{footnote}{-1}%
\endgroup
}
\allowdisplaybreaks[3]
\begin{document}
\title{\textbf{\Large Front propagation into unstable states for periodic monotone reaction-diffusion systems}}
\author{
Liangliang Deng\textsuperscript{a,}\thanks{Corresponding author.},~
Arnaud Ducrot\textsuperscript{b},~and
Quentin Griette\textsuperscript{b}
\\
{\small \textsuperscript{a}
School of Science and Engineering, The Chinese University of Hong Kong,}\\
{\small Shenzhen, Guangdong 518172, P.R. China}\\
{\small \& School of Mathematics Sciences, University of Science and Technology of China,}\\
{\small Hefei, Anhui 230026, P.R. China,}\\
{\small \textsuperscript{b}
Universit\'{e} Le Havre Normandie, Normandie Univ., } \\
{\small LMAH UR 3821, 76600 Le Havre, France.}
}
\maketitle
\blfootnote{E-mail addresses: dengll@cuhk.edu.cn (L. Deng), arnaud.ducrot@univ-lehavre.fr (A. Ducrot), \\quentin.griette@univ-lehavre.fr (Q. Griette).}
\begin{abstract}
    This paper is concerned with the invasion fronts of spatially periodic monotone reaction-diffusion systems in a multi-dimensional setting.
    We study the pulsating travelling waves that connect the trivial equilibrium,
    for which all components of the state variable are identically equal to zero,
    to a uniformly persistent stationary state,
    for which all components are uniformly positive.
    When the trivial equilibrium is linearly unstable,
    we prove that any pulsating wave travels at a speed no less than that of the linearized system at the equilibrium, namely the minimal speed is linearly determined.
    It is also shown that pulsating waves are strictly monotonic in time whenever the nonlinearity is subhomogeneous.
    Beyond these general qualitative properties,
    the main focus of the paper is to derive sufficient conditions for
    the existence and nonexistence of pulsating waves propagating
    in any given direction.
Our proof of the existence part relies upon a new level of understanding of the multi-dimensional pulsating waves observed from a direction-dependent coordinate system.

\textbf{Keywords:}
monotone systems, semilinear parabolic equations, pulsating waves,
rational direction of propagation.

\textbf{AMS Subject Classification (2020):}
35C07; 35K40, 35K57.

\end{abstract}
\section{Introduction and main results}
We consider  the following quasimonotone system of semilinear parabolic equations
\begin{align}\label{monotone-system}
  \frac{\partial u}{\partial t}+Lu=f(x,u),~~~(t,x)\in\mathbb{R}\times\mathbb{R}^N,
\end{align}
where $u=(u_1,\ldots,u_d)^T$ for some $d>1$ is the unknown of $t\in\mathbb{R}$ and $x\in\mathbb{R}^N$ for some $N\geq1$,
and $f=(f_1,\ldots,f_d)^T:$ $\mathbb{R}^N\times\mathbb{R}^d\to\mathbb{R}^d$ is a  nonlinear vector-valued function satisfying
\begin{equation}\label{quasimonotone-condition}
  \forall u\in\Sigma\subseteq\mathbb{R}^d,~~~~
  \frac{\partial f_i}{\partial u_j}(\cdot,u)\geq0~~
  \text{for all $\,i,j\in\{1,\ldots,d\}$ whenever $i\neq j$}.
\end{equation}
The symbol $L:={\rm diag}(L^1,\ldots,L^d)$ in \eqref{monotone-system} denotes a diagonal matrix  of second-order elliptic operators being either in divergence form
\begin{align}\label{divergence}
L^i(x)u_i:=-\nabla\cdot(A^i(x)\nabla u_i)+q^i(x)\cdot\nabla u_i,
~~i=1,\ldots,d,
\end{align}
or in nondivergence form
\begin{align}\label{nondivergence}
L^i(x)u_i:=-\tr(A^i(x)D^2u_i)+
q^i(x)\cdot\nabla u_i,~~i=1,\ldots,d,
\end{align}
where $\nabla$, $\nabla\cdot$, and $D^2$ denote the gradient, divergence, and Hessian with respect to $x\in\mathbb{R}^N$, respectively.
Such types of parabolic systems are also referred to as
\emph{(quasi-)monotone} or \emph{cooperative} reaction-diffusion systems.
In this regard, we refer to \cite{murray,volpert2014elliptic} for many aspects of modelling biomedical and ecological applications.
An extensive review of earlier mathematical theories is provided in the monograph by Smith \cite[Chap.7]{Smith1995monotone}.

As is well known, the large time behaviour of the homogeneous Fisher-KPP equation with compactly supported initial data is governed by travelling wave solutions \cite{fisher1937thewave,kolmogorov1937study}.
Such solutions are a special class of entire solutions to nonlinear diffusion equations.
These two pioneering papers have led to extensive research on the analysis of front propagation phenomena in reaction-diffusion equations and their systems.
As far as travelling waves for systems are concerned, we refer to Volpert et al. \cite{3V-book,volpert2014elliptic} for a comprehensive introduction and numerous self-contained treatments, and to Ogiwara and Matano \cite{ogiwara1998monotonicity,ogiwara1999stability} for some qualitative analyses.
Although a great deal of literature is devoted to homogeneous equations,
the notion of travelling waves has been extended to non-autonomous heterogeneous equations over the past two decades,
see Berestycki and Hamel \cite{berestycki2007generalized} for a systematic survey.
In this work, we shall study travelling waves for the monotone system \eqref{monotone-system} with spatially periodic heterogeneities in arbitrary dimensions.

\subsection{Setting of the problem}\label{setting-assup}
Before formulating the problem,
let us fix some notations to prevent any future confusion.
For any two vectors $u$, $v\in\mathbb{R}^d$,
we write $u\leq v$ (resp. $u\ll v$) to mean $u_i\leq v_i$ (resp. $u_i<v_i$)
for all $i\in\{1,\ldots,d\}$,
while $u<v$ if $u\leq v$ and $u_i<v_i$ for some $i$.
The maximum between the two is understood componentwise, i.e.
$\max(u,v)=(\max(u_1,v_1),\ldots,\max(u_d,v_d))^T$.
The minimum is defined similarly.
With this definition, the maximum between two vectors coincides with the least upper bound, and the minimum is the greatest lower bound.

We
denote the maximum (or uniform) norm of a vector $u\in\mathbb{R}^d$ by
$$
|u|_\infty:=\max_{1\leq i\leq d}(|u_i|)=\lim_{p\to\infty}|u|_p
$$
where $|u|_p:=(\sum_{i=1}^{d}|u_i|^p)^{1/p}$ with $p\geq1$ is the $p$-norm and by convention $|\cdot|:=|\cdot|_2$ denotes the usual Euclidean norm.
The function spaces $C^{1+{\alpha}/{2},2+\alpha}(\mathbb{R}\times\mathbb{R}^N)^d$ with $\alpha\geq0$ and $L^{p}(\mathbb{R}\times\mathbb{R}^N)^d$ with $p\geq1$
are equipped  with the following norms:
$$
\|w\|_{C^{1+{\alpha}/{2},2+\alpha}(\mathbb{R}\times\mathbb{R}^N)^d}
=\max_{1\leq i\leq d}\|w_i\|_{C^{1+{\alpha}/{2},2+\alpha}
(\mathbb{R}\times\mathbb{R}^N)},~~
\|w\|_{L^{p}(\mathbb{R}\times\mathbb{R}^N)^d}=
\big|(\|w_i\|_{L^{p}(\mathbb{R}\times\mathbb{R}^N)})_{1\leq i\leq d}\big|_p.
$$
Let $BUC(\mathbb{R}\times\mathbb{R}^N)^d$ be the Banach space of all bounded and uniformly continuous functions from $\mathbb{R}\times\mathbb{R}^N$ to $\mathbb{R}^d$ endowed with the usual supremum norm
$$
\|u\|_\infty=\max_{1\leq i\leq d}\sup_{(t,x)\in\mathbb{R}\times\mathbb{R}^N}|u_i(t,x)|.
$$
With a slight abuse of notation, let $\bm{1}:=(1,\ldots,1)^T$, $\bm{0}:=(0,\ldots,0)^T\in\mathbb{R}^d$ for future convenience.
A vector $u\in\mathbb{R}^d$ is said to be nonnegative, positive, and strictly positive, if it satisfies $u\geq\bm{0}$, $u>\bm{0}$, and $u\gg\bm{0}$, respectively.

Let us now elaborate on the general hypotheses.
Denote by $\mathbb{T}^{N}=(\mathbb{R}/\mathbb{Z})^N$ the $N$-dimensional unit torus.
We assume that the coefficients involved in $L$ given by \eqref{divergence} or \eqref{nondivergence} satisfy:
\begin{enumerate}[leftmargin=*, label=(C)]
  \item\label{item:coefficient} The matrix fields $A^i:\mathbb{T}^{N}\to\mathcal{S}_{N}(\mathbb{R})$ (the space of real symmetric $N\times N$ matrices) with $A^i(x)=(a^i_{\ell m}(x))_{1\leq \ell,m \leq N}$ are
      of class $C^{1,\alpha}$ for some $\alpha>0$,
      and uniformly elliptic in the sense that there exist positive constants $\underline{\gamma}\leq\overline{\gamma}$ such that %without loss of generality,
     \begin{equation}\label{elliptic-condition}
      \forall(x,\xi)\in\mathbb{T}^{N}\times\mathbb{R}^{N},~~
      \underline{\gamma}|\xi|^2\leq\xi^{T}A^i(x)\xi\leq\overline{\gamma}|\xi|^2~~
       \text{for all $i\in\{1,\ldots,d\}$},
     \end{equation}
     and the vector fields $q^i:\mathbb{T}^N\to\mathbb{R}^N$ with $q^i(x)=(q^i_{\ell}(x))_{1\leq \ell\leq N}$ are of class $C^\alpha$.
\end{enumerate}
To formulate the conditions for $f(x,u)$, we first recall some concepts from \cite{sweers1992strong,birindelli1999existence}.
\begin{definition}\label{cooperative-irreducible}\upshape
A matrix-valued function $H(z)=\left(h_{ij}(z)\right)_{1\leq i,j\leq d}$ from an arbitrary domain $\Omega$ to $\mathcal{M}_d(\mathbb{R})$ (the space of real $d\times d$ matrices) is said to be:
\begin{itemize}
  \item \emph{cooperative} if $h_{ij}(z)\geq0$ whenever $i\neq j$ for all $i, j\in\{1,\ldots,d\}$ and $z\in\Omega$;
  \item \emph{fully coupled} if for any nontrivial partitions $I, J\subset\{1,\ldots,d\}$, that is, $I\cap J=\varnothing$ and $I\cup J=\{1,\ldots,d\}$, there exist $i_0\in I$ and $j_0\in J$ such that $h_{i_0j_0}(z)\not\equiv0$.
\end{itemize}
\end{definition}
\noindent
Equivalently, the matrix field $H$ with bounded continuous entries is cooperative if and only if
  $$
  \underline{H}:=\Big(\inf_{z\in\Omega}h_{ij}(z)\Big)_{1\leq i, j \leq d}
  \text{ is essentially nonnegative (or called quasi-positive)},
  $$
while it is fully coupled if and only if
$$
  \overline{H}:=\left(\norm{h_{ij}}_\infty\right)_{1\leq i, j \leq d}
  \text{ is irreducible}.
$$
\noindent
We are now in a position to state the set of hypotheses about the nonlinearity.
\begin{enumerate}[label=(N\arabic*)]
  \item\label{item:smoothness} The nonlinearity $f=(f_1,\ldots,f_d)^T: \mathbb{T}^N\times\mathbb{R}^d\to\mathbb{R}^d$
      is continuous and
      of class $C^\alpha$ in $x$, locally uniformly for $u\in\mathbb{R}^d$.
      The partial derivatives $\partial f_i/\partial{u_j}$ exist and are continuous in $\mathbb{T}^N\times\mathbb{R}^d$ for all $i, j\in\{1,\ldots,d\}$.
      Moreover, $f(x,\bm{0})\equiv\bm{0}$ and
      $f$ is of class $C^{1,\beta}(\mathbb{T}^N\times B_\infty(\bm{0},\sigma))$
      for some $\beta\in(0,1]$, where $B_\infty(\bm{0},\sigma):=
      \{u\in\mathbb{R}^d\mid|u|_\infty\leq\sigma\}$ with a small $\sigma>0$.
  \item\label{item:monotonicity} The Jacobian matrix field of $f(x,u)$ with respect to $u$, denoted as $D_uf(x,u)$, is cooperative for all  $(x,u)\in\mathbb{T}^N\times\Sigma$ with some $\Sigma\subseteq\mathbb{R}^d$ in the sense of Definition \ref{cooperative-irreducible}.
  \item\label{item:upper-bound} There exists a constant $\widehat{\eta}>0$ such that
          $f(x,\widehat{\eta}\bm{1})\leq\bm{0}$ for all $x\in\mathbb{T}^N$.
  \item\label{item:irreducibility} The matrix field $D_uf(\cdot,u)$ $:\mathbb{T}^N\to \mathcal{M}_d(\mathbb{R})$ is fully coupled for every $u\in\mathbb{R}^d$ in the sense of Definition \ref{cooperative-irreducible}.
\end{enumerate}

Assumption
\ref{item:monotonicity} is equivalent to condition \eqref{quasimonotone-condition}.
It is also possible to weaken \ref{item:monotonicity} such that $D_uf(\cdot,u)$ is cooperative only for $u\in B_\infty^+(\bm{0},\widehat{\eta}):=\{u\in[0,\infty)^d\mid|u|_\infty\leq\widehat{\eta}\}$.
We mention that a weak condition of \emph{local monotonicity} was formulated by Volpert \cite{3V-book,volpert2014elliptic}.
Under \ref{item:monotonicity}, a sufficient condition for \ref{item:upper-bound} to hold is that for each $i\in\{1,\ldots,d\}$,
$$
  \forall u>\bm{0},~~
  \lim_{l\to+\infty}\frac{f_i(x,lu_1,\ldots,lu_d)}{l|u|_\infty}<0~
  \text{ uniformly for }x\in\mathbb{T}^N.
$$
Assumption \ref{item:irreducibility} means that the system \eqref{monotone-system} links all the unknowns $u_i$ in a strong sense, i.e.
it cannot be separated into several subsystems independent of each other.
Unless otherwise specified, the assumptions of \ref{item:coefficient} and \ref{item:smoothness}-\ref{item:irreducibility} are made throughout the paper.

Let us now introduce the object we want to study.
In the periodic framework, the notion of travelling periodic/pulsating waves was first proposed by Shigesada, Kawasaki and Teramoto \cite{shigesada1986traveling} for a one-dimensional equation of Fisher type, and further generalized by Berestycki and Hamel \cite{berestycki2002front} in a quite general setting,
see also Xin \cite{xin2000front}, Hudson and Zinner \cite{hudson1995existence}, Ogiwara and Matano \cite{ogiwara1998monotonicity,ogiwara1999stability}, Heinze \cite{heinze2001wave} for earlier works.
As far as the monotone system \eqref{monotone-system} is concerned,
such solutions are defined as follows.
\begin{definition}\label{pulsating-wave-def}\upshape
Let $e\in\mathbb{S}^{N-1}:=\{\zeta\in\mathbb{R}^N\mid|\zeta|=1\}$ be given.
A \emph{pulsating travelling wave} propagating in the direction $e$ with the average (or effective) speed $c\neq0$ is a classical bounded entire solution $u\equiv u(t,x)$ to \eqref{monotone-system} which is strictly positive and satisfies
\begin{align}
\forall k\in\mathbb{Z}^{N},~\forall(t,x)\in\mathbb{R}\times\mathbb{R}^N,~~~
&u\left(t+\frac{k\cdot e}{c},x\right)=u(t,x-k),\label{pulsating-condition}\\
\liminf_{x\cdot e\to-\infty}u(t,x)\gg\bm{0},
&~~~\lim_{x\cdot e\to+\infty}u(t,x)=\bm{0},\label{limiting-condition}
\end{align}
where the limits are understood to hold locally uniformly for $t\in\mathbb{R}$ and uniformly with respect to the vectors of
$e^\perp:=\{y\in\mathbb{R}^{N}\mid y\cdot e=0\}$.
\end{definition}
Note that,
since we are interested only in invasion fronts propagating into the unstable state $\bm{0}$,
a persistence-like property is introduced to describe one of the limiting states of pulsating waves,
rather than imposing an \emph{a priori} given steady state for \eqref{monotone-system}.
The goal of this paper is threefold.
When the trivial steady state of \eqref{monotone-system} is unstable, we first establish some qualitative estimates for pulsating waves.
In particular, we prove that any pulsating wave given by Definition \ref{pulsating-wave-def} is strictly monotonic in time if the nonlinearity is subhomogeneous.
The second part is devoted to the existence of pulsating waves.
In the third part, we prove the nonexistence of pulsating waves when the trivial steady state of \eqref{monotone-system} is stable.
Throughout the last two parts, we also discuss the asymptotic persistence (known as the \emph{hair-trigger effect}) and extinction for the Cauchy problem associated with \eqref{monotone-system}.

Let us add some other related works on the pulsating wave-like solutions for spatially periodic reaction-diffusion equations and their systems.
As far as scalar equations are concerned, we refer to \cite{hamel2024propagation,Bages2012,berestycki2005analysis-2,hamel2011uniqueness,ducrot2016multi,
ding2017bistable,ducrot2012terrace,giletti2019pulsating,ding2021admissible}.
Recent advances in non-monotone systems can also be found in \cite{SI-system,griette2021propagation,griette2025front} and the references therein.
From a different perspective, Weinberger \cite{weinberger2002spreading} developed a general theory of spreading properties for a discrete-time recursion governed by an order preserving periodic operator, see \cite{Weinberger1982long} for his earliest work in this direction.
This theory introduces a profound understanding of the relationship between travelling waves and the asymptotic spreading of solutions for a wide class of reaction-diffusion models.
The results in \cite{weinberger2002spreading} have further been extended in the past few years,
see Liang and Zhao \cite{liang2007asymptotic,liang2010spreading} for monostable case.
We emphasize that applying the abstract results of \cite{liang2007asymptotic,liang2010spreading,du2022propagation} to monotone systems
does require stronger assumptions on \eqref{monotone-system} than those made here.
This is also one of our primary motivations for the present work.
A detailed discussion will follow the statement of our main results.
In addition, Fang and Zhao \cite{fang2015bistable} conducted a deep study of travelling waves for bistable monotone semiflows, which includes results on one-dimensional bistable pulsating waves for monotone systems.
Giletti and Rossi \cite{giletti2019pulsating} extended the results of \cite{fang2015bistable,ducrot2012terrace}
to multidimensional scalar equations of both bistable and multistable types.
The aforementioned works contributed with different methodologies for studying travelling waves of monotone systems and generalized several results of Volpert et al. \cite{volpert2014elliptic,3V-book}.
It should also be noted that the presence of multiple steady states in such systems gives rise to numerous other types of travelling waves or pulses  \cite{volpert2014elliptic,3V-book}, including bistable ones,
which are not addressed here within the periodic framework adopted in this work.

Let us also mention that several works have focused on pulsating waves in space-time periodic media.
We refer to Nadin \cite{nadin2009traveling} for scalar equations,
and to Fang et al. \cite{fang2017} for monostable monotone semiflows with one spatial dimension, as well as Du et al. \cite{du2022propagation} with arbitrary space dimensions.
Furthermore, the notion of generalized travelling waves in random media was introduced by Matano \cite{Matano2003} (see also Shen \cite{shen2004traveling}),
and transition fronts under a fairly general setting, by Berestycki and Hamel \cite{berestycki2007generalized}.
The literature on this topic is now extensive and a complete review is beyond the scope of this work.

To conclude this section, let us point out that our motivation for this work is to understand the kind of additional conditions to be imposed on the system \eqref{monotone-system}
which allow front propagation into the unstable trivial equilibrium.

\subsection{Statement of the results}\label{statement-main-results}
To state our main results, several notions related to linear problems
are required.
For this purpose,  for each $e\in\mathbb{S}^{N-1}$ and each $\lambda\in\mathbb{R}$, we first introduce the operators $L^i_{\lambda e} $ as follows:
\begin{equation}\label{modified-operator}
 \begin{split}
&L^i_{\lambda e}\psi:=e^{\lambda e\cdot x}L^i\left(e^{-\lambda e\cdot x}\psi\right)\\
&\qquad~\,=-\nabla\cdot(A^i\nabla\psi)+2\lambda eA^i\nabla\psi
+q^i\cdot\nabla\psi-\left[\lambda^2eA^ie-\lambda\nabla\cdot(A^ie)+\lambda
q^i\cdot e\right]\psi
\end{split}
\end{equation}
or
$$
L^i_{\lambda e}\psi:=-\tr(A^iD^2\psi)+2\lambda eA^i\nabla\psi
+q^i\cdot\nabla\psi-(\lambda^2eA^ie+\lambda q^i\cdot e)\psi,
$$
where we recall that the operators $L^i$ are given by \eqref{divergence} or \eqref{nondivergence}.

Set $L_{\lambda e}:={\rm diag}(L^1_{\lambda e},\ldots,L^d_{\lambda e})$ and
let $k(\lambda,e)$ denote the \emph{periodic principal eigenvalue} of the system of linear elliptic operators
   $$
   \phi\mapsto L_{\lambda e}\phi-D_uf(x,\bm{0})\phi
   $$
posed for $\phi\in C^2(\mathbb{T}^N)^d$ with periodicity conditions.
In other words, $k(\lambda,e)$ is the unique real number such that there exists a function $\phi_{\lambda e}\in C^2(\mathbb{T}^N)^d$ satisfying
\begin{equation}\label{periodic-PE}
 \begin{dcases}
  L_{\lambda e}\phi_{\lambda e}-D_uf(x,\bm{0})\phi_{\lambda e}=
  k(\lambda,e)\phi_{\lambda e}~\text{ in }\mathbb{T}^N,\\
  \phi_{\lambda e}\gg\bm{0}~\text{ in }\mathbb{T}^N,~~~
  \|\phi_{\lambda e}\|_\infty=1.
  \end{dcases}
\end{equation}
When $\lambda=0$, we set
\begin{equation}\label{periodic-threshold}
\lambda^{p}_{1}:=k(0,e).
\end{equation}
Observe from \eqref{modified-operator} that $\lambda^p_{1}$ does not depend on $e$.
In what follows, we say that $\bm{0}$ is a linearly unstable solution of \eqref{monotone-system} if $\lambda^p_{1}<0$, while it is stable if $\lambda^p_{1}\geq0$.

We now introduce the \emph{Dirichlet principal eigenvalue} for the operator $L_{\lambda e}-D_uf(x,\bm{0})$ with $\lambda=0$.
Let $B_R$ be the open ball of $\mathbb{R}^N$ with centre $\bm{0}$ and radius $R$.
For any $R>0$, we consider the unique $\lambda^R_1\in\mathbb{R}$ such that there exists
a function $\varphi^R\in C^2(B_R)^d\cap C^1(\overline{B_R})^d$ which satisfies
\begin{equation}\label{Dirichlet-PE}
 \begin{dcases}
  L\varphi^R-D_uf(x,\bm{0})\varphi^R
  =\lambda^R_1\varphi^R&\text{in }B_R,\\
  \varphi^R=\bm{0}&\text{on }\partial B_R,\\
  \varphi^R\gg\bm{0}~\text{ in }B_R,~~~
  \|\varphi^R\|_\infty=1.
  \end{dcases}
\end{equation}
The existence and uniqueness of the principal eigenpairs follow from the Krein-Rutman theory and the fact that under assumptions \ref{item:monotonicity} and \ref{item:irreducibility}, $D_uf(x,\bm{0})$ is cooperative and fully coupled in the sense of Definition \ref{cooperative-irreducible},
see \cite{sweers1992strong,lam2016asymptotic,birindelli1999existence}.
With the notion of the \emph{generalized principal eigenvalue} for scalar elliptic equations (see \cite{berestycki2007Liouville,berestycki2005analysis-1} for instance),
we also introduce such a quantity for the operator $L-D_uf(x,\bm{0})$ in $\mathbb{R}^N$.
This definition here reads
\begin{equation}\label{generalized-PE}
\lambda_1:=\sup\left\{\lambda\in\mathbb{R}\mid\exists\,\varphi\in
          C^2(\mathbb{R}^N)^d,~\phi\gg\bm{0}\text{ and }
 \big(L-D_uf(x,\bm{0})\big)\varphi\geq\lambda\varphi
 \text{ in }\mathbb{R}^N\right\}.
\end{equation}

Assume now $\lambda^p_{1}<0$.
Then the following quantity turns out to be well defined:
\begin{equation}\label{formula-minimal-speed-inf}
  c^*(e):=\inf_{\lambda>0}\frac{-k(\lambda,e)}{\lambda}.
\end{equation}
Moreover, the infimum in \eqref{formula-minimal-speed-inf} can also be reached on $(0,\infty)$ and there holds
\begin{equation}\label{relations-among-3PEs}
 \lambda^\infty_{1}:=\lim_{R\to+\infty}\lambda^R_1\in\mathbb{R},~~~
 \lambda^p_{1}\leq\lambda^\infty_{1}=\lambda_{1}=
 \max_{\lambda e\in\mathbb{R}^{N}}k(\lambda,e).
\end{equation}
Further properties for the quantities introduced above will be stated in Section \ref{sect:linear-problem}.

With these quantities defined, our main results are stated as the following three theorems.
\begin{theorem}\label{main-result-qualitative-property}
Assume $\lambda^p_{1}<0$. Let $(c,u)$ be a pulsating travelling wave solution to   \eqref{monotone-system}.
Then the wave speed $c$ satisfies
$$
c\geq c^*(e).
$$
If $f$ is further assumed to be subhomogeneous, i.e.
  \begin{equation}\label{subhomogeneity}
  \forall\theta\in(0,1),~
  \forall(x,u)\in\mathbb{T}^N\times[0,\infty)^d,~~
  \theta f(x,u)\leq f(x,\theta u),~~
  \end{equation}
then the function $u=u(t,x)$ is componentwise increasing in $t$ if $c>0$ and decreasing if $c<0$.
Moreover, one has $c\partial_t u(t,x)\gg\bm{0}$ for all $(t,x)\in\mathbb{R}\times\mathbb{R}^N$.
\end{theorem}

\begin{theorem}\label{main-result-existence}
Assume $\lambda_{1}<0$.
Suppose the nonlinearity $f$ satisfies
  \begin{equation}\label{sublinearity}
  \forall(x,u)\in\mathbb{T}^N\times[0,\infty)^d,~~~
  f(x,u)\leq D_uf(x,\bm{0})u.
  \end{equation}
Then, for every $e\in\mathbb{S}^{N-1}$ and for all $c\geq c^*(e)$,
the system \eqref{monotone-system} admits a pulsating travelling wave in direction $e$ with speed $c$.
\end{theorem}
\begin{theorem}\label{main-result-nonexistence}
Assume $\lambda^p_{1}\geq0$.
Let condition \eqref{sublinearity} be fulfilled, and
assume further that the inequality \eqref{sublinearity} is strict whenever $u\gg\bm{0}$, i.e.
\begin{equation}\label{strict-sublinear-condition}
  \forall(x,u)\in\mathbb{T}^N\times(0,\infty)^d,~~
  f(x,u)<D_uf(x,\bm{0})u.
\end{equation}
Then there exists no pulsating travelling wave of \eqref{monotone-system} in any given direction.
\end{theorem}
\begin{remark}\label{stronger-conditions}
The set of our hypotheses \ref{item:smoothness}-\ref{item:irreducibility} together with \eqref{sublinearity} is actually an analogue of (F1)-(F5) formulated by Barles et al. \cite{barles1990wave}, in which the authors investigated
front propagation in homogeneous media by
developing a viscosity solution approach for Hamilton-Jacobi equations.
In fact, under assumptions \ref{item:smoothness} and \eqref{subhomogeneity}, the sublinear condition \eqref{sublinearity} holds.
Furthermore, if $f$ is strictly subhomogeneous, i.e.
\begin{equation}\label{strict-subhomogeneity}
   \forall\theta\in(0,1),~
  \forall(x,u)\in\mathbb{T}^N\times(0,\infty)^d,~~
  \theta f(x,u)<f(x,\theta u),~~
\end{equation}
then the strictly sublinear condition \eqref{strict-sublinear-condition} is also fulfilled.
\end{remark}
\begin{remark}\label{two-principal-eigenvalues}
  If $\lambda_{1}<0$, it follows from \eqref{formula-minimal-speed-inf} and \eqref{relations-among-3PEs} that $c^*(e)>0$ for every $e\in\mathbb{S}^{N-1}$.
  Furthermore, if $\lambda^p_{1}=\lambda_{1}$, then the first two theorems show that the system \eqref{monotone-system} admits a pulsating wave if and only if $c\geq c^*(e)$.
 Nevertheless, these two quantities are not necessarily equal.
Depending on the assumptions on the linearized operator, a strict inequality $\lambda^p_{1}<\lambda_{1}$ is possible,
see Remark \ref{sufficient-condition} below.
\end{remark}
\begin{remark}\label{null-entire-solution}
It is worth pointing out that the proof of the linear determinacy for the minimal wave speed does not require the limiting condition of pulsating waves at $x\cdot e=-\infty$ in \eqref{limiting-condition}.
Moreover, the existence of pulsating waves is guaranteed under assumption \eqref{sublinearity} even if the monotonicity does not hold in general.

Theorem \ref{main-result-nonexistence} asserts that the linearly stable equilibrium precludes invasion by any uniformly positive state.
Indeed, under assumptions of Theorem \ref{main-result-nonexistence},
Proposition \ref{extinction} in Sect.\,\ref{sect:nonexistence} implies that there is no nonnegative bounded entire solution to \eqref{monotone-system} besides the trivial one.
It should also be stressed that condition \eqref{strict-sublinear-condition} is essential to the validity of Theorem \ref{main-result-nonexistence}.
Recall that in the case of scalar reaction-diffusion equations, pulsating waves are known to exist for both ignition-type nonlinearities
(satisfying $\lambda^p_{1}=0$, see \cite{berestycki2002front})
and bistable ones (satisfying $\lambda^p_{1}>0$, see \cite{ducrot2016multi,ding2017bistable}).
\end{remark}

Before proceeding, let us further comment on the above results and recall some earlier works in the literature.
Theorem \ref{main-result-qualitative-property} provides a sufficient condition for pulsating waves to be monotonic in time.
As a consequence, if \eqref{subhomogeneity} holds, then any pulsating wave $(u,c)$ converges, as $t\to+\infty$ if $c>0$ or as $t\to-\infty$ if $c<0$,
to a bounded strictly positive steady state for \eqref{monotone-system},
which can also be verified to be $\mathbb{Z}^N$-periodic because of \eqref{pulsating-condition}.
This observation naturally motivates us to seek pulsating waves that connect the unstable state $\bm{0}$ and a stable steady state $p$ for \eqref{monotone-system}.
To this end, if there exists a minimal periodic steady state $p$ in the class of all strictly positive bounded steady states for \eqref{monotone-system},
and that it is asymptotically stable for the semiflow generated by \eqref{monotone-system} in the set $\{\varphi\in C(\mathbb{T}^N)^d\mid \bm{0}<\varphi\leq p\}$, then the limiting condition \eqref{limiting-condition} could be refined to
\begin{equation}\label{two-limits}
  \lim_{x\cdot e\to-\infty}|u(t,x)-p(x)|_\infty=0,~~~\lim_{x\cdot e\to+\infty}u(t,x)=\bm{0},
\end{equation}
locally uniformly for $t\in\mathbb{R}$ and uniformly in $e^\perp$.
Under such a monostable setting,
the existence of pulsating waves satisfying \eqref{two-limits} is a direct consequence of Weinberger's result in \cite{weinberger2002spreading} for one component case and Liang-Zhao's result in \cite{liang2010spreading} for one-dimensional case.
Nevertheless, besides contributing additional qualitative properties,
the present paper provides a feasible and alternative approach to establishing the existence of pulsating waves for monostable monotone systems.
Furthermore, our work extends some of the results on planar travelling waves in homogeneous media from Volpert et al. \cite{3V-book} and Li et al. \cite{li2005spreading},
as well as those on curved travelling waves in straight cylindrical domains from Volpert \cite{volpert2014elliptic}.

However, it is technically challenging to find some sharp sufficient conditions under which the system \eqref{monotone-system} goes under the above-mentioned framework.
A special case worth noting is that by strengthening \eqref{subhomogeneity} to \eqref{strict-subhomogeneity},
one can apply the theory of strongly monotone (guaranteed here by assumption \ref{item:irreducibility}) and strictly subhomgeneous maps, see \cite[Sect\,2.3]{Zhaobook} for instance, to yield a threshold-type result for the semiflow generated by \eqref{monotone-system} in terms of the stability properties of the trivial equilibrium,
see \cite{huang2022propagation,wang2015pulsating,deng2021propagation} for some examples of such arguments.
This also suggests us to establish the \emph{Liouville-type} results for monotone system \eqref{monotone-system} as preparation for studying invasion dynamics.
Observe that under assumption \eqref{sublinearity}, the system \eqref{monotone-system} shares structural similarities with scalar Fisher-KPP equations.
According to the results of Berestycki et al. \cite{berestycki2005analysis-1,berestycki2007Liouville},
the strong-KPP condition is almost sharp for the uniqueness and stability of bounded positive steady states for scalar semilinear parabolic equations in the whole space.
In fact, under assumptions $\lambda^p_{1}<0$ and \eqref{sublinearity},
one can readily prove that there exists at least a strictly positive periodic steady states for the monotone system \eqref{monotone-system}.
But the uniqueness (or even the minimality) and the stability within the class of all nonnegative bounded steady states, namely without assuming a priori that any steady state for \eqref{monotone-system} is uniformly positive and periodic,
is far from straightforward.

On the other hand,
in order to establish monotonicity results, we assume in Theorem \ref{main-result-qualitative-property} that $f$ is subhomogeneous in the sense of \eqref{subhomogeneity}.
This condition plays a role similar to that of the strong-KPP assumption made by Berestycki et al. \cite{berestycki2005analysis-2} for scalar equations.
But a weaker assumption was formulated by Hamel \cite{hamel2008qualitative},
which also guarantees the monotonicity of monostable pulsating waves.
A generalization to monotone systems becomes then natural yet more complicated.
To sum up,
it remains an interesting and open problems to derive sharp conditions that yield Liouville-type results for the stationary system of \eqref{monotone-system} in the whole space.
For these reasons, we will not here formulate any assertion leading to \eqref{two-limits} but adopt the weaker condition \eqref{limiting-condition}.

Let us finally mention that Du et al. \cite{du2022on,du2022propagation} extended the results of \cite{yu2017propagation,fang2017} concerning a class of monostable monotone systems to a multi-dimensional periodic medium in which the system admits equilibria lying on the boundary of the positive cone in a Banach lattice.
Typically, the nonlinearity considered in \cite{du2022on} takes the form
\begin{equation}\label{du-framework}
\begin{split}
  &f_1(x,u_1,\ldots,u_d)=u_1h_1(x,u_1,\ldots,u_d),\\
  &f_{i}(x,u_1,\ldots,u_d)=\sum_{j=1}^{i-1}c_{ij}(x)u_j+
  u_ih_i(x,u_1,\ldots,u_d),~~~2\leq i\leq d,
\end{split}
\end{equation}
under certain conditions imposed on the functions $h_i$ and $c_{ij}$.
A classical model that fits their setting is a cooperative system derived from the Lotka-Volterra competitive system after performing a well known change of variables.
However,
there are many examples of $f$ which could not be unified as \eqref{du-framework}, see Sect.\,\ref{examples} below for details.
In this regard, our setting thus covers a large class of other types of nonlinearities not previously treated, thereby contributing new results.
\subsection{Elements of the proof}\label{outline-proof}
The proof of Theorem \ref{main-result-qualitative-property} adapts the arguments of Berestycki et al. \cite{berestycki2005analysis-2} and Hamel \cite{hamel2008qualitative}
who investigated scalar equations of strong-KPP and monostable types.
The proof of Theorem \ref{main-result-existence} uses a general methodology developed by the first two authors of the present paper in \cite{deng2023existence,SI-system}.
Specifically, the first step constructs pulsating waves in each direction of unit vectors with rational coordinates (denoted as $\mathbb{Q}^N\cap\mathbb{S}^{N-1}$),
and the second step performs an approximation procedure for a general direction in $\mathbb{S}^{N-1}$.
For the latter purpose and to prove Theorem \ref{main-result-nonexistence}, we adapt some arguments developed by Berestycki et al. for scalar equations \cite{berestycki2005analysis-1,berestycki2007Liouville}.
Observe that the profile of one-dimensional pulsating waves at time $T=t+{1}/{c}$
is an exact copy of the profile at $T=t$, shifted in space by one unit of length.
This observation provides a key motivation for formulating the pulsating wave as a fixed-point problem.
However, in the multi-dimensional setting, the problem
(that is, to verify \eqref{pulsating-condition}) becomes much more involved.

We now highlight the main new ingredients in the present paper for constructing a pulsating wave propagating along any rational direction, which differ from those in \cite{deng2023existence,SI-system}.
Firstly, we construct sub- and supersolutions for all admissible speeds, especially for the critical case, inspired again by \cite{hamel2008qualitative}.
Secondly, by slightly modifying the coordinate transformation introduced in \cite{deng2023existence},
pulsating waves in a rational direction,
observed from a new system of coordinates,
become (roughly speaking) pulsating waves propagating inside a cylinder with a periodic cross-section.
Based on that and inspired by Ogiwara and Matano \cite[Sect.7]{ogiwara1999stability},
we reformulate a pulsating wave with speed $c>0$
as a fixed point of the solution map at $t=\tau_1/c$ composed with a shift by $\tau_1$ along the cylinder axis, where $\tau_1>0$ is a parameter determined by the period of the medium and the rational direction of propagation.
Lastly,
we address this issue by considering an initial/boundary-value problem posed in a semi-infinite periodic cylinder with a moving (at velocity $c$) boundary
subject to the homogeneous Neumann data,
which is then solved by the Schauder fixed-point theorem (see Proposition \ref{rational-direction-existence}).
In particular, to verify the limiting conditions of the solution, a novel lower barrier is also constructed (see \eqref{bound-rational-c}-\eqref{lower-barrier-finite-c*}).
This approach was developed by Griette and Matano \cite{griette2021propagation} who investigated pulsating one-dimensional waves for other models.

Let us provide further remarks on our method.
For scalar equations, a common and powerful approach to constructing pulsating waves is to solve a semilinear degenerate elliptic equation via a regularization procedure  \cite{berestycki2002front,berestycki2005analysis-2,ducrot2016multi}.
This approach requires refined uniform estimates of the solution with respect to a vanishing viscosity parameter.
However, it would certainly require more PDE techniques and possibly make additional assumptions on the coefficients in \eqref{nondivergence} to extend such estimates, especially gradient estimates, to general systems as \eqref{monotone-system}.
In contrast, a key contribution of our method,
as already demonstrated for the scalar Fisher-KPP equation in \cite{deng2023existence},
lies in overcoming the difficulty caused by elliptic degeneracy.
The present work generalizes this method to the monotone system \eqref{monotone-system} and provides a new approach to constructing pulsating waves in rational directions.
In \cite{deng2023existence}, the construction of pulsating waves in a rational direction  relies on the approximation of the equation in bounded domains.
To do so, we first solved a space-time periodic boundary value problem and then performed an iteration procedure.
Here we devise a subtle fixed-point problem as outlined above to construct pulsating waves in rational directions,
which actually leads to considerable simplifications compared to \cite{deng2023existence,Bages2012}.
For a non-rational direction of propagation, we recover existence by passing the limit as in \cite{deng2023existence,SI-system}.
\subsection{Examples}\label{examples}
Monotone reaction-diffusion systems arise in diverse applications \cite{volpert2014elliptic}.
As a topic in theoretical ecology, such systems describe the dynamics of mutualistic symbiosis between two or more species \cite{murray}.
Here we briefly present several examples to illustrate the applications and motivations of the present work.
We provide only the explicit form of the nonlinearities that fall within our settings,
while the corresponding modelling mechanism refers to the references cited.

The first example is a very classical model (known as a ``positive feedback loop" of $d$ species) for controlling protein synthesis in the cell
\cite[Sect.\,4.2 and Sect.\,7.7]{Smith1995monotone}
(see also \cite[Sect.\,7.2]{murray}) in which the nonlinearity is defined as
\begin{equation}\label{example1-f}
\begin{split}
  &f_1(u_1,\ldots,u_d)=g(u_d)-\alpha_1u_1,\\
  &f_{i}(u_1,\ldots,u_d)=u_{i-1}-\alpha_iu_i,~~2\leq i\leq d,
\end{split}
\end{equation}
with positive coefficients $\alpha_i$ and $d>2$.
The function $g:\mathbb{R}_+\to\mathbb{R}_+$ is assumed to be:
\begin{equation}\label{example-g}
\forall s>0,~~~~0<g(s)<M,~~g^\prime(s)>0~\text{ and }~g(0)=0.
\end{equation}
A typical example of $g$ from \cite{Smith1995monotone} is given by
$$
g(s)=\frac{s^p}{1+s^p}~~\text{  for some integer $p\geq1$}.
$$
With such a choice and for any $p>1$, the function
$f=(f_1,\ldots,f_d)$ defined in \eqref{example1-f} satisfies assumptions \ref{item:smoothness}-\ref{item:irreducibility} but verifies neither subhomogeneity \eqref{subhomogeneity} nor sublinearity \eqref{sublinearity}.
When $p=1$, then \eqref{subhomogeneity} is fulfilled.
Furthermore, a slight modification of the assumptions on $g$ satisfying \eqref{example-g} can yield further insights.
We consider two types for the function $g$ as follows:
\begin{itemize}
  \item For all $s>0$,~~~$0<g(s)\leq g^\prime(0)s$ (weak-KPP hypothesis),
  \item The map $s\mapsto\frac{g(s)}{s}$ is strictly decreasing in $(0,\infty)$ (strong-KPP hypothesis).
\end{itemize}
It is straightforward to verify that if $g$ is weak-KPP type, then the function $f$ defined in \eqref{example1-f} only satisfies the sublinear condition \eqref{sublinearity}.
However, the strictly subhomogeneous condition \eqref{strict-subhomogeneity} is verified provided $g$ is strong-KPP type.
In the latter case, condition \eqref{strict-sublinear-condition} is thus guaranteed from Remark \ref{stronger-conditions}.

A second example is an SIR-type model for rabies circulation between two populations \cite{deng2021propagation}, in which
the model assumes that susceptible individuals remain motionless to capture the disease dynamics.
After performing a change of variables, the authors turned such an epidemic model into a monotone system from which the nonlinearity is of the form
\begin{equation*}
\begin{split}
&f_1(x,u_1,u_2)=S_1^0(x)\left(1-e^{-(\beta_{11}(x)u_1+\beta_{12}(x)u_2)}\right)
-\delta_{1}(x)u_1,\\
&f_2(x,u_1,u_2)=S_2^0(x)\left(1-e^{-(\beta_{21}(x)u_1+\beta_{22}(x)u_2)}\right)
-\delta_{2}(x)u_2.
\end{split}
\end{equation*}
With all parameters being periodic in $x$ and positive everywhere,
the function $f=(f_1,f_2)$ satisfies not only assumptions \ref{item:smoothness}-\ref{item:irreducibility} but also the strictly (in fact, strongly) subhomogeneous condition \eqref{strict-subhomogeneity}.
Hence all the aforementioned conditions are fulfilled in this example.

The last example is drawn from an ecological model of population migration among distinct patches (such as cities) \cite[Sect.\,4.4]{Smith1995monotone} in which the nonlinearity is defined as
\begin{equation*}
\begin{split}
  f_1(u_1,\ldots,u_d)&=\epsilon(u_2-u_1)+r_1u_1\Big(1-\frac{u_1}{K_1}\Big),\\
  f_2(u_1,\ldots,u_d)&=\epsilon(u_1+u_3-2u_2)+r_2u_2\Big(1-\frac{u_2}{K_2}\Big),\\
  &\vdots\\
  f_d(u_1,\ldots,u_d)&=\epsilon(u_{d-1}-u_d)+r_du_d\Big(1-\frac{u_d}{K_d}\Big).
\end{split}
\end{equation*}
One can easily check that the function $f=(f_1,\ldots,f_d)$ satisfies assumptions \ref{item:smoothness}-\ref{item:irreducibility} and it is strongly subhomogeneous since the population growth obeys the Logistic law.

\vspace{0.2cm}
\noindent
\emph{Outline of the paper.}
The remainder of the paper is as follows:
Sect.\,\ref{sect:linear-problem} contains some preliminaries and the proof of which is provided in Appendix \ref{appendix:strict-concavity}.
In Sect.\,\ref{sect:qualitative-estimates}, we derive the lower bound for the wave speeds, its linear determinacy, and the monotonicity results.
Sect.\,\ref{sect:existence} and Sect.\,\ref{sect:nonexistence} prove the existence and nonexistence of pulsating waves, respectively.
Appendix \ref{appendix:Harnack-inequa} contains two Harnack-type estimates that are used several times in our proofs.
\section{Preliminaries}\label{sect:linear-problem}
In this section, we collect some important properties of the principal eigenpairs for a cooperative fully coupled system of elliptic operators with periodic coefficients.
The relationship among the principal eigenvalues in several senses is also analyzed.
\begin{proposition}\label{generalized-principal-eigenvalue-properties}
Let $H\in C^\alpha(\mathbb{T}^N,\mathcal{M}_{d}(\mathbb{R}))$ be a cooperative and fully coupled matrix field according to Definition \ref{cooperative-irreducible}.
Let $(\lambda^R_{1},\varphi^R)$ be the principal eigenpair of the problem
    $$
    \begin{dcases}
       (L-H)\varphi^R=\lambda^R_{1}\varphi^R&\text{in }B_R,\\
        \varphi^R=\bm{0}&\text{on }\partial B_R,\\
        \varphi^R\gg\bm{0}&\text{in }B_R,
    \end{dcases}
    $$
    where $L$ denotes a diagonal matrix of linear elliptic operators given by \eqref{nondivergence}.
Let
  \begin{equation}\label{generalized-principal-eigenvalue}
\lambda_1=\sup\left\{\lambda\in\mathbb{R}\mid\exists\,\psi\in
 C^2(\mathbb{R}^N)^d,~\psi\gg\bm{0}\text{ and }
 (L-H)\psi\geq\lambda\psi\text{ in }\mathbb{R}^N\right\}
\end{equation}
denote the generalized principal eigenvalue of the operator $L-H$ in $\mathbb{R}^N$.

Then $\lambda^R_1\searrow\lambda_1$ as $R\to+\infty$.
Moreover, there exists a generalized principal eigenfunction $\varphi\in C^2(\mathbb{R}^N)^d$ associated with $\lambda_1$, i.e. it satisfies $\varphi\gg\bm{0}$ and
$$
(L-H)\varphi=\lambda_1\varphi~\text{ in $\mathbb{R}^N$}.
$$
\end{proposition}
\begin{proposition}\label{periodic-principal-eigenvalue-properties}
Let $L_{\lambda e}$ be a diagonal matrix of linear elliptic operators given by \eqref{modified-operator}
and $H\in C^\alpha(\mathbb{T}^N,\mathcal{M}_{d}(\mathbb{R}))$ be cooperative and fully coupled in the sense of Definition \ref{cooperative-irreducible}.
Then the following statements hold.
\begin{enumerate}[label=\textup{(\roman*)}]
  \item\label{item:existence-uniqueness}
  For all $\lambda\in\mathbb{R}$ and for each $e\in\mathbb{S}^{N-1}$, there exists a unique $k=k(\lambda,e)\in\mathbb{R}$ (principal eigenvalue) and a function $\phi=\phi_{\lambda e}\in C^2(\mathbb{T}^N)^d$ with $\phi_{\lambda e}\gg\mathbf{0}$, unique up to positive scalar multiple,
  which solve the system of periodic elliptic equations
      \begin{equation}\label{nonsymmetric-periodic-PE}
      (L_{\lambda e}-H)\phi=k\phi~~\text{in }\mathbb{T}^N.
      \end{equation}
    \item\label{item:max-min} The principal eigenvalue $k(\lambda,e)$ can be characterized by
        \begin{equation}\label{max-min}
        k(\lambda,e)=\max_{\phi\gg\bm{0},\,\phi\in C^2(\mathbb{T}^N)^d}
        \min_{1\leq i\leq d}\inf_{x\in\mathbb{T}^N}
        \frac{\big(e^i(L_{\lambda e}-H)\phi\big)(x)}{\phi_i(x)},
        \end{equation}
        where ${e^i}\in\mathbb{S}^{d-1}$ denotes the row vector with only the $i$-th coordinate equal to $1$.
    \item\label{item:concavity} For any given $e\in\mathbb{S}^{N-1}$, the function $\lambda\mapsto k(\lambda,e)$ is analytic and strictly concave in $\mathbb{R}$.
        Furthermore, there exists $\beta\in\mathbb{R}$ such that
        \begin{equation}\label{upper-estimate-principal-eigenvalue}
          \forall\lambda\in\mathbb{R},~~
          k(\lambda,e)\leq-\underline{\gamma}\lambda^2+\beta\lambda+
          \max_{1\leq i\leq d}\|h_{ii}\|_\infty,
        \end{equation}
    where $\underline{\gamma}$ has been defined in \eqref{elliptic-condition}.
    \item\label{item:continuity} Let $\phi_{\lambda e}$ be the unique principal eigenfunction such that \begin{align}\label{normalization-principal-eigenfunction}
         \phi_{\lambda e}\gg\bm{0}~\text{ and }~\|\phi_{\lambda e}\|_\infty=1.
        \end{align}
        Then for all $\lambda\in\mathbb{R}$, the map $\mathbb{S}^{N-1}\ni e\mapsto k(\lambda,e)$ is continuous and the normalized principal eigenfunction $\phi_{\lambda e}$ also depends continuously on $\lambda e\in\mathbb{R}^{N}$ with respect to the uniform topology.
    \item\label{item:maximum} One has
    \begin{equation}\label{maximum-periodic-PE}
    \lambda_{1}=\max_{\lambda e\in\mathbb{R}^N}k(\lambda,e),
    \end{equation}
    where $\lambda_1$ is defined by \eqref{generalized-principal-eigenvalue}.
    Furthermore, for any given $e\in\mathbb{S}^{N-1}$, the maximum in \eqref{maximum-periodic-PE} is attained at a unique $\bar{\lambda}\in\mathbb{R}$  such that the function $\mathbb{R}^N\ni x\mapsto \phi_{\bar{\lambda}e}(x)e^{\bar{\lambda}e\cdot x}$ is a generalized principal eigenfunction associated with $\lambda_{1}$.
  \end{enumerate}
\end{proposition}
Before proceeding further, we mention a recent work by Girardin and Mazari \cite{girardin2023generalized} who made a deep study of principal eigenvalues for a linear cooperative system of space-time periodic parabolic operators.
In particular, their results can be applied to explain in various ways how spatial heterogeneity affects $k(\lambda,e)$.
The aforementioned results in the elliptic setting are analogues to their results.
For instance, Propositions 3.2-3.3 in \cite{girardin2023generalized} are more general than Proposition \ref{generalized-principal-eigenvalue-properties}.
But we include in Appendix \ref{appendix:strict-concavity} the proof of partial results of Proposition \ref{periodic-principal-eigenvalue-properties} for completeness of the present paper.
Most of the proofs are direct adaptations of the arguments developed by Berestycki et al. \cite{berestycki2008asymptotic,berestycki2002front,berestycki2005analysis-2} and
Nadin \cite{nadin2009principal} for scalar equations.

The next result is a direct consequence of Proposition \ref{periodic-principal-eigenvalue-properties}.
\begin{proposition}\label{minimal-speed}
Assume $k(0,e)<0$. Then the following conclusions hold true.
\begin{enumerate}[label=\textup{(\roman*)}]
\item The quantity $c^*(e)$ defined by \eqref{formula-minimal-speed-inf} is a real number.
    Moreover, for every fixed $c\geq c^*(e)$, the following quantity
\begin{equation}\label{decay-rate}
  \lambda_c(e):=\min\{\lambda>0\mid k(\lambda,e)+c\lambda=0\}
\end{equation}
is well-defined and there holds
\begin{equation}\label{formula-minimal-speed}
  c^*(e)=\min_{\lambda>0}\frac{-k(\lambda,e)}{\lambda}.
\end{equation}
\item Both $\lambda_c(e)$ and $c^*(e)$ are continuous in $e\in\mathbb{S}^{N-1}$.
\item The equation
  \begin{equation}\label{char-equa}
    k(\lambda,e)+c\lambda=0~~\text{ for }\lambda>0
  \end{equation}
  has solutions if and only if $c\geq c^*(e)$. For $c=c^*(e)$, there exists a unique solution $\lambda^*:=\lambda_{c^*}(e)>0$, while for all $c>c^*(e)$, there are exactly two solutions $\lambda^-_c:=\lambda_{c}(e)$ and $\lambda^+_c$ satisfying  $0<\lambda^-_c<\lambda^+_c$.
\end{enumerate}
\end{proposition}
\begin{remark}\label{sufficient-condition}
  From \cite[Proposition 2.2]{griette2021propagation}, two sufficient conditions for $\lambda_1=\lambda^p_1$ to hold can be summarized as follows:
\begin{itemize}
  \item For each $i\in\{1,\ldots,d\}$, the operator $L^i$ is in divergence form \eqref{divergence} and has no advection, that is, $q^i(x)\equiv\bm{0}$.
      Moreover, the matrix field $H$ in \eqref{nonsymmetric-periodic-PE} is symmetric.
  \item For each $i\in\{1,\ldots,d\}$, the operator $L^i$ is in nondivergence form \eqref{nondivergence}.
      Furthermore, the diffusion matrix field $x\mapsto A^i(x)$ in $L^i$ and the function $x\mapsto H(x)$ in \eqref{nonsymmetric-periodic-PE} are even, while the advection vector field $x\mapsto q^i(x)$ in $L^i$ is odd.
\end{itemize}
Indeed, one can prove that the function $\lambda\mapsto k(\lambda,e)$ is even provided one of the above conditions is fulfilled.
Hence its global maximum over $\mathbb{R}$ is necessarily attained at $\lambda=0$ due to
the strict concavity of $\lambda\mapsto k(\lambda,e)$.
We then have $\lambda_1=k(0,\cdot)=\lambda^p_1$ by \eqref{maximum-periodic-PE}.
However, a very interesting counter-example was constructed by Griette and Matano  \cite{griette2021propagation} for which $\lambda^p_1<\lambda_1$ holds, see Remark 4.2 and Appendix of \cite{griette2021propagation} for details.
Nevertheless, due to the smoothness assumptions of \ref{item:coefficient}, our proof
only needs to handle the operators defined by \eqref{nondivergence}.
\end{remark}
\section{Qualitative properties of pulsating waves}\label{sect:qualitative-estimates}
This section is devoted to the proof of Theorem \ref{main-result-qualitative-property}.
\subsection{Lower bound on the wave speed}
We prove here that the speed $c$ of pulsating waves for \eqref{monotone-system} is always bounded from below by the quantity $c^*(e)$,
which especially implies that the system \eqref{monotone-system} has no such solutions for all $c<c^*(e)$.
The proof of the first part of Theorem \ref{main-result-qualitative-property}
is given by the following proposition.
\begin{proposition}\label{lower-bound-speed}
Let $u=(u_1,\ldots,u_d)^T(t,x)$ be a pulsating travelling wave of \eqref{monotone-system} propagating in the direction $e\in\mathbb{S}^{N-1}$ with speed $c\neq0$.
Using the change of variables
\begin{equation}\label{moving-coordinate}
U(s,x)=u\left(\frac{x\cdot e-s}{c},x\right)~~
\text{ for $s\in\mathbb{R}$ and $x\in\mathbb{R}^N$},
\end{equation}
then each component $U_i$ of $U=(U_1,\ldots,U_d)^T$ satisfies
\begin{align*}
-\infty<\underline{\mu}\leq\liminf_{s\to+\infty}
\left(\min_{x\in\mathbb{R}^N}\frac{-\partial_s U_i(s,x)}
{U_i(s,x)}\right)\leq
\limsup_{s\to+\infty}\left(\max_{x\in\mathbb{R}^N}
\frac{-\partial_s U_i(s,x)}{U_i(s,x)}\right)
\leq\overline{\Lambda}<+\infty,
\end{align*}
where $\underline{\mu}$ and $\overline{\Lambda}$ are defined as
\begin{equation}\label{min-max-limits}
\underline{\mu}=\min_{1\leq i\leq d}\mu_i~\text{ and }~~
\overline{\Lambda}=\max_{1\leq i\leq d}\Lambda_i
\end{equation}
with
\begin{equation}\label{derivative-lower-upper-limits}
\mu_i:=\liminf_{s\to+\infty}\left(\min_{x\in\mathbb{R}^N}
\frac{-\partial_s U_i(s,x)}{U_i(s,x)}\right)
~\text{ and }~~
\Lambda_i:=\limsup_{s\to+\infty}\left(\max_{x\in\mathbb{R}^N}
\frac{-\partial_s U_i(s,x)}{U_i(s,x)}\right).
\end{equation}

Assume further that $\lambda^p_{1}<0$. Then both $\underline{\mu}$ and
$\overline{\Lambda}$ are positive and solve \eqref{char-equa}.
In particular, one has
$$
c\geq c^*(e),
$$
where $\lambda_1^p$ and $c^*(e)$ are given by \eqref{periodic-threshold} and \eqref{formula-minimal-speed}, respectively.
\end{proposition}
\begin{proof}
The proof relies on arguments that are close to the ones used by Hamel
\cite[Proposition 2.2]{hamel2008qualitative}.
Here is actually a generalization towards reaction-diffusion systems.

We first prove that two quantities $\underline{\mu}$ and $\overline{\Lambda}$ in \eqref{min-max-limits} are well-defined.
To this end, let us rewrite the system \eqref{monotone-system} as
\begin{equation}\label{lin-equ}
\partial_{t}u(t,x)+Lu(t,x)=H(t,x)u(t,x),~~~t\in\mathbb{R},~x\in\mathbb{R}^{N},
\end{equation}
where the matrix-valued function
$H=(h_{ij})_{1\leq i,j\leq d}:\mathbb{R}\times\mathbb{R}^N\to\mathcal{M}_d(\mathbb{R})$
is defined by
\begin{equation}\label{H-def}
h_{ij}(t,x)=\int_{0}^{1}\frac{\partial f_i}{\partial u_j}
\left(x,su_1(t,x),\ldots,su_d(t,x)\right)\mathrm{d}s.
\end{equation}
By assumption \ref{item:smoothness} on $f$ coupled with the boundedness and regularity of $u$ as a pulsating wave for \eqref{monotone-system},
it is clearly true that all entries of $H$ are globally bounded and of class $C^{\alpha/2,\alpha}$.
Fix any $(t_0,x_0)\in\mathbb{R}\times\mathbb{R}^N$ and
let $P_r(t_0,x_0):=[t_0-r^2,t_0]\times \overline{B_r(x_0)}$ be the closed parabolic cylinder of radius $r$, height $r^2$, and top center $(t_0,x_0)$.
From the Schauder interior estimates for parabolic systems
(see \cite[Chap.\uppercase\expandafter{\romannumeral7}]{ladyvzenskaja1988linear} or \cite{schlag1996schauder}), there exists $C_r>0$ independent of $(t_0,x_0)$ such that
$$
\forall r\in(0,1),~~\|u\|_{C^{1+{\alpha}/{2},2+\alpha}(P_{r})^d}\leq C_r\|u\|_{L^\infty(P_1)^d}.
$$
In particular, one has, for all $(t_0,x_0)\in\mathbb{R}\times\mathbb{R}^N$ and for each $i\in\{1,\ldots,d\}$,
\begin{equation}\label{inter-esti}
 |\partial_tu_i(t_0,x_0)|+|\nabla u_i(t_0,x_0)|
\leq C\max_{1\leq j\leq d}\max_{(t,x)\in P_{1}}u_j(t,x).
\end{equation}
Let now $\theta_c>1$ be given where $c\neq0$ is the speed of the pulsating wave $u$ of \eqref{monotone-system}.
By our assumptions \ref{item:monotonicity} and \ref{item:irreducibility}, the matrix $H$ given by \eqref{H-def} is cooperative and fully coupled.
Applying the Harnack inequality from Theorem \ref{Har-ineq} to \eqref{lin-equ}, we then deduce that there exists $\widehat\kappa_{\theta_c}>0$ independent of $(t_0,x_0)$ such that
\begin{align*}
\max_{1\leq j\leq d}\max_{(t,x)\in P_{1}}u_j(t,x)&\leq
\max_{1\leq j\leq d}\max_{(t,x)\in [t_0-\theta,t_0]\times[x_0-\theta_c,x_0+\theta_c]^N}u_j(t,x)\\
&\leq
\widehat\kappa_{\theta_c}\min_{1\leq j\leq d}
\min_{(t,x)\in\left[t_0+\frac{3\theta_c}{2},t_0+2\theta_c\right]\times
[x_0-\theta_c,x_0+\theta_c]^N}u_j(t,x).
\end{align*}
Choose a vector $k_0\in\mathbb{Z}^N$ such that $3\theta_c/2\leq(k_0\cdot e)/c\leq 2\theta_c$ and $|k_0|\leq\theta_c$
(increasing the value of $\theta_c$ if necessary, such a vector always exists).
Then we have that
$$
\forall(t_0,x_0)\in\mathbb{R}\times\mathbb{R}^N,~~
\max_{1\leq j\leq d}\sup_{(t,x)\in P_{1}}u_j(t,x)\leq
\widehat\kappa_{\theta_c}
u_i\left(t_0+\frac{k_0\cdot e}{c},x_0+k_0\right),
~~i=1,\ldots,d.
$$
Using the pulsating properties \eqref{pulsating-condition} of the wave $u$ together with \eqref{inter-esti}, one gets
\begin{equation}\label{global-estimate}
  \sup_{(t_0,x_0)\in\mathbb{R}\times\mathbb{R}^N}
  \left(\frac{|\partial_tu_i(t_0,x_0)|}{u_i(t_0,x_0)}
  +\frac{|\nabla u_i(t_0,x_0)|}{u_i(t_0,x_0)}\right)
  \leq \widehat\kappa_{\theta_c}C<+\infty,~~~i=1,\ldots,d.
\end{equation}
Observe from \eqref{moving-coordinate} that
$$
  \frac{-\partial_s U_i(s,x)}{U_i(s,x)}=
  \frac{\partial_t u_i\big((x\cdot e-s)/c,x\big)}
  {cu_i\big((x\cdot e-s\big)/c,x\big)},~~~~i=1,\ldots,d.
$$
Hence \eqref{global-estimate} ensures that the function $\partial_sU_i/U_i$ for $i\in\{1,\ldots,d\}$ is globally bounded in $\mathbb{R}\times\mathbb{R}^N$.
As a consequence, all the quantities $\mu_i$ and $\Lambda_i$, defined in \eqref{derivative-lower-upper-limits},  are real numbers.

Our next aim is to prove that both $\underline{\mu}$ and $\overline{\Lambda}$, defined in \eqref{min-max-limits}, are positive.
By \eqref{pulsating-condition}, the function $U:\mathbb{R}\times\mathbb{R}^N\to\mathbb{R}^d$ defined by \eqref{moving-coordinate} is $\mathbb{Z}^N$-periodic with respect to the second variable.
Moreover, it satisfies
\begin{equation}\label{U-positivity-limit}
  U(+\infty,\cdot)=\bm{0}~\text{ and }~
  U(s,x)\gg\bm{0}~\text{ for all }(s,x)\in\mathbb{R}\times\mathbb{R}^N.
\end{equation}
From the definition of $\underline{\mu}$ in \eqref{min-max-limits} and using the periodicity of $U$, there exist $i_0\in\{1,\ldots,d\}$ and a sequence $\{(s_n,x_n)\}_{n\in\mathbb{N}}$ such that
\begin{equation}\label{limit-derivative-minimal-bound}
  \text{$x_n\in\mathbb{T}^N$, $s_n\to+\infty$ ~and~ }
  \frac{-\partial_s U_{i_0}(s_n,x_n)}{U_{i_0}(s_n,x_n)}
  \to\underline{\mu}~~\mbox{ as }n\to\infty.
\end{equation}
By the compactness of $\mathbb{T}^N$, one may assume, up to extracting a subsequence, that $x_n\to x_\infty\in\mathbb{T}^N$ as $n\to\infty$.
Define
\begin{equation}\label{vn-def}
  v^n(t,x)=\frac{u(t+t_n,x)}{u_{i_0}(t_n,x_n)}~
  \text{ with }t_n=\frac{x_n\cdot e-s_n}{c}.
\end{equation}
For each $n\in\mathbb{N}$, each component of $v^n$ satisfies
\begin{equation}\label{vn-equ}
  \partial_tv^n_i(t,x)+L^iv^n_i(t,x)-
  \frac{f_i(x,u(t+t_n,x))}{u_i(t+t_n,x)}v^n_i(t,x)=0,~~~i=1,\ldots,d,
\end{equation}
where the operator $L^i$ is given by \eqref{nondivergence}.
Moreover, since $ct_n=x_n\cdot e-s_n\to-\infty$ as $n\to\infty$ by \eqref{limit-derivative-minimal-bound},
it follows from \eqref{pulsating-condition} and \eqref{limiting-condition} that
\begin{equation}\label{time-shift}
 \lim_{n\to\infty}u(t+t_n,x)=\bm{0}~~
  \text{ locally uniformly for $(t,x)\in\mathbb{R}\times\mathbb{R}^N$.}
\end{equation}
Now, using Taylor's expansion for each $f_i$ and by assumption \ref{item:smoothness}, one has
\begin{align*}
  \frac{f_i(x,u(t+t_n,x))}{u_i(t+t_n,x)}
  =&\frac{\partial f_i(x,\bm{0})}{\partial u_1}
  \frac{u_1(t+t_n,x)}{u_i(t+t_n,x)}
  +\cdots+
  \frac{\partial f_i(x,\bm{0})}{\partial u_i}+\cdots
+\frac{\partial f_i(x,\bm{0})}{\partial u_d}
  \frac{u_d(t+t_n,x)}{u_i(t+t_n,x)}\\
  &+\frac{1}{2}\frac{u^T(t+t_n,x)D_u^2f_i(x,\bm{0})\,u(t+t_n,x)}{u_i(t+t_n,x)}
  +\frac{o(|u(t+t_n,x)|^2)}{u_i(t+t_n,x)}.
\end{align*}
Recalling the definition \eqref{vn-def} of $v_n$,
the above formula applied in \eqref{vn-equ} then yields
$$
\partial_tv^n_i+L^iv^n_i-\sum_{j=1}^{d}\frac{\partial f_i(x,\bm{0})}{\partial u_j}v^n_j
  =\frac{1}{2}\frac{u^T(t+t_n,x)D_u^2f_i(x,\bm{0})\,u(t+t_n,x)}{u_{i_0}(t_n,x_n)}+
  \frac{o(|u(t+t_n,x)|^2)}{u_{i_0}(t_n,x_n)}
$$
for all $i\in\{1,\ldots,d\}$.
Note also that from \eqref{global-estimate}, the functions $v^n(t,x)=u(t+t_n,x)/u_{i_0}(t_n,x_n)$ are locally bounded in $\mathbb{R}\times\mathbb{R}^N$.
Combining this fact with \eqref{time-shift}, and from standard parabolic estimates, $v^n$ converges in $C^{1,2}_{\rm loc}(\mathbb{R}\times\mathbb{R}^N)^d$, up to a diagonal extraction of a subsequence,  to a nonnegative entire solution $v^\infty$ to the system
\begin{equation}\label{v-infty-equ}
\partial_tv^\infty(t,x)+Lv^\infty(t,x)-D_uf(x,\bm{0})
v^\infty(t,x)=\bm{0}~~\text{ for }(t,x)\in\mathbb{R}\times\mathbb{R}^N.
\end{equation}
Since $u$ is a pulsating wave according to Definition \ref{pulsating-wave-def}, the function $v^n$ defined in \eqref{vn-def} satisfies the same pulsating relations as in \eqref{pulsating-condition} and $v^n\gg\bm{0}$.
Therefore, the above convergence also yields $v^\infty\geq\bm{0}$ and
\begin{equation}\label{v-infty-pulsating}
\forall(t,x,k)\in\mathbb{R}\times\mathbb{R}^N\times\mathbb{Z}^N,~~
v^\infty\left(t+\frac{k\cdot e}{c},x\right)=v^\infty(t,x-k).
\end{equation}
In particular, since $D_uf(x,\bm{0})$ is cooperative, we have
$$
\begin{dcases}
\partial_tv^\infty_{i_0}+L^{i_0}v^\infty_{i_0}-
\frac{\partial f_{i_0}(x,\bm{0})}{\partial u_{i_0}}v^\infty_{i_0}=
\sum_{j=1,\,j\neq i_0}^{d}\frac{\partial f_{i_0}(x,\bm{0})}
{\partial u_j}v^\infty_j \geq0\text{ ~in }\mathbb{R}\times\mathbb{R}^N,\\
\forall(t,x)\in\mathbb{R}\times\mathbb{R}^N,~~v^\infty_{i_0}(t,x)\geq0,
~~v^\infty_{i_0}(0,x_\infty)=1,
\end{dcases}
$$
where the operator $L^{i_0}$ is given by \eqref{nondivergence} with $i=i_0$.
It then follows from the strong maximum principle that $v^\infty_{i_0}$ is positive everywhere in $\mathbb{R}\times\mathbb{R}^N$.
Since $D_uf(x,\bm{0})$ is also fully coupled, applying the Harnack inequality from Theorem \ref{Har-ineq} together with the pulsating property of $v^\infty$,
one can easily deduce that all other components of $v^\infty$ are also positive everywhere,
that is, $v^\infty(t,x)\gg\bm{0}$ for all $(t,x)\in\mathbb{R}\times\mathbb{R}^N$.

On the other hand, from \eqref{vn-def}, one has, for each $i\in\{1,\ldots,d\}$,
$$
\frac{\partial_t v^n_i(t,x)}{v^n_i(t,x)}=
\frac{\partial_tu_i(t+t_n,x)}{u_i(t+t_n,x)}=c\times
\frac{-\partial_s U_i\left(x\cdot e-c(t+t_n),x\right)}
{U_i\left(x\cdot e-c(t+t_n),x\right)},
~~(t,x)\in\mathbb{R}\times\mathbb{R}^N.
$$
Since $ct_n=x_n\cdot e-s_n\to-\infty$ as $n\to\infty$, we infer from the definition of the quantity $\underline{\mu}$ in \eqref{min-max-limits} and  \eqref{derivative-lower-upper-limits}  that for all $(t,x)\in\mathbb{R}\times\mathbb{R}^N$,
\begin{equation}\label{v-infty-derivative-bound}
\begin{split}
\left(\frac{\partial_t v^\infty_1(t,x)}{v^\infty_1(t,x)},\cdots,
\frac{\partial_tv^\infty_d(t,x)}{v^\infty_d(t,x)}\right)
\geq c\underline{\mu}\bm{1}~~
\text{ if }c>0,\\
\left(\frac{\partial_tv^\infty_1(t,x)}{v^\infty_1(t,x)},\cdots,
\frac{\partial_tv^\infty_d(t,x)}{v^\infty_d(t,x)}\right)
\leq c\underline{\mu}\bm{1}~~
\text{ if }c<0,
\end{split}
\end{equation}
while ${\partial_tv^\infty_{i_0}(0,x_\infty)}/{v^\infty_{i_0}(0,x_\infty)}=c\underline{\mu}$
due to \eqref{limit-derivative-minimal-bound}.
Let us now prove that $\partial_tv^\infty\equiv c\underline{\mu}v^\infty$.
For that purpose, we define a function $w=(w_1,\ldots,w_d)^T$ for $(t,x)\in\mathbb{R}\times\mathbb{R}^N$ as
$$
w(t,x)=\partial_tv^\infty(t,x)- c\underline{\mu}v^\infty(t,x).
$$
It is straightforward to check that $w$ satisfies the same equation \eqref{v-infty-equ} as $v^\infty$.
Due to $v^\infty\gg\bm{0}$ and \eqref{v-infty-derivative-bound}, one has that
\begin{equation}\label{w-sign}
\forall(t,x)\in\mathbb{R}\times\mathbb{R}^N,~~~
w(t,x)\geq\bm{0}~(\text{resp.}~\leq\bm{0})
~~~\text{if }c>0~(\text{resp. }\text{if }c<0).
\end{equation}
Note also that
$$
w_{i_0}(0,x_\infty)=\partial_tv^\infty_{i_0}(0,x_\infty)-
c\underline{\mu}v^\infty_{i_0}(0,x_\infty)=0.
$$
Therefore, $w_{i_0}$ reaches its minimum or maximum (depending on the sign of $c$) $0$ at $(0,x_\infty)$.
From the strong maximum principle and using the pulsating relations of $v_{i_0}^\infty$, one can deduce that $w_{i_0}(t,x)=0$ for all $(t,x)\in\mathbb{R}\times\mathbb{R}^N$.
As argued above, the fact that $D_uf(x,\bm{0})$ is fully coupled further yields  $w=\bm{0}$.

In other words, one gets $\partial_tv^\infty\equiv c\underline{\mu}v^\infty$.
Since $v^\infty$ is a solution of \eqref{v-infty-equ},
it can be written as
$$
v^\infty(t,x)=e^{-\underline{\mu}(x\cdot e-ct)}\psi(x).
$$
where $\psi:=(\psi_1,\ldots,\psi_d)$ satisfies
$$
-L_{\underline{\mu}e}\psi+D_uf(x,\bm{0})\psi
=c\underline{\mu}\psi
\text{ ~in }\mathbb{T}^N.
$$
Here the operator
$L_{\underline{\mu}e}:={\rm diag}(L^1_{\underline{\mu}e},\ldots,L^d_{\underline{\mu}e})$ is  given by \eqref{modified-operator} with $\lambda=\underline{\mu}$.
By \eqref{v-infty-pulsating} and $v^\infty\gg\bm{0}$, it is easy to verify that $\psi$ is $\mathbb{Z}^N$-periodic and $\psi\gg\bm{0}$.
We then infer from Proposition \ref{periodic-principal-eigenvalue-properties} that
$$
-k\big(\underline{\mu},e\big)=c\underline{\mu}.
$$
As for the quantity $\overline{\Lambda}$ defined in \eqref{min-max-limits}, replacing $\underline{\mu}$ by $\overline{\Lambda}$ and reversing inequalities in \eqref{v-infty-derivative-bound}, one can similarly repeat the above proof to get
$$
-k\big(\overline{\Lambda},e\big)=c\overline{\Lambda}.
$$
Since $\lambda\mapsto k(\lambda,e)$ is strictly concave,
$k(0,e)=\lambda^p_{1}<0$  and $c\neq0$, it follows that $\underline{\mu}$ and $\overline{\Lambda}$ cannot be zero and have the same sign.
From the definitions of $\Lambda_i$ in \eqref{derivative-lower-upper-limits} and $\overline{\Lambda}$ in \eqref{min-max-limits} together with \eqref{U-positivity-limit},
we conclude that both $\underline{\mu}$ and $\overline{\Lambda}$ are positive.

Finally, from Proposition \ref{minimal-speed}, one has
$$
c=\frac{-k\big(\overline{\Lambda},e\big)}{\overline{\Lambda}}=
\frac{-k\big(\underline{\mu},e\big)}{\underline{\mu}}
\geq\min_{\lambda>0}\frac{-k(\lambda,e)}{\lambda}=c^*(e).
$$
This ends the proof of Proposition \ref{lower-bound-speed}.
\end{proof}
\subsection{Monotonicity of the wave with respect to time}
We prove here that if the nonlinearity is subhomogeneous, then any pulsating travelling wave of \eqref{monotone-system} is strictly monotonic in time.

The following result corresponds to the second part of Theorem \ref{main-result-qualitative-property}.
\begin{proposition}\label{monotonocity-proof}
Assume that \eqref{subhomogeneity} holds.
Then any pulsating wave $(c,u)$ of \eqref{monotone-system} is componentwise increasing in $t$ if $c>0$ and decreasing if $c<0$.
Moreover, one has $c\partial_t u(t,x)\gg\bm{0}$ for all $(t,x)\in\mathbb{R}\times\mathbb{R}^N$.
\end{proposition}
\begin{proof}
The proof follows from the arguments of Berestycki et al. \cite[Theorem 1.2]{berestycki2005analysis-2} in which the monotonicity of pulsating fronts for strong KPP equations was proved.

Let $u(t,x)$ be a pulsating travelling wave of \eqref{monotone-system} propagating along the direction $e\in\mathbb{S}^{N-1}$ with speed $c\neq0$.
With the change of variables of Proposition \ref{lower-bound-speed}, we consider the function $U=(U_1,\ldots,U_d)^T(s,x)$ for $(s,x)\in\mathbb{R}\times\mathbb{T}^N$ defined by
$$
U(s,x)=u(t,x)~
\text{ with }t=\frac{x\cdot e-s}{c}\in\mathbb{R}.
$$
Then it is continuous and positive.
By Proposition \ref{lower-bound-speed}, each component of $U$ satisfies
$$
\liminf_{s\to+\infty,\,x\in\mathbb{T}^N}
\frac{-\partial_s U_i(s,x)}{U_i(s,x)}>0,~~i=1,\ldots,d,
$$
and thus there exists $\bar{s}\in\mathbb{R}$ such that
$$
\forall s\geq\bar{s},~\forall x\in\mathbb{T}^N,~~~
\partial_s U_i(s,x)<0,~~i=1,\ldots,d.
$$
Moreover, since the limiting condition \eqref{limiting-condition} for $u(t,x)$ is equivalent to
\begin{equation}\label{U-LR-limits}
 \lim_{s\to+\infty}U(s,x)=\bm{0}~\text{ and }~
 \liminf_{s\to-\infty}U(s,x)\gg\bm{0}~
  \text{ uniformly in }\mathbb{T}^N,
\end{equation}
then by continuity, we have
\begin{equation}\label{positive-inf}
\inf_{s\leq\bar{s},\,x\in\mathbb{T}^N}U(s,x)\gg\bm{0}.
\end{equation}
As a consequence, there exists some $R\in\mathbb{R}$ such that $R\geq\bar{s}$ and
\begin{equation}\label{left-semi-interval-monotonicity}
\forall\sigma\geq0,~\forall s\geq R,~\forall x\in\mathbb{T}^N,~~
U^\sigma(s,x):=U(s+\sigma,x)\leq U(s,x).
\end{equation}
One can assume $R>0$.
In the sequel, fix any $\sigma\geq0$ and set
$$
\gamma^*:=
\sup\left\{\gamma\geq0\mid\forall(s,x)\in(-\infty,R]\times\mathbb{T}^N,~
\gamma U^\sigma(s,x)<U(s,x)\right\}.
$$
Since $U$ is bounded and $\inf_{(-\infty, R)\times\mathbb{T}^N}U^\sigma\gg\bm{0}$ by \eqref{positive-inf}, the quantity $\gamma^*$ is finite and positive.

One has to prove that $\gamma^*\geq1$.
To this end, we argue by contradiction and assume $\gamma^*<1$.
Since $U(-\infty,\cdot)\gg\bm{0}$ by \eqref{U-LR-limits} and from the definition of $\gamma^*$, there exist $(s_0,x_0)\in(-\infty,R]\times\mathbb{T}^N$ and $i_0\in\{1,\ldots,d\}$ such that
\begin{equation}\label{contact-point}
\gamma^*U^\sigma(s_0,x_0)\leq U(s_0,x_0)
~\text{ and }~
\gamma^*U_{i_0}^\sigma(s_0,x_0)=U_{i_0}(s_0,x_0).
\end{equation}
By continuity, we have
$$
\gamma^*U^\sigma(s,x)\leq U(s,x)~~
\text{ for all }(s,x)\in(-\infty,R]\times\mathbb{T}^N,
$$
and
$$
\gamma^*U^\sigma(s,x)\leq U(s,x)~~
\text{ for all }(s,x)\in[R,+\infty)\times\mathbb{T}^N
$$
due to $\gamma^*<1$ and \eqref{left-semi-interval-monotonicity}.
Coming back to the original variables $(t,x)\in\mathbb{R}\times\mathbb{R}^N$, we define a function $w=(w_1,\ldots,w_d)^T(t,x)$ as
\begin{equation}\label{w-def}
w(t,x)=U(x\cdot e-ct,x)-\gamma^*U^\sigma(x\cdot e-ct,x).
\end{equation}
Then, $w(t,x)\geq\bm{0}$ for all $(t,x)\in\mathbb{R}\times\mathbb{R}^N$. Moreover, $w$ satisfies
$$
\partial_t w(t,x)+Lw(t,x)=f(x,U(x\cdot e-ct,x))
-\gamma^*f(x,U^\sigma(x\cdot e-ct,x)).
$$
From assumption \eqref{subhomogeneity},
one has $\gamma^*f(x,U^\sigma)\leq f(x,\gamma^*U^\sigma)$ due to $0<\gamma^*<1$
and thus
$$
\partial_t w+Lw=f(x,U)-\gamma^*f(x,U^\sigma)\geq
f(x,U)-f(x,\gamma^*U^\sigma).
$$
Let us now examine the differential inequality satisfied by the $i_0$-th component of $w$:
$$
\partial_t w_{i_0}+L^{i_0}w_{i_0}\geq
f_{i_0}(x,U)-f_{i_0}\big(x,\gamma^*U^\sigma\big),
$$
where $L^{i_0}$ is given by \eqref{nondivergence} with $i=i_0$.
Since $f$ is quasimonotone and $w\geq\bm{0}$, then
$$
f_{i_0}(x,U)-f_{i_0}(x,\gamma^*U^\sigma)
=\sum_{j=1}^{d} h_{i_0j}(t,x)w_j\geq
h_{i_0i_0}(t,x)w_{i_0},
$$
where one has set
$$
h_{i_0j}(t,x):=\int_{0}^{1}\frac{\partial f_{i_0}}{\partial u_j}
\big(x,\gamma^*U^\sigma_1(x\cdot e-ct,x)+sw_1(t,x),\ldots,
\gamma^*U^\sigma_d(x\cdot e-ct,x)+sw_d(t,x)\big)\mathrm{d}s.
$$
Therefore, we obtain that
$$
\forall(t,x)\in\mathbb{R}\times\mathbb{R}^N,~~
\partial_t w_{i_0}(t,x)+L^{i_0}w_{i_0}(t,x)-
h_{i_0i_0}(t,x)w_{i_0}(t,x)\geq0.
$$
Setting $t_0:={(x_0\cdot e-s_0)}/{c}$,
we then infer from \eqref{contact-point} and \eqref{w-def} that $w_{i_0}(t_0,x_0)=0$.
As $w_{i_0}$ is nonnegative in $\mathbb{R}\times\mathbb{R}^N$, it follows from the strong maximum principle that
$$
\forall t\leq t_0,~\forall x\in\mathbb{R}^N,~~~w_{i_0}(t,x)=0.
$$
Furthermore, since $U^\sigma(x\cdot e-ct,x)$ and $U(x\cdot e-ct,x)$ are $\mathbb{Z}^N$-periodic in the second variable, one can easily check that the function $w$ defined by \eqref{w-def} satisfies
$$
w\left(t+\frac{k\cdot e}{c},x\right)=w(t,x-k)~~
\text{ for all }(t,x,k)\in\mathbb{R}\times\mathbb{R}^N\times\mathbb{Z}^{N}
$$
and thus $w_{i_0}(t,x)=0$ for all $(t,x)\in\mathbb{R}\times\mathbb{R}^N$.
However, since $U^\sigma$ has the same asymptotics as in \eqref{U-LR-limits} and due to $\gamma^*<1$, then there exists $\epsilon>0$ independent of $\sigma$ such that
\begin{align*}
  &\forall x\in\mathbb{R}^N,~~\liminf_{t\to+\infty}w_{i_0}(t,x)=\liminf_{t\to+\infty}
  \left(U_{i_0}-\gamma^*U^\sigma_{i_0}\right)(x\cdot e-ct,x)
  =(1-\gamma^*)\epsilon>0~\text{ if $c>0$},\\
  &\forall x\in\mathbb{R}^N,~~\liminf_{t\to-\infty}w_{i_0}(t,x)=\liminf_{t\to-\infty}\left(
  U_{i_0}-\gamma^*U^\sigma_{i_0}\right)(x\cdot e-ct,x)
  =(1-\gamma^*)\epsilon>0~\text{ if $c<0$}.
\end{align*}
This leads to a contradiction. Consequently, one gets $\gamma^*\geq1$.
From the definition of $\gamma^*$ and by \eqref{left-semi-interval-monotonicity}, one has that
$$
\forall\sigma\geq0,~~ U(s,x)\geq U(s+\sigma,x)~\text{ for all }
(s,x)\in\mathbb{R}\times\mathbb{T}^N.
$$
Coming back to the original variables $(t,x)$, this means
that
$$
\forall\sigma\geq0,~~u(t,x)\geq u\left(t-\frac{\sigma}{c},x\right)~\text{ for all }
(t,x)\in\mathbb{R}\times\mathbb{R}^N.
$$
We further claim that the inequality is componentwise strict as soon as $\sigma>0$.
To see this, for any given $\sigma>0$ and $c\neq0$, we set
$$
v(t,x):=u(t,x)-u\left(t-\frac{\sigma}{c},x\right).
$$
From the uniqueness of the corresponding Cauchy problem, applying the strong maximum principle together with assumption \ref{item:irreducibility} to the linear parabolic system satisfied by $v$,
we can deduce that $v\equiv\bm{0}$ or $v\gg\bm{0}$ in $\mathbb{R}\times\mathbb{R}^N$. If $v\equiv\bm{0}$, we obtain in particular that
$$
\forall x\in\mathbb{R}^N,~~~u(0,x)=U(x\cdot e,x)=
U(x\cdot e+\sigma,x)=u\left(-{\sigma}/{c},x\right).
$$
which is impossible because of \eqref{U-LR-limits}.
We thereby conclude that $u(t,x)$ is componentwise increasing in the $t$ variable if $c>0$ and decreasing if $c<0$.

Lastly, by \eqref{global-estimate}, the function $w(t,x):=\partial_t u(t,x)$ is globally bounded.
If $c>0$, then $w$ is nonnegative and $w\not\equiv\bm{0}$.
Since $f$ is quasimonotone, each component of $w$ satisfies
$$
\partial_tw_i+L^iw_i-\frac{\partial f_i(x,u)}{\partial u_i}w_i
=\sum_{j=1,\,j\neq i}^{d}\frac{\partial f_i(x,u)}
{\partial u_j}w_j \geq0,~~~i=1,\ldots,d,
$$
Note that the coefficients of the zero-order term in the above equation are bounded.
With the same arguments as before, one can prove that all components of $w$ are positive everywhere in $\mathbb{R}\times\mathbb{R}^N$, that is, $\partial_t u\gg\bm{0}$ if $c>0$.
The case $c<0$ can be treated similarly.
That completes the proof of Proposition \ref{monotonocity-proof}.
\end{proof}

\section{Existence of pulsating waves}\label{sect:existence}
This section is devoted to the proof of Theorem \ref{main-result-existence}.
\subsection{Construction of two sub- and supersolutions}
In order to prove the existence results, we shall use appropriate sub- and supersolutions.
Let $h=(h_1,\ldots,h_d)^T$ be the vector-valued function defined by
\begin{equation}\label{supersol-form-any-speed}
 h(t,x)=e^{-\lambda_c(e)(x\cdot e-ct)}\phi_{\lambda_c(e)}(x)
\end{equation}
where $\phi_{\lambda_c(e)}\in C^2(\mathbb{T}^N)^d$ is the unique principal eigenfunction of \eqref{nonsymmetric-periodic-PE} in the sense of \eqref{normalization-principal-eigenfunction} corresponding to $\lambda=\lambda_c(e)$, and $\lambda_c(e)$ is defined by \eqref{decay-rate}.
By Propositions \ref{periodic-principal-eigenvalue-properties}-\ref{minimal-speed} and from \eqref{nonsymmetric-periodic-PE} applied with $H\equiv D_uf(x,\bm{0})$,
it is straightforward to check that $h$ is a supersolution of \eqref{monotone-system} for all $c\geq c^*(e)$ due to assumption \eqref{sublinearity}.
In what follows, we shall construct another supersolution of \eqref{monotone-system} for $c=c^*(e)$ and two subsolutions of \eqref{monotone-system} for $c>c^*(e)$ and $c=c^*(e)$, respectively.

For notational concision, let us temporarily forget the dependence of $k(\lambda,e)$, $\phi_{\lambda_c(e)}$, $\lambda_c(e)$ and $c^*(e)$ on $e\in\mathbb{S}^{N-1}$ in the next three lemmas.
Since $f(x,\bm{0})\equiv\bm{0}$ and $f\in C^{1,\beta}(\mathbb{T}^N\times B_\infty(\bm{0},\sigma))^d$ with $\sigma>0$  by assumption \ref{item:smoothness},
then there exists $M>0$ such that
\begin{align}\label{sharp-regularity-hypothesis}
   \left|{f(x,u)-D_uf(x,\bm{0})u}\right|_\infty\leq M|u|_\infty^{1+\beta}
   ~~\text{ for all }(x,u)\in\mathbb{T}^N\times\mathbb{R}^d
   \text{ with }|u|_\infty\leq\sigma.
\end{align}
Even if it means decreasing $\sigma>0$, one can assume without loss of generality that
\begin{equation}\label{sigma}
  0<\sigma<\widehat{\eta}
\end{equation}
where $\widehat{\eta}$ is given in assumption \ref{item:upper-bound}.
Let us mention that conditions similar to \eqref{sharp-regularity-hypothesis} have early been formulated for scalar nonlinear functions in \cite{hamel2008qualitative},
but there is a slight difference:
We allow in \eqref{sharp-regularity-hypothesis} that all coordinates of the vector $u$  may not have the same sign.
What it comes into play will be clearly understood in the next two lemmas.
\begin{lemma}\label{subsol-construction-larger-speeds}
Let $\omega=(\omega_1,\ldots,\omega_d)^T$ be the vector-valued function defined by
\begin{align}\label{subsol-form-larger-speeds}
\omega(t,x)=e^{-\lambda_c(x\cdot e-ct)}\phi_{\lambda_c}(x)-
Ke^{-(\lambda_c+\delta)(x\cdot e-ct)}\phi_{\lambda_c+\delta}(x)
\end{align}
where $\phi_{\lambda_c}$ and $\phi_{\lambda_c+\delta}$ are the unique principal eigenfunctions of \eqref{nonsymmetric-periodic-PE} in the sense of \eqref{normalization-principal-eigenfunction}
corresponding to $\lambda=\lambda_{c}$ and $\lambda_{c}+\delta$ with $\delta>0$, respectively.
Then for any $c>c^{*}$, there exist $K>0$ and
$\delta\in\big(0,\min(\beta\lambda_c, \lambda^+_{c}-\lambda_{c})\big)$,
where $\lambda^+_{c}$ and $\beta$ are given in Proposition \ref{minimal-speed} and \eqref{sharp-regularity-hypothesis}, such that$\colon$
\begin{itemize}
  \item $\sup_{(t,x)\in\mathbb{R}\times\mathbb{R}^N}\omega(t,x)\ll\widehat{\eta}\bm{1}$,
  \item  $\omega(t,x)\ll\bm{0}$ for all $(t,x)\in\mathbb{R}\times\mathbb{R}^N$ satisfying $x\cdot e-ct\leq0$,
  \item there holds
  $$
 \partial_t\omega+L\omega\leq f(x,\omega)
  $$
  on the set
  \begin{equation}\label{Omega+}
   \Omega_+:=\Big\{(t,x)\in\mathbb{R}\times\mathbb{R}^{N}\,\big|\,
   \max_{1\leq i\leq d}\omega_i(t,x)>0\Big\},
  \end{equation}
\end{itemize}
where $\widehat{\eta}$ is given in \ref{item:upper-bound} and $L={\rm diag}(L^1,\ldots,L^d)$ denotes a diagonal matrix of the operators $L^i$ given by \eqref{nondivergence}.
\end{lemma}
\begin{proof}
Let
\begin{equation}\label{a0-b0-def}
a_0:={\min_{1\leq i\leq d}\min_{x\in\mathbb{T}^N}(\phi_{\lambda_c})_i(x)}~\text{ and }~
b_0:={\min_{1\leq i\leq d}\min_{x\in\mathbb{T}^N}(\phi_{\lambda_c+\delta})_i(x)}.
\end{equation}
Since $\bm{0}\ll\phi_\lambda\leq\bm{1}$ by \eqref{normalization-principal-eigenfunction}, one has
\begin{align*}
\max_{1\leq i\leq d}\omega_i(t,x)&\leq
e^{-\lambda_c(x\cdot e-ct)}-b_0Ke^{-(\lambda_c+\delta)(x\cdot e-ct)}\\
&\leq\left(\frac{1}{b_0K}\right)^{\frac{\lambda_c}{\delta}}\times
\left[\left(\frac{\lambda_c}{\lambda_c+\delta}\right)^{\frac{\lambda_c}{\delta}}-
\left(\frac{\lambda_c}{\lambda_c+\delta}\right)^{\frac{\lambda_c+\delta}{\delta}}\right]
\end{align*}
where the second inequality holds since the resulting upper bound is the global maximum of the middle term and it is reached on the line $x\cdot e-ct=s_0$ with
$s_0:=\delta^{-1}\ln[b_0K(\lambda_c+\delta)/{\lambda_c}]$.
By \eqref{sigma}, we then obtain that
$$
\sup_{(t,x)\in\mathbb{R}\times\mathbb{R}^N}{\omega(t,x)}\leq
\sigma\bm{1}\ll\widehat{\eta}\bm{1}~~
\text{ if $K>K_1$}
$$
where
\begin{equation}\label{K1-def}
K_1:=\frac{1}{b_0}\left(\frac{1}{\sigma}\right)^{\frac{\delta}{\lambda_c}}\times
\left[\left(\frac{\lambda_c}{\lambda_c+\delta}\right)^{\frac{\lambda_c}{\delta}}-
\left(\frac{\lambda_c}{\lambda_c+\delta}\right)^{\frac{\lambda_c+\delta}{\delta}}\right]
^{\frac{\delta}{\lambda_c}}>0.
\end{equation}
Furthermore, in order to match the condition \eqref{sharp-regularity-hypothesis},
we need to search for $K$ such that
\begin{equation}\label{max-norm-control}
\forall(t,x)\in\Omega_+,~~~
|\omega(t,x)|_\infty=\max_{1\leq i\leq d}|\omega_i(t,x)|\leq\sigma,
\end{equation}
where $\Omega_+$ is defined by \eqref{Omega+} and $\sigma$ is given in \eqref{sharp-regularity-hypothesis}.
It suffices to choose some $K>K_1$ such that
$$
\forall(t,x)\in\Omega_+,~~~a_0e^{-\lambda_c(x\cdot e-ct)}
-Ke^{-(\lambda_c+\delta)(x\cdot e-ct)}\geq-\sigma,
$$
where $a_0$ has been defined in \eqref{a0-b0-def}.
Observe that
$\Omega_+\subset\left\{(t,x)\mid
x\cdot e-ct\geq\delta^{-1}\ln(b_0K)\right\}$. Then
\begin{align*}
 a_0e^{-\lambda_c(x\cdot e-ct)}-Ke^{-(\lambda_c+\delta)(x\cdot e-ct)}
 \geq a_0(b_0K)^{-\frac{\lambda_c}{\delta}}-
 K(b_0K)^{-\frac{\lambda_c+\delta}{\delta}}
 ~~\text{ in }\Omega_+,
\end{align*}
where $b_0$ has been defined in \eqref{a0-b0-def}.
Hence \eqref{max-norm-control} holds as long as $K\geq\max(K_1,K_2)$ where $K_1$ is given by \eqref{K1-def} and one has set
\begin{equation}\label{K2-def}
K_2:=\left(\frac{1-a_0b_0}{\sigma b_0}\right)^{\frac{\delta}{\lambda_c}}>0.
\end{equation}

Next, one can easily check that $\omega(t,x)\ll\bm{0}$ for any $\delta>0$ and for all
$(t,x)\in\mathbb{R}\times\mathbb{R}^N$ satisfying $x\cdot e-ct\leq0$ as long as $K>1/b_0$ where $b_0$ is defined in \eqref{a0-b0-def}.
On the other hand, direct computations show
$$
\partial_t\omega+L\omega
 =D_uf(x,\bm{0})\omega-\left(c\lambda_{c}+c\delta+k(\lambda_{c}+\delta)\right)
 K\phi_{\lambda_{c}+\delta}(x)e^{-(\lambda_{c}+\delta)(x\cdot e-ct)}.
$$
From Propositions \ref{periodic-principal-eigenvalue-properties} and \ref{minimal-speed},
if $c>c^*$, one has
$$
 r_\delta:=c\lambda_c+c\delta+k(\lambda_c+\delta)>0~~
 \text{ for all }\delta\in(0,\lambda_c^+-\lambda_c).
$$
To conclude one has to check that
\begin{equation}\label{ineq-sub}
\forall(t,x)\in\Omega_+,~~D_uf(x,\bm{0})\omega(t,x)-f(x,\omega(t,x))\leq
r_\delta K\phi_{\lambda_{c}+\delta}(x) e^{-(\lambda_{c}+\delta)(x\cdot e-ct)}.
\end{equation}

Fix $\delta\in\left(0,\min\{\beta\lambda_c,\lambda^+_{c}-\lambda_{c}\}\right)$ and
$K\geq\max(K_1,K_2,b^{-1}_0,Mb^{-1}_0/r_\delta)$.
Due to \eqref{max-norm-control},
one can then infer from \eqref{sharp-regularity-hypothesis} and the expression \eqref{subsol-form-larger-speeds} of $\omega$ that for all $(t,x)\in\Omega_+$,
\begin{align*}
|f(x,\omega(t,x))-D_uf(x,\bm{0})\omega(t,x)|_\infty\leq
M|\omega(t,x)|_\infty^{1+\beta}
\leq M\Big(\max_{1\leq i\leq d}(\phi_{\lambda_c})_i(x)\Big)^{1+\beta}
e^{-\lambda_c(1+\beta)(x\cdot e-ct)}.
\end{align*}
As a consequence, \eqref{ineq-sub} holds if
$$
M\left(\max_{1\leq i\leq d}\left(\phi_{\lambda_c}\right)_i(x)\right)^{1+\beta}
e^{-\lambda_c(1+\beta)(x\cdot e-ct)}\leq
r_\delta K\min_{1\leq i\leq d}(\phi_{\lambda_{c}+\delta})_i(x)
e^{-(\lambda_{c}+\delta)(x\cdot e-ct)}.
$$
With the above choice for the parameters $K$ and $\delta$, we further obtain that all points $(t,x)\in\mathbb{R}\times\mathbb{R}^N$ contained in $\Omega_+$ satisfy $x\cdot e-ct>0$ and that for all $(t,x)\in\Omega_+$,
\begin{align*}
&~M\left(\max_{1\leq i\leq d}\left(\phi_{\lambda_c}\right)_i(x)\right)^{1+\beta}
e^{-(1+\beta)\lambda_c(x\cdot e-ct)}-
r_\delta K\min_{1\leq i\leq d}(\phi_{\lambda_{c}+\delta})_i(x)
e^{-(\lambda_{c}+\delta)(x\cdot e-ct)}\\
&\leq\left[M\left(\max_{1\leq i\leq d}\left(\phi_{\lambda_c}\right)_i(x)\right)^{1+\beta}-
r_\delta K\min_{1\leq i\leq d}(\phi_{\lambda_{c}+\delta})_i(x)\right]
e^{-(\lambda_{c}+\delta)(x\cdot e-ct)}\\
&\leq(M-r_\delta Kb_0)e^{-(\lambda_{c}+\delta)(x\cdot e-ct)}\\
&\leq0.
\end{align*}
This ends the proof of Lemma \ref{subsol-construction-larger-speeds}.
\end{proof}

Inspired by \cite[Proposition 4.5]{hamel2008qualitative},
we now construct a delicate subsolution of \eqref{monotone-system} for $c=c^*$.
Recall first that
from Proposition \ref{minimal-speed} and the formula \eqref{formula-minimal-speed} of $c^*$, we know that $\lambda^*$ is a unique minimizer of the function
$(0,\infty)\ni\lambda\mapsto-k(\lambda)/\lambda$.
Therefore, by analyticity and strict concavity of $\lambda\mapsto k(\lambda)$, we have
\begin{equation}\label{multiplicity}
   k(\lambda^*)+c^*\lambda^*=0,~k^{\prime}(\lambda^*)+c^*=0
   ~\text{ and }~k^{\prime\prime}(\lambda^*)\leq0.
\end{equation}
Furthermore, it also holds that
\begin{equation}\label{r*-delta-def}
r^*_\delta:=(\lambda^*+\delta)c^*+k(\lambda^*+\delta)<0
~\text{ for all }\delta>0.
\end{equation}

To proceed we define
\begin{equation}\label{v-star}
 v^*(t,x;\lambda):=e^{-\lambda\left(x\cdot e-c^*t\right)}\phi_{\lambda}(x)~~
\text{ for $(t,x)\in\mathbb{R}\times\mathbb{R}^{N}$ and $\lambda\in\mathbb{R}$}.
\end{equation}
Since the normalized principal eigenfunction
 $\phi_{\lambda}$ of \eqref{nonsymmetric-periodic-PE} in the sense of
 \eqref{normalization-principal-eigenfunction} is also analytic with respect to
$\lambda\in\mathbb{R}$, we have that
\begin{equation}\label{v-derivative-lambda}
\frac{\partial v^*(t,x;\lambda)}{\partial\lambda}
=-\left[(x\cdot e-c^*t)\phi_{\lambda}(x)-\frac{\partial\phi_{\lambda}(x)}
{\partial\lambda}\right]e^{-\lambda(x\cdot e-c^*t)},
\end{equation}
and that $\frac{\partial\phi_{\lambda}}
{\partial\lambda}\in C^2(\mathbb{T}^N)^d$ satisfies
$$
L_{\lambda}\frac{\partial\phi_{\lambda}}{\partial\lambda}
-D_uf(x,\bm{0})\frac{\partial\phi_{\lambda}}{\partial\lambda}
-k(\lambda)\frac{\partial\phi_{\lambda}}{\partial\lambda}
=L^\prime_{\lambda}\phi_{\lambda}-k^\prime(\lambda)\phi_{\lambda}
~~\text{ in }\mathbb{T}^N,
$$
where $L^\prime_{\lambda}$ denotes the derivative of the operator $L_{\lambda}$ given by \eqref{modified-operator} with respect to $\lambda\in\mathbb{R}$.
\begin{lemma}\label{subsol-construction-critical}
   Let $\omega^*=(\omega^*_1,\ldots,\omega^*_d)^T$ be the vector-valued function defined by
  \begin{equation}\label{subsol-form-critical}
  \omega^*(t,x)=
  \begin{dcases}
    \kappa_1v^*(t,x;\lambda^*+\delta)-\kappa_2\left(\frac{\partial v^*(t,x;\lambda)}
    {\partial\lambda}\bigg|_{\lambda=\lambda^*}+Kv^*(t,x;\lambda^*)\right)
           & \text{if }x\cdot e-c^*t\geq 0, \\
    \bm{0} & \text{if }x\cdot e-c^*t<0,
  \end{dcases}
\end{equation}
where $v^*$ is defined by \eqref{v-star}. Then for any fixed $\delta\in(0,\beta\lambda^*)$ where $\beta$ is given in \eqref{sharp-regularity-hypothesis}, there exist $K>0$, $s_0>0$ and $0<\kappa_1\leq\kappa_2$ (depending on $s_0$) such that$\colon$
\begin{itemize}
  \item $\omega^*(t,x)<\bm{0}$ for all $(t,x)\in\mathbb{R}\times\mathbb{R}^N$ with
  $x\cdot e-c^*t=0$,
  \item $\omega^*(t,x)\gg\bm{0}$ for all $(t,x)\in\mathbb{R}\times\mathbb{R}^N$ with $x\cdot e-c^*t\geq s_0$,
  \item $\sup_{(t,x)\in\mathbb{R}\times\mathbb{R}^N}\omega(t,x)\ll\widehat{\eta}\bm{1}$,
  \item there holds
    $$
       \partial_t\omega^*+L\omega^*\leq f(x,\omega^*)
    $$
  on the set
    \begin{equation}\label{Omega+*}
     \Omega^*_+:=\Big\{(t,x)\in\mathbb{R}\times\mathbb{R}^N\,\big|\,
     \max_{1\leq i\leq d}\omega_i^*(t,x)>0\Big\},
    \end{equation}
\end{itemize}
where $\widehat{\eta}$ is given in \ref{item:upper-bound} and $L={\rm diag}(L^1,\ldots,L^d)$ denotes a diagonal matrix of the operators $L^i$ given by \eqref{nondivergence}.
\end{lemma}
\begin{proof}
Since $\phi_{\lambda^*}$ and $\phi_{\lambda^*+\delta}$ are positive, one may fix some $K>0$ such that
\begin{equation}\label{K-choice}
\forall x\in\mathbb{T}^N,~~
\max_{1\leq i\leq d}\max_{x\in\mathbb{T}^N}(\phi_{\lambda^*+\delta})_i(x)
-\left(\frac{\partial\phi_{\lambda}(x)}{\partial\lambda}
\bigg|_{\lambda=\lambda^*}\right)_i
\leq
K\min_{1\leq i\leq d}\min_{x\in\mathbb{T}^N}
\left(\phi_{\lambda^*}\right)_i(x).
\end{equation}
From \eqref{v-derivative-lambda} applied at $\lambda=\lambda^*$,
such a choice for the parameter $K$ in \eqref{K-choice} then ensures that the function $\omega^*$ given by \eqref{subsol-construction-critical} is nonpositive and nonzero whenever $x\cdot e-c^*t=0$ as soon as $\kappa_2\geq\kappa_1>0$.
To proceed let us set, for notational concision,
\begin{equation}\label{cap-w-def}
W(t,x):=-\frac{\partial v^*(t,x;\lambda)}{\partial\lambda}
\bigg|_{\lambda=\lambda^*}-Kv^*(t,x;\lambda^*)
\end{equation}
so that
\begin{equation}\label{omega*-expression}
  \omega^*(t,x)=\kappa_1v^*(t,x;\lambda^*+\delta)+\kappa_2W(t,x)
  ~~\text{ if $x\cdot e-c^*t\geq0$}.
\end{equation}
Notice also that for each $\lambda>0$,
\begin{equation}\label{leading-term}
\frac{\partial v^*(t,x;\lambda)}{\partial\lambda}
=-(x\cdot e-c^*t)v^*(t,x;\lambda)
+o\big((x\cdot e-c^*t)v^*(t,x;\lambda)\big)~
\text{ as $x\cdot e-c^*t\to+\infty$}.
\end{equation}
Applying this estimate at $\lambda=\lambda^*$, we find that the leading term of $W$ is still given by
$$
(x\cdot e-c^*t)e^{-\lambda^*(x\cdot e-c^*t)}\phi_{\lambda^*}(x)
~~\text{  as $x\cdot e-c^*t\to+\infty$}.
$$
Thus, as $\delta>0$ and $\bm{0}\ll\phi_\lambda\leq\bm{1}$ (by \eqref{normalization-principal-eigenfunction}),
then there exists $s_0>0$ large enough such that
\begin{equation}\label{w-estimate-leading-edge}
\begin{split}
v^*(t,x;\lambda^*+\delta)&\leq e^{-(\lambda^*+\delta)(x\cdot e-c^*t)}\bm{1}
\leq W(t,x)\leq2(x\cdot e-c^*t)e^{-\lambda^*(x\cdot e-c^*t)}\bm{1},\\
&\text{for all $(t,x)\in\mathbb{R}\times\mathbb{R}^N$ satisfying
$x\cdot e-c^*t\geq s_0$},
\end{split}
\end{equation}
whence it holds that
$$
\forall \kappa_1>0, \kappa_2>0,~~~
\omega^*(t,x)=\kappa_1v^*(t,x;\lambda^*+\delta)+\kappa_2W(t,x)\gg\bm{0}
$$
for all $(t,x)\in\mathbb{R}\times\mathbb{R}^N$ satisfying $x\cdot e-c^*t\geq s_0$.

On the other hand, direct computations show that the function $v^*$ given by \eqref{v-star} satisfies
\begin{equation}\label{v-star-equation}
  \partial_t v^*(t,x;\lambda)+Lv^*(t,x;\lambda)-D_uf(x,\bm{0})v^*(t,x;\lambda)=
  \left(c^*\lambda+k(\lambda)\right)v^*(t,x;\lambda).
\end{equation}
Differentiating both sides of \eqref{v-star-equation} with respect to $\lambda$ yields
  \begin{align*}
    &\partial_t\left(\frac{\partial v^*(t,x;\lambda)}
    {\partial\lambda}\right)+L\left(\frac{\partial v^*(t,x;\lambda)}
    {\partial\lambda}\right)-D_uf(x,\bm{0})
    \frac{\partial v^*(t,x;\lambda)}{\partial\lambda}\\
    &=\left(c^*+k^\prime(\lambda)\right)v^*(t,x;\lambda)+
    (k(\lambda)+c^*\lambda)\frac{\partial v^*(t,x;\lambda)}
    {\partial\lambda}.
  \end{align*}
It then follows from \eqref{multiplicity} that
  \begin{equation}\label{star-v-star-equation}
    \partial_t\left(\frac{\partial v^*(t,x;\lambda)}
    {\partial\lambda}\bigg|_{\lambda=\lambda^*}\right)+
    L\left(\frac{\partial v^*(t,x;\lambda)}{\partial\lambda}\bigg|_{\lambda=\lambda^*}\right)
    =D_uf(x,\bm{0})\frac{\partial v^*(t,x;\lambda)}
    {\partial\lambda}\bigg|_{\lambda=\lambda^*}.
  \end{equation}
Note also that the function $v^*$ given by \eqref{v-star} with $\lambda=\lambda^*+\delta$ satisfies
\begin{align*}
  \partial_tv^*(t,x;\lambda^*+\delta)+Lv^*(t,x;\lambda^*+\delta)-
  D_uf(x,\bm{0})v^*(t,x;\lambda^*+\delta)=r_\delta^*v^*(t,x;\lambda^*+\delta)
\end{align*}
where $r_\delta^*$ has been defined in \eqref{r*-delta-def}.
As a result, the function $\omega^*$ given by \eqref{subsol-form-critical} when $x\cdot e-c^*t\geq 0$ then satisfies
\begin{align*}
\partial_t\omega^*(t,x)+L\omega^*(t,x)=D_uf(x,\bm{0})\omega^*(t,x)
+\kappa_1r_\delta^*v^*(t,x;\lambda^*+\delta).
\end{align*}

Our next aim is to prove that
\begin{equation}\label{ineq-proof}
\forall(t,x)\in\Omega_+^*,~~~
 D_uf(x,\bm{0})\omega^*(t,x)+\kappa_1r_\delta^*v^*(t,x;\lambda^*+\delta)\leq f(x,\omega^*(t,x)),
\end{equation}
where $\Omega_+^*$ is defined by \eqref{Omega+*}.
Let
\begin{equation}\label{Omega-prime-def}
\Omega^\prime:=\{(t,x)\in\mathbb{R}\times\mathbb{R}^N\mid x\cdot e-c^*t\geq0\}
\end{equation}
and $\overline{W}:=\|W\|_{L^\infty(\Omega^\prime)}$ where $W$ is given by \eqref{cap-w-def}.
From \eqref{omega*-expression}, we then find that
\begin{equation}\label{omega*-upper-bound}
\forall(t,x)\in\Omega^\prime,~~~
\omega^*(t,x)\leq\left(\kappa_1+\kappa_2\overline{W}\right)\bm{1}.
\end{equation}
To match the condition \eqref{sharp-regularity-hypothesis}, we take now any
$$
0<\kappa_1\leq\kappa_2\leq\frac{\sigma}{1+\overline{W}}
$$
so that $|\omega^*(t,x)|_\infty\leq\sigma$ for all $(t,x)\in\Omega^\prime$
where $\sigma$ is given in \eqref{sharp-regularity-hypothesis}.
In particular, due to \eqref{sigma}, it is straightforward to see from the expression \eqref{subsol-form-critical} of $\omega^*$ that
$$
\sup_{(t,x)\in\mathbb{R}\times\mathbb{R}^N}\omega^*(t,x)\ll\widehat{\eta}\bm{1}.
$$
By \eqref{sharp-regularity-hypothesis}, one also has
\begin{equation*}
  \left|{f(x,\omega^*(t,x))-D_uf(x,\bm{0})\omega^*(t,x)}\right|_\infty
  \leq M|\omega^*(t,x)|_\infty^{1+\beta}~~\text{ for all }(t,x)\in\Omega^\prime.
\end{equation*}
On the other hand, due to \eqref{K-choice}, all points $(t,x)\in\mathbb{R}\times\mathbb{R}^N$ contained in $\Omega_+^*$ necessarily satisfy $x\cdot e-c^*t>0$ and thus $\Omega_+^*\subset\Omega^\prime$,
where $\Omega_+^*$ and $\Omega^\prime$ are defined by \eqref{Omega+*} and \eqref{Omega-prime-def}, respectively.
Note furthermore that the constant $r_\delta^*$, given in \eqref{r*-delta-def}, is negative.
Consequently, to prove \eqref{ineq-proof} it suffices to show that
$$
\forall(t,x)\in\Omega^\prime,~~~\kappa_1r_\delta^*v^*(t,x;\lambda^*+\delta)+
M\bm{1}|\omega^*(t,x)|_\infty^{1+\beta}\leq\bm{0}.
$$

To that aim, we distinguish between two cases. If $x\cdot e-c^*t\geq s_0$, it follows from $\bm{0}\ll\phi_{\lambda^*+\delta}\leq\bm{1}$, \eqref{r*-delta-def}, \eqref{omega*-expression}-\eqref{w-estimate-leading-edge} and $\kappa_2\geq\kappa_1>0$ that
\begin{align*}
  &\kappa_1r_\delta^*v^*(t,x;\lambda^*+\delta)
  +M\bm{1}|\omega^*(t,x)|_\infty^{1+\beta}\\
&\leq\kappa_1r^*_\delta e^{-(\lambda^*+\delta)(x\cdot e-c^*t)}\bm{1}+
M\bm{1}\left[2\kappa_2\times 2(x\cdot e-c^*t)
e^{-\lambda^*(x\cdot e-c^*t)}\right]^{1+\beta}\\
&=\left[\kappa_1r^*_\delta e^{-\delta(x\cdot e-c^*t)}
+M\kappa_2^{1+\beta}4^{1+\beta}(x\cdot e-c^*t)^{(1+\beta)}
e^{-\beta\lambda^*(x\cdot e-c^*t)}\right]\bm{1}e^{-\lambda^*(x\cdot e-c^*t)},
\end{align*}
where $s_0$ is given in \eqref{w-estimate-leading-edge}.
Since $0<\delta<\beta\lambda^*$ and $r^{*}_\delta<0$ by \eqref{r*-delta-def}, then there exists a constant $C_1=C_1(s_0)>0$ such that
\begin{equation*}
\begin{split}
0<&(-{r^{*}_\delta}^{-1})4^{1+\beta}M(x\cdot e-c^*t)^{(1+\beta)}
e^{\delta(x\cdot e-c^*t)}\leq C_1e^{\beta\lambda^*(x\cdot e-c^*t)}\\
&\text{for all $(t,x)\in\mathbb{R}\times\mathbb{R}^N$ satisfying $x\cdot e-c^*t\geq s_0$},
\end{split}
\end{equation*}
whence there holds
\begin{align*}
  \kappa_1r^{*}_\delta v^*(t,x;\lambda^*+\delta)
  +M\bm{1}|\omega^*(t,x)|_\infty^{1+\beta}
\leq r^{*}_\delta\big(\kappa_1-C_1\kappa_2^{1+\beta}\big)
\bm{1}e^{-(\lambda^*+\delta)(x\cdot e-c^*t)}.
\end{align*}
If $0\leq x\cdot e-c^*t<s_0$, it follows from $\bm{0}\ll\phi_{\lambda^*+\delta}\leq\bm{1}$, $\kappa_2\geq\kappa_1>0$, \eqref{r*-delta-def} and  \eqref{omega*-upper-bound} that
\begin{align*}
  &\kappa_1r^{*}_\delta v^*(t,x;\lambda^*+\delta)
  +M\bm{1}|\omega^*(t,x)|_\infty^{1+\beta}\\
&\leq\kappa_1r^*_\delta e^{-(\lambda^*+\delta)s_0}\bm{1}+
M\bm{1}\left[\kappa_2\big(1+\overline{W}\big)\right]^{1+\beta}\\
&\leq r^*_\delta\big(\kappa_1-C_2\kappa_2^{1+\beta}\big)
e^{-(\lambda^*+\delta)s_0}\bm{1},
\end{align*}
where one has set
$$
C_2:=(-{r^{*}_\delta}^{-1})M(1+\overline{W})^{1+\beta}
e^{(\lambda^*+\delta)s_0}>0.
$$
In both cases, \eqref{ineq-proof} then follows.

Finally, to conclude we set $C(s_0):=\max(C_1,C_2)$ and fix $\kappa_1$, $\kappa_2$ such that
$$
0<C\kappa_2^{1+\beta}\leq\kappa_1\leq\kappa_2\leq\frac{\sigma}{1+\overline{W}}.
$$
Under such a choice and with the parameter $K$ satisfying \eqref{K-choice}, the function $\omega^*$ given by \eqref{subsol-form-critical} then satisfies all the properties stated in Lemma \ref{subsol-construction-critical} and thereby the proof is done.
\end{proof}
\begin{lemma}\label{supersol-construction-critical}
   With the notations of Lemma \ref{subsol-construction-critical}, let $h^*=(h^*_1,\ldots,h^*_d)^T$ be the vector-valued function given by
  \begin{equation}\label{supersol-form-critical}
  h^*(t,x)=
  \begin{dcases}
  Mv^*(t,x;\lambda^*)
    -\kappa_2\frac{\partial v^*(t,x;\lambda)}{\partial\lambda }
    \bigg|_{\lambda=\lambda^*}
  &\text{if }x\cdot e-c^*t>0, \\
    \widehat{\eta}\bm{1} & \text{if }
    x\cdot e-c^*t\leq0,
  \end{dcases}
\end{equation}
Then there exists a constant $M>0$ large enough such that
$$\forall(t,x)\in\mathbb{R}\times\mathbb{R}^N,~~~
h^*(t,x)\gg\omega^*(t,x),~~h^*(t,x)\gg\bm{0},
$$
and $h^*\to\bm{0}$ uniformly as $x\cdot e-c^*t\to+\infty$.
Moreover, $h^*$ is a classical supersolution of \eqref{monotone-system} whenever $x\cdot e-c^*t>0$.
\end{lemma}
\begin{proof}
Due to $\phi_{\lambda^*}\gg\bm{0}$ and $\phi_{\lambda^*+\delta}\gg\bm{0}$, under the fixed parameters in \eqref{subsol-form-critical}, one can choose $M_1>0$ such that
$$
(M_1+\kappa_2K)\min_{1\leq i\leq d}\min_{x\in\mathbb{T}^N}
\left(\phi_{\lambda^*}\right)_i(x)>\kappa_1
\max_{1\leq i\leq d}\max_{x\in\mathbb{T}^N}(\phi_{\lambda^*+\delta})_i(x).
$$
By Lemma \ref{subsol-construction-critical}, one has $|\omega^*(t,x)|_\infty<\widehat{\eta}$ for all $(t,x)\in\mathbb{R}\times\mathbb{R}^N$ where $\omega^*$ is given by \eqref{subsol-form-critical}.
Thus, for any $M>M_1$, one can easily check that $h^*(t,x)\gg\omega^*(t,x)$ for all $(t,x)\in\mathbb{R}\times\mathbb{R}^N$.
Thanks to \eqref{leading-term}, one also has
$h^*(t,x)\to\bm{0}$ uniformly as $x\cdot e-c^*t\to+\infty$.

Now, using \eqref{v-star-equation} with $\lambda=\lambda^*$ and   \eqref{star-v-star-equation}, we obtain that
$$
\partial_t h^*+Lh^*=D_uf(x,\bm{0})h^*.
$$
It then follows from \eqref{sublinearity} that $h^*(t,x)$ is a supersolution of \eqref{monotone-system} for all $(t,x)\in\mathbb{R}\times\mathbb{R}^N$ satisfying $x\cdot e-c^*t>0$ as soon as $h^*\geq\bm{0}$.

It remains to fix a suitable $M>0$ such that $h^*\gg\bm{0}$.
As in the proof of Lemma \ref{subsol-construction-critical}, we distinguish between two cases.
If $x\cdot e-c^*t\geq s_0$ ($s_0>0$ is as in Lemma \ref{subsol-construction-critical}), then the estimate \eqref{leading-term} applied at $\lambda=\lambda^*$ yields $h^*\gg\bm{0}$ for any $M>0$.
Now, recall from $\eqref{v-derivative-lambda}$ that
$$
h^*(t,x)=e^{-\lambda^*(x\cdot e-c^*t)}
\left[M\phi_{\lambda^*}(x)+\kappa_2(x\cdot e-c^*t)
-\kappa_2\left(\frac{\partial\phi_{\lambda}(x)}{\partial\lambda}
\bigg|_{\lambda=\lambda^*}\right)\right].
$$
where $\partial_\lambda\phi_{\lambda}|_{\lambda=\lambda^*}$ is bounded and $\phi_{\lambda^*}\gg\bm{0}$.
Thus, if $0<x\cdot e-c^*t<s_0$, then there exists $M_2>0$ large enough such that
$$
M_2\min_{1\leq i\leq d}\min_{x\in\mathbb{T}^N}\left(\phi_{\lambda^*}\right)_i(x)
+\kappa_2(x\cdot e-c^*t)\geq\kappa_2
\max_{1\leq i\leq d}\max_{x\in\mathbb{T}^N}
\left(\frac{\partial\phi_{\lambda}}{\partial\lambda}
\bigg|_{\lambda=\lambda^*}\right)_i(x)
$$
for all $(t,x)\in\mathbb{R}\times\mathbb{R}^N$ satisfying $0<x\cdot e-c^*t<s_0$.
To conclude we fix some $M>\max(M_1,M_2)$ and then Lemma \ref{supersol-construction-critical} follows.
\end{proof}

\subsection{Description of a direction-dependent coordinate system}
\label{subsect:new-coordinate}
Here we introduce a coordinate system that depends on the direction of propagation for pulsating waves.

Let $e\in\mathbb{S}^{N-1}$ be given and denote by $\{e^1,\ldots,e^N\}$ the canonical basis of $\mathbb R^N$.
Consider a linear orthogonal transformation $R\in\mathcal{O}(\mathbb{R}^{N})$ such that
\begin{equation}\label{orthogonal-transf}
Re^1=e~\text{ and }~
e^\perp={\rm span}\{Re^i\}_2^N=
\left\{x\in\mathbb{R}^{N}\mid
\forall y\in\mathbb{R}^{N-1},~x=R(0,y)^{T}\right\}.
\end{equation}
For any $k\in\mathbb{Z}^{N}$, we write $k=(k\cdot e)e+k_\perp$ with
$k_\perp:=k-(k\cdot e)e\in e^\perp$.
By \eqref{orthogonal-transf}, we then have
\begin{equation}\label{R-property}
\begin{split}
&R^{-1}k=
\begin{pmatrix}
k\cdot e\\
\bm{0}
\end{pmatrix}
+R^{-1}k_\perp
~~\text{ and}\\
&R^{-1}k_\perp=
\begin{pmatrix}
 0\\
 \widehat{R}k_\perp
\end{pmatrix}
\in \bigoplus_{i=2}^N \mathbb R e^i=\{0\}\times\mathbb R^{N-1}
\end{split}
\end{equation}
for a linear map $\widehat{R}:e^\perp\to\mathbb{R}^{N-1}$.
Let $u(t,x)$ be a pulsating travelling wave of \eqref{monotone-system} propagating in the direction $e$ with speed $c\neq0$.
We define a function $\tilde{u}=(\tilde{u}_1,\ldots,\tilde{u}_d)^T(t,r,y)$ for
$(t,r,y)\in\mathbb{R}\times\mathbb{R}\times\mathbb{R}^{N-1}$ as
\begin{equation}\label{directional-var}
\tilde{u}(t,r,y)=u(t,R(r,y)^T)~\text{ with }r=x\cdot e.
\end{equation}
From \eqref{pulsating-condition} and \eqref{R-property}, we obtain
\begin{align*}
\tilde{u}\left(t+\frac{k\cdot e}{c},r,y\right)
&=u\left(t, R(r,y)^T-k\right)\\
&=u\left(t, R(r-k\cdot e,r)^T-R^{-1}k_\perp\right)\\
&=\tilde{u}\left(t,r-k\cdot e,y-\widehat{R}k_\perp\right).
\end{align*}
Furthermore, $\tilde{u}$ satisfies the system of parabolic equations
\begin{align}\label{new-var-system}
\partial_{t}\tilde{u}+\mathcal{L}\tilde{u}=\tilde{f}(r,y,\tilde{u}),
~~~(t,r,y)\in\mathbb{R}\times\mathbb{R}\times\mathbb{R}^{N-1},
\end{align}
where the nonlinearity $\tilde{f}=(\tilde{f}_1,\ldots,\tilde{f}_d)^T: \mathbb{R}\times\mathbb{R}^{N-1}\times\mathbb{R}^d\to\mathbb{R}^d$ is defined by
$$
\tilde{f}(r,y,\tilde{u})=f(R(r,y)^T,\tilde{u}),
$$
and $\mathcal{L}:={\rm diag}(\mathcal{L}^1,\cdots,\mathcal{L}^d)$ denotes a diagonal matrix of operators with $\mathcal{L}^i$ being a second-order elliptic operator of the form
\begin{equation}\label{nondivergence-operator-new-var}
\mathcal{L}^i(r,y)\tilde{u}_i:=-\tr\big(\widetilde{A}^i(r,y)
D^2_{(r,y)}\tilde{u}_i\big)
+\widetilde{q}^{\,i}(r,y)\cdot\nabla_{(r,y)}\tilde{u}_i,
~~i=1,\ldots,d.
\end{equation}
Here the matrix-valued function $\widetilde{A}^i:\mathbb{R}\times\mathbb{R}^{N-1}
\to\mathcal{S}_{N}(\mathbb{R})$ is given by
$$
\widetilde{A}^i(r,y)=RA^i(R(r,y)^T)R^T,
~~i=1,\ldots,d,
$$
and the vector-valued function $\widetilde{q}^{\,i}:\mathbb{R}\times\mathbb{R}^{N-1}\to\mathbb{R}^N$ is given by
$$
\widetilde{q}^{\,i}(r,y)=q^i(R(r,y)^T)R,
~~i=1,\ldots,d.
$$

As a conclusion, for each given $e\in\mathbb{S}^{N-1}$, the pulsating condition \eqref{pulsating-condition} for the wave $(c,u)$ of \eqref{monotone-system} can be rewritten as the following condition for the function $\tilde{u}$ given by \eqref{directional-var}:
\begin{equation}\label{R-pulsating-condition}
\begin{split}
 &\forall k=(k\cdot e)e+k_\perp\in \mathbb Z^N,~
  \forall(t,r,y)\in\mathbb R\times\mathbb{R}\times\mathbb{R}^{N-1},\\
  &\qquad\tilde{u}\left(t+\frac{k\cdot e}{c},r,y\right)=
  \tilde{u}\left(t,r-k\cdot e,y-\widehat{R}k_\perp\right),
  \end{split}
\end{equation}
and the limiting condition \eqref{limiting-condition} is now equivalent to
\begin{equation}\label{new-var-limiting-condition}
  \liminf_{r\to-\infty}\tilde{u}(t,r,y)\gg\bm{0}~~\text{ and }~~
  \lim_{r\to+\infty}\tilde{u}(t,r,y)=\bm{0}
\end{equation}
locally uniformly for $t\in\mathbb{R}$ and uniformly for $y\in\mathbb{R}^{N-1}$,
in which $\tilde{u}$ becomes an entire solution to the parabolic system \eqref{new-var-system}.
As in \cite{deng2023existence}, we call \eqref{R-pulsating-condition} the \emph{$R$-pulsating condition} in the sequel.

Since $f(\cdot,u)$ and all coefficients of \eqref{divergence} are $\mathbb{Z}^N$-periodic,
it is straightforward to check from \eqref{orthogonal-transf}-\eqref{R-property} that $\tilde{f}(\cdot,\cdot,\tilde{u})$ and all coefficients of the operator $\mathcal{L}$ involved in \eqref{new-var-system} satisfy
\begin{equation}\label{new-coef-pulsating}
\begin{split}
&\forall k=(k\cdot e)e+k_\perp\in \mathbb Z^N,
~\forall(r,y)\in\mathbb{R}\times\mathbb{R}^{N-1},\\
&\qquad\tilde{a}\left(r+k\cdot e,y+\widehat{R}k_\perp\right)=\tilde{a}(r,y),
\end{split}
\end{equation}
where $\tilde{a}\equiv\widetilde{A}^i(r,y),\,\widetilde{q}^{\,i}(r,y),\,\tilde{f}(r,y,\cdot)$
and one has set $\tilde{a}(r,y)=a(x)$ with $x=R(r,y)^T$.
\subsection{Construction of the wave in each rational direction}
Under the framework of section \ref{subsect:new-coordinate}, we prove in this part the existence of pulsating waves for \eqref{monotone-system}
propagating in the direction of any unit vector with rational coordinates.

In this and the next subsection we always assume $\lambda_1<0$, where $\lambda_1$ is defined by \eqref{generalized-PE}, from which the formulae \eqref{maximum-periodic-PE} and \eqref{formula-minimal-speed} imply that
$$
c^*(e)>0~\text{ for all }e\in\mathbb{S}^{N-1}.
$$
We also require a key lemma, which has been proved in \cite{deng2023existence}.
\begin{lemma}\label{change-basis}
Let $\zeta\in\mathbb{Q}^N\cap\mathbb{S}^{N-1}$ for some $N\geq2$ be given.
Then one can find an orthogonal basis $\{\zeta^1,\ldots,\zeta^N\}$ of $\mathbb{R}^N$ and $N$ constants $\tau_i>0$ such that
\begin{equation}\label{tau1}
  \zeta^1=\zeta,~~~\tau_1=\min_{k\cdot\zeta>0,\,k\in\mathbb{Z}^N}k\cdot\zeta,
\end{equation}
and
  $$
  \mathbb{Z}^N=\bigoplus_{i=1}^N\tau_i\mathbb{Z}\zeta^i,
  ~~~\zeta^{i}\perp\zeta^{j}~(i\neq j),~~i,j=1,\ldots,N.
  $$
\end{lemma}

Now, let $\zeta\in\mathbb{Q}^N\cap\mathbb{S}^{N-1}$ be given.
For any $k\in\mathbb{Z}^N$, we decompose $k$ along $\zeta$ and $\zeta^\perp$, that is,
$k=(k\cdot\zeta)\zeta+k_\perp$ with $k_\perp\in\zeta^\perp$.
By Lemma \ref{change-basis}, we have
$$
k\cdot\zeta=\tau_{1}p_1,~~~
k_\perp=\sum_{i=2}^{N}\tau_ip_i\zeta^i,
$$
where $p_i\in\mathbb{Z}$ for all $1\leq i\leq N$.
The $R$-pulsating condition \eqref{R-pulsating-condition} then reduces to
\begin{equation}\label{rational-direction-pulsating}
  \begin{dcases}
  \forall(t,r)\in\mathbb{R}^{2},~\forall p_1\in\mathbb{Z},~~
  \tilde{u}\left(t+\frac{p_1\tau_1}{c},r,y\right)=\tilde{u}(t,r-p_1\tau_1,y),
  \\
  \forall y\in\mathbb{R}^{N-1},~\forall l\in\tau_2\mathbb{Z}\times\cdots
  \times\tau_N\mathbb{Z},~~\tilde{u}(t,r,y+l)=\tilde{u}(t,r,y),
  \end{dcases}
\end{equation}
where $r=x\cdot\zeta$ by \eqref{directional-var}.
Moreover,
the property \eqref{new-coef-pulsating} satisfied by $\tilde{f}$ and all coefficients involved in \eqref{nondivergence-operator-new-var} reduces to the following periodicity:
\begin{equation}\label{coef-periodicity-rational}
\begin{split}
&\forall(r,y)\in\mathbb{R}\times\mathbb{R}^{N-1},
~\forall l\in\tau_2\mathbb{Z}\times\cdots\times\tau_N\mathbb{Z},\\
&\qquad \tilde{a}(r+\tau_1,y)=\tilde{a}(r,y+l)=\tilde{a}(r,y).
\end{split}
\end{equation}
For notational simplicity, set $\tau^\prime:=(\tau_2,\ldots,\tau_N)$ and define an $(N-1)$-dimensional torus as
$$
\mathbb{T}^{N-1}_{\tau^\prime}=\mathbb{R}/\tau_2\mathbb{Z}\times\cdots
\times\mathbb{R}/\tau_N\mathbb{Z}.
$$
We then conclude from \eqref{new-var-system}-\eqref{new-coef-pulsating} that a pulsating travelling wave of \eqref{monotone-system}  propagating along the direction $\zeta\in\mathbb{Q}^{N}\cap\mathbb{S}^{N-1}$ with speed $c\neq0$ is a solution to the problem
\begin{equation}\label{rational-direction-problem}
\begin{dcases}
\partial_{t}\tilde{u}+\mathcal{L}\tilde{u}=\tilde{f}(r,y,\tilde{u})~~~~~\text{ for }
(t,r,y)\in\mathbb{R}\times\mathbb{R}\times\mathbb{T}^{N-1}_{\tau^\prime},\\
\forall(t,r)\in\mathbb{R}^2,~y\in\mathbb{T}^{N-1}_{\tau^\prime},~~
\tilde{u}\left(t+\frac{\tau_1}{c},r,y\right)=\tilde{u}(t,r-\tau_1,y),
\\\bm{0}\ll\tilde{u}\leq\widehat{\eta}\bm{1},~
\lim_{r\to+\infty}\tilde{u}(t,r,y)=\bm{0},~~
  \liminf_{r\to-\infty}\tilde{u}(t,r,y)\gg\bm{0},
\end{dcases}
\end{equation}
where the convergence holds locally uniformly for $t\in\mathbb{R}$ and uniformly for
$y\in\mathbb{T}^{N-1}_{\tau^\prime}$.

To address \eqref{rational-direction-problem}, we follow and extend some ideas taken from    \cite{griette2021propagation}.
Roughly speaking, we first consider a similar problem posed
in a semi-infinite periodic cylinder with a boundary moving at a velocity $c>0$ equipped with homogeneous Neumann data.
This approximated problem will be solved via a fixed-point argument.
Afterwards, we pass to the limit to find a one-dimensional front propagation inside the whole cylinder.

The main result of this subsection reads as follows.
\begin{proposition}\label{rational-direction-existence}
Let $\zeta\in\mathbb{Q}^N\cap\mathbb{S}^{N-1}$ be given and assume \eqref{sublinearity} holds.
Then for every $c\geq c^*(\zeta)$, the problem \eqref{rational-direction-problem} has a classical solution $\tilde{u}_c$.
\end{proposition}
\begin{proof}
The proof is divided into three steps.
We shall treat two cases $c>c^*(\zeta)$ and $c=c^*(\zeta)$ in parallel,
rather than using a usual limiting argument to deal with the latter.
In addition, the uniform positivity of solutions as $r\to-\infty$ will be directly derived from a novel lower barrier.
\begin{step}
Some preliminaries.
\end{step}
For any $c>c^*(\zeta)$, we define for all $(t,r,y)\in\mathbb{R}\times\mathbb{R}\times\mathbb{R}^{N-1}$:
\begin{align}
\tilde{h}(t,r,y)&:=h(t,R(r,y)^T)=e^{-\lambda_c(r-ct)}\phi_{\lambda_c}(R(r,y)^T),\label{h-tilde-def}\\
\tilde{\omega}(t,r,y)&:=\omega(t,R(r,y)^T)=h(t,R(r,y)^T)
-Ke^{-(\lambda_c+\delta)(r-ct)}\phi_{\lambda_c+\delta}(R(r,y)^T),\label{omega-tilde-def}
\end{align}
where $h$ and $\omega$ are defined by \eqref{supersol-form-any-speed} and \eqref{subsol-form-larger-speeds} (recall that $x=R(r,y)^T$ under the change of variables \eqref{directional-var}).
Similarly, for $c=c^*(\zeta)$, we define
$$
\tilde{h}^*(t,r,y):=h^*(t,R(r,y)^T),~~~
\tilde{\omega}^*(t,r,y):=\omega^*(t,R(r,y)^T),
$$
where $h^*$ and $\omega^*$ are defined by \eqref{supersol-form-critical} and \eqref{subsol-form-critical}.
It is clear to see that $\tilde{h}$, $\tilde{\omega}$, $\tilde{h}^*$ and $\tilde{\omega}^*$ satisfy the same pulsating properties of $(t,r)$ and periodicity of $y$ as in \eqref{rational-direction-pulsating}.
Observe also that
\begin{equation*}
\begin{split}
&\forall n\in\mathbb{N}^*,~~~\forall(t,r,y)\in\mathbb{R}\times\mathbb{R}\times
\mathbb{T}_{\tau^\prime}^{N-1},\\
\tilde{\omega}(t,r+n\tau_1,&y)\ll\tilde{h}(t,r+n\tau_1,y)=
e^{-n\tau_1\lambda_c}\tilde{h}(t,r,y)\ll\tilde{h}(t,r,y),\\
\end{split}
\end{equation*}
where equality holds since the principal eigenfunction $\phi_{\lambda}(R(r,y)^T)$ in the direction $\zeta\in\mathbb{Q}^N\cap\mathbb{S}^{N-1}$ satisfies the same periodicity of $(r,y)$ as in \eqref{coef-periodicity-rational}.
As for the case $c=c^*(e)$, from \eqref{v-star}, \eqref{v-derivative-lambda} and \eqref{supersol-form-critical}, we find that
$$
h^*(t,R(r+n\tau_1,y)^T)=e^{-n\tau_1\lambda^*}h^*(t,R(r,y)^T)
+\kappa_2n\tau_1e^{-\lambda^*(r+n\tau_1-c^*(\zeta)t)}\phi_{\lambda^*}(R(r,y)^T).
$$
Thus, if $r+n\tau_1-c^*(\zeta)t>0$, one can choose a constant $\widehat{\gamma}=\widehat{\gamma}_{\kappa_2,\tau_1}>0$ such that
$$
\forall n\in\mathbb{N}^*,~\forall(t,r,y)\in\mathbb{R}^2\times
\mathbb{T}_{\tau^\prime}^{N-1},~~
\tilde{h}^*(t,r+n\tau_1,y)\leq\widehat{\gamma}\tilde{h}^*(t,r,y).
$$
By Lemma \ref{supersol-construction-critical}, one also has
\begin{equation*}
\begin{split}
\forall n\in\mathbb{N}^*,~\forall(t,r,y)\in\mathbb{R}^2\times
\mathbb{T}_{\tau^\prime}^{N-1},~~
\tilde{\omega}^*(t,r+n\tau_1,y)\ll\widehat{\gamma}\tilde{h}^*(t,r,y).
\end{split}
\end{equation*}
Consider now the following functions defined for $(t,r,y)\in\mathbb{R}\times\mathbb{R}\times\mathbb{T}_{\tau^\prime}^{N-1}$ as
\begin{align}
&\overline{\tilde{u}}_c(t,r,y)=\min\big(\widehat{\eta}\bm{1},\tilde{h}(t,r,y)\big),
~~~~~~\underline{\tilde{u}}_c(t,r,y)=\max\left(\bm{0},
\sup_{n\in\mathbb{N}}\tilde{\omega}(t,r+n\tau_1,y)\right),\label{bound-rational-c}\\
&\overline{\tilde{u}}_{c^*}(t,r,y)=\min\big(\widehat{\eta}\bm{1},\widehat{\gamma}
\tilde{h}^*(t,r,y)\big),~~
\underline{\tilde{u}}_{c^*}(t,r,y)=\max\left(\bm{0},
\sup_{n\in\mathbb{N}}\tilde{\omega}^*(t,r+n\tau_1,y)\right).\label{bound-rational-c*}
\end{align}
Then, $\overline{\tilde{u}}_c\gg\underline{\tilde{u}}_c\geq\bm{0}$ and $\overline{\tilde{u}}_{c^*}\gg\underline{\tilde{u}}_{c^*}\geq\bm{0}$.
Moreover, $\underline{\tilde{u}}_c$ and $\underline{\tilde{u}}_{c^*}$ are uniformly positive (that is, bounded from below by a positive constant) as $r\to-\infty$.
\begin{step}
An approximated problem posed in semi-infinite periodic cylinders.
\end{step}
Denote by $\lfloor x\rfloor$ the largest integer below $x$ and fix $a<0$ such that
$$
\forall t\geq0,~r\leq\lfloor{-a}/{\tau_1}\rfloor+ct,~y\in\mathbb{T}_{\tau^\prime}^{N-1},
~~\min_{1\leq i\leq d}\tilde{h}_i(t,r,y)\geq\widehat{\eta}
~~\text{ if $c>c^*(\zeta)$}.
$$
Note that by Lemma \ref{supersol-construction-critical}, the above holds for any $a<0$ when $c=c^*(\zeta)$.
Define
\begin{align}
&\underline{\tilde{u}}_c^a(t,r,y)=\max\left(\bm{0},
\max_{0\leq n\leq\lfloor-a/\tau_1\rfloor}\tilde{\omega}(t,r+n\tau_1,y)\right),
\label{lower-barrier-finite-c}\\
&\underline{\tilde{u}}_{c^*}^a(t,r,y)=\max\left(\bm{0},
\max_{0\leq n\leq\lfloor-a/\tau_1\rfloor}\tilde{\omega}^*(t,r+n\tau_1,y)\right).
\label{lower-barrier-finite-c*}
\end{align}
For all $c\geq c^*(\zeta)$, it is clear that $\underline{\tilde{u}}_c^a\to\underline{\tilde{u}}_c$ as $a\to-\infty$ where $\underline{\tilde{u}}_c$ and $\underline{\tilde{u}}_{c^*}$ are defined in \eqref{bound-rational-c} and \eqref{bound-rational-c*}.
For brevity we drop the subscripts $c$ and $c^*$ in \eqref{bound-rational-c}-\eqref{lower-barrier-finite-c*} when $t=0$.

Set $\Sigma^+_a:=(a,\infty)\times\mathbb{T}_{\tau^\prime}^{N-1}$.
From the above construction and Lemmas \ref{subsol-construction-larger-speeds}-\ref{subsol-construction-critical},
we find that $\overline{\tilde{u}}(0,r,y)\gg\underline{\tilde{u}}^a(0,r,y)$, and that $\overline{\tilde{u}}(0,r,y)$ and $\underline{\tilde{u}}^a(0,r,y)$ are identically equal to $\widehat{\eta}\bm{1}$ and $\bm{0}$ near the bottom of $\Sigma^+_a$, respectively.
Let $E_a$ be the set defined by
$$
E_a=\left\{\varphi\in BUC
(\overline{\Sigma^+_a})^d\mid\underline{\tilde{u}}^a(0,r,y)\leq
\varphi(r,y)\leq\overline{\tilde{u}}(0,r,y)\right\},
$$
which is convex and compactly included in $BUC(\overline{\Sigma^+_a})^d$.
We now equip $E_a$ with the uniform convergence topology and
set up a fixed-point problem for an approximation of the solution to \eqref{rational-direction-problem} in $E_a$.
To do so, fix $c\geq c^*(\zeta)$ and define a map as
\begin{equation}\label{Q-def}
\begin{split}
Q^a\colon E_a&\to BUC(\overline{\Sigma^+_a})^d,\\
\varphi&\mapsto\tilde{u}\left(\frac{\tau_1}{c},\cdot+\tau_1,\cdot\right),
\end{split}
\end{equation}
where $\tilde{u}$ is the solution to the initial-boundary value problem
\begin{equation}\label{fixed-point-problem}
  \begin{dcases}
  \partial_{t}\tilde{u}+\mathcal{L}\tilde{u}=\tilde{f}(r,y,\tilde{u})
   &\text{for }(t,r,y)\in(0,\infty)\times(a+ct,\infty)
   \times\mathbb{T}^{N-1}_{\tau^\prime},\\
   \partial_{r}\tilde{u}(t,a+ct,y)=\bm{0}
  &\text{for }(t,y)\in(0,\infty)\times
  \mathbb{T}^{N-1}_{\tau^\prime},\\
  \tilde{u}(0,r,y)=\varphi(r,y)
  &\text{for }(r,y)\in\overline{\Sigma^+_a}.
  \end{dcases}
\end{equation}
\noindent
It is known that for any given $\varphi\in BUC(\overline{\Sigma^+_a})^d$, problem \eqref{fixed-point-problem} admits a unique solution $\tilde{u}(t;\varphi)\in BUC([a+ct,\infty)\times\mathbb{T}^{N-1}_{\tau^\prime})^d$ defined on $\mathbb{R}_+$,
and it further becomes smooth as $t$ runs.

Our next aim is to check that the set $E_a$ is stable under $Q^a$,
and that the map $Q^a$ is compact from $E_a$ into itself.
The former follows from the following claim.
\begin{claim}\label{invariance}
  The solution $\tilde{u}=\tilde{u}(t,r,y)$ of \eqref{fixed-point-problem} satisfies
  $$
  \underline{\tilde{u}}^a_c(t,r,y)\leq\tilde{u}(t,r,y)\leq\overline{\tilde{u}}_c(t,r,y)
  ~~\text{ in }\mathbb{R}_+\times[a+ct,\infty)\times\mathbb{T}^{N-1}_{\tau^\prime},
  $$
where $\overline{\tilde{u}}_c$ and $\underline{\tilde{u}}^a_c$ for all $c\geq c^*(\zeta)$ are defined in \eqref{bound-rational-c}-\eqref{lower-barrier-finite-c*}, respectively.
\end{claim}

Before proving this claim let us proceed with our argument.
For any given $\varphi\in E_a$, we infer from Claim \ref{invariance} and the pulsating properties of $\underline{\tilde{u}}^a_c$ and $\overline{\tilde{u}}_c$ that the map $Q^a$, defined by \eqref{Q-def}, satisfies
$$
\underline{\tilde{u}}^a(0,r,y)=
\underline{\tilde{u}}^a_c\left(\frac{\tau_1}{c},r+\tau_1,y\right)\leq
Q^a(\varphi)(r,y)
\leq\overline{\tilde{u}}_c\left(\frac{\tau_1}{c},r+\tau_1,y\right)
=\overline{\tilde{u}}(0,r,y)
$$
for all $(r,y)\in\overline{\Sigma^+_a}$, that is, $Q^a(\varphi)\in E_a$.
The latter compactness property is standard, owing to the smoothing effect of the uniform parabolic operators.
For completeness we include its proof.
Take $\{\varphi^n\}_{n\in\mathbb{N}}\subset E_a$ and set $\tilde{u}^n:=Q^a(\varphi^n)$.
Due to $\tau_1/c>0$, the classical parabolic estimates yield that for every fixed $M>0$, $\{\tilde{u}^n(\tau_1/c,\cdot+\tau_1,\cdot)\}_{n\in\mathbb{N}}$ is uniformly bounded in $C^{2+\alpha}([a,M]\times\mathbb{T}^{N-1}_{\tau^\prime})^d$.
Hence one may extract (using a diagonal process) a subsequence of $\{\tilde{u}^n\}$, not relabelled, which converges in
$C^2_{\rm loc}([a,\infty)\times\mathbb{T}^{N-1}_{\tau^\prime})^d$ to a vector-valued function $\tilde{u}$ satisfying \eqref{Q-def} and remaining in $E_a$ due to Claim \ref{invariance}.
Furthermore, the convergence is uniform for $r\in[a,\infty)$.
Indeed, since $\overline{\tilde{u}}(0,r,\cdot)\to\bm{0}$ as $r\to+\infty$, by continuity, for any given $\epsilon>0$ and $M$ sufficiently large, there exists
$n_0\in\mathbb{N}$ independent of $M$ such that
$$
\forall n\geq n_0,~\forall r\geq M,~\forall y\in\mathbb{T}^{N-1}_{\tau^\prime},
~~|\tilde{u}^n(\tau_1/c,r+\tau_1,y)|_\infty\leq\epsilon.
$$
Thus, the map $Q^a$ is compact from $E_a$ into itself.

As a consequence, the set $\overline{Q^a(E_a)}$ is compact in $E_a$.
We then conclude from the Schauder fixed-point theorem that $Q^a$ has a fixed point in $E_a$, denoted as $\tilde{u}^a$ satisfying
$$
Q^a(\tilde{u}^a)=\tilde{u}^a\left(\frac{\tau_1}{c},\cdot+\tau_1,\cdot\right)\in E_a.
$$
\begin{step}
Passage to the limit in the infinite cylinders.
\end{step}
Take any sequence $\{a_n\}_{n\in\mathbb{N}}\subset(-\infty,0)$ such that $a_n\to-\infty$ as $n\to\infty$.
Let
$$
Q^{a_n}(\tilde{u}^{a_n})(r,y)=\tilde{u}^{a_n}\left(\frac{\tau_1}{c},r+\tau_1,y\right)
~\text{ for }(r,y)\in\overline{\Sigma^+_{a_n}}:=[a_n,\infty)\times
\mathbb{T}_{\tau^\prime}^{N-1}
$$
be the solutions to the problem \eqref{Q-def}.
Due to parabolic regularity,
$Q^{a_n}(\tilde{u}^{a_n})$ for each $n\in\mathbb{N}$ is bounded in $C^{2+\alpha}(\overline{\Sigma^+_{a_n}})^d$.
Up to extracting a subsequence,
one may assume that
$$
\tilde{u}^{a_n}\left(\frac{\tau_1}{c},r+\tau_1,\cdot\right)\to
\tilde{u}^{\infty}\left(\frac{\tau_1}{c},r+\tau_1,\cdot\right)
\text{ locally uniformly for }r\in\mathbb{R}.
$$
Moreover, the positive invariance of $E_{a_n}$ ensures that
$$
\tilde{u}^{\infty}\in E_\infty:=\left\{\varphi\in BUC
(\mathbb{R}\times\mathbb{T}_{\tau^\prime}^{N-1})^d
\mid\underline{\tilde{u}}(0,r,y)\leq
\varphi(r,y)\leq\overline{\tilde{u}}(0,r,y)\right\},
$$
whence
$\tilde{u}^{\infty}$ becomes a fixed point of the map $Q^\infty$ given by \eqref{Q-def} with $a=-\infty$, that is, $Q^\infty(\tilde{u}^{\infty})=\tilde{u}^{\infty}$.
As a consequence, for each given $c\geq c^*(\zeta)$, the problem \eqref{rational-direction-problem} admits a classical entire solution $\tilde{u}\equiv\tilde{u}^{\infty}(t,r,y)$ satisfying
$$
\begin{dcases}
\tilde{u}^{\infty}\left(t+\frac{\tau_1}{c},r,y\right)=\tilde{u}^{\infty}(t,r-\tau_1,y)\\
\underline{\tilde{u}}_c(t,r,y)\leq\tilde{u}^\infty(t,r,y)\leq\overline{\tilde{u}}_c(t,r,y)
\end{dcases}~\text{ for all }
(t,r,y)\in\mathbb{R}\times\mathbb{R}\times\mathbb{T}^{N-1}_{\tau^\prime}.
$$
Furthermore, from our construction of the lower and upper bounds in \eqref{bound-rational-c}-\eqref{bound-rational-c*},
it is straightforward to see that
$$
\lim_{r\to+\infty}\tilde{u}^\infty(t,r,y)=\bm{0}~~\text{ and }~~
  \liminf_{r\to-\infty}\tilde{u}^\infty(t,r,y)\gg\bm{0}
$$
locally uniformly in $r\in\mathbb{R}$ and uniformly in $y\in\mathbb{T}^{N-1}_{\tau^\prime}$.

Eventually, we conclude from the above three steps that for each $c\geq c^*(\zeta)$, the function $\tilde{u}_c:=\tilde{u}^\infty(t,r,y)$ is the desired solution to \eqref{rational-direction-problem}.
This ends the proof of Proposition \ref{rational-direction-existence}.
\end{proof}
\begin{proof}[Proof of Claim \ref{invariance}]
We only deal with the case $c>c^*(\zeta)$.
The case $c=c^*(\zeta)$ can be handled similarly.
Moreover, it is straightforward to derive the upper estimates for the solution $\tilde{u}$ of problem \eqref{fixed-point-problem} because $\widehat{\eta}\bm{1}$ and $\tilde{h}$ are two global supersolutions to \eqref{new-var-system}.

Let us now turn to the proof of the lower estimates.
Due to $\tilde{u}(0,r,y)\geq\underline{\tilde{u}}^a(0,r,y)\not\equiv\bm{0}$,
the strong maximum principle and Hopf lemma together with assumption \ref{item:irreducibility} yields
\begin{equation}\label{left-stable-conclusion1}
\tilde{u}(t,r,y)\gg\bm{0}~~~\text{for all $t>0$ and }
(r,y)\in[a+ct,\infty)\times\mathbb{T}^{N-1}_{\tau^\prime}.
\end{equation}
Since $\tilde{f}$ and all coefficients of $\mathcal{L}$ in \eqref{fixed-point-problem} are periodic in $(r,y)$ as stated in \eqref{coef-periodicity-rational},
we infer from Lemma \ref{subsol-construction-larger-speeds} that
\begin{align*}
\partial_{t}\tilde{u}+\mathcal{L}\tilde{u}-\tilde{f}(r,y,\tilde{u})\geq
\partial_{t}\tilde{v}^l+\mathcal{L}\tilde{v}^l-\tilde{f}(r,y,\tilde{v}^l)
~~\text{ in }\{t>0\}\cap\Omega^l_+,
\end{align*}
where one has set (remember that $r=x\cdot\zeta$ by \eqref{directional-var})
\begin{equation}\label{left-stable-v-tilde-def}
\begin{split}
&\tilde{v}^l(t,r,y):=\tilde{\omega}(t,r+l\tau_1,y)~~\text{ for each $l\in\{0,1,\ldots,\lfloor-a/\tau_1\rfloor\}$, and}\\
&\Omega^l_+:=\Big\{(t,r,y)\in\mathbb{R}\times\mathbb{R}\times\mathbb{T}^{N-1}_{\tau^\prime}
\,\big|\,\max_{1\leq i\leq d}\tilde{v}_i^l(t,r,y)>0\Big\}.
\end{split}
\end{equation}
Notice that the map $l\mapsto\tilde{v}^l$ is decreasing.
Due to $a<0$ and $l\leq\lfloor-a/\tau_1\rfloor$, one also has
\begin{equation}\label{Omega-a}
\Omega^l_+\subset\Omega^a:=\left\{(t,r,y)\in\mathbb{R}\times\mathbb{R}\times
\mathbb{T}^{N-1}_{\tau^\prime}\mid r-ct\geq a\right\},
\end{equation}
which implies that
\begin{equation}\label{left-stable-conclusion2}
\tilde{v}^l<\bm{0}~~\text{ in }\Omega^a\setminus\Omega^l_+.
\end{equation}

On the other hand, the above construction \eqref{left-stable-v-tilde-def} implies that
$\tilde{v}^l\leq\bm{0}$ on $\partial\Omega^l_+$.
Combining \eqref{left-stable-conclusion1} with \eqref{Omega-a},
it then follows that
$\tilde{u}\gg\tilde{v}^l$ on $\{t>0\}\cap\partial\Omega^l_+$.
Recalling the choice of initial data in \eqref{fixed-point-problem}, we know that for all $(r,y)\in\mathbb{R}\times\mathbb{T}^{N-1}_{\tau^\prime}$ and for each $l\in\mathbb{N}$,
$(\tilde{u}-\tilde{v}^l)(0,r,y)\geq(\underline{\tilde{u}}^a-\tilde{v}^l)(0,r,y)>\bm{0}$, where $\underline{\tilde{u}}^a$ is defined by \eqref{lower-barrier-finite-c}.
To summarize, we have
\begin{equation}\label{left-stable-problem1}
\begin{dcases}
\partial_{t}(\tilde{u}_i-\tilde{v}^l_i)+\mathcal{L}^i(\tilde{u}_i-\tilde{v}^l_i)
\geq\sum_{j=1}^{d}\tilde{h}^l_{ij}(\tilde{u}_j-\tilde{v}^l_j)
~\text{ in }\{t>0\}\cap\Omega^l_+,\\
\tilde{u}_i-\tilde{v}^l_i>0~\text{ on }\{t>0\}\cap\partial\Omega^l_+,
~~~~i=1,\ldots,d,\\
\forall(r,y)\in\mathbb{R}\times\mathbb{T}^{N-1}_{\tau^\prime},~~
(\tilde{u}_i-\tilde{v}^l_i)(0,r,y)\geq(\neq)\,0,~~~i=1,\ldots,d,
\end{dcases}
\end{equation}
where the operator $\mathcal{L}^i$ is defined by \eqref{nondivergence-operator-new-var} and one has set
$$
\tilde{h}^l_{ij}(t,r,y):=\int_{0}^{1}\frac{\partial \tilde{f}_{i}}{\partial\tilde{u}_j}
(r,y,\tilde{v}^l_1+s\tilde{u}_1,\ldots,\tilde{v}^l_d+s\tilde{u}_d)\mathrm{d}s,
~~l\in\mathbb{N},~i,j=1,\ldots,d.
$$
Note that the functions $\tilde{h}^l_{ij}$ are bounded.
Following
\cite[page 189 in Chap.3]{protter2012maximum} and
\cite[Lemma 2.2]{berestycki2008asymptotic}, we next prove that
\begin{equation}\label{left-stable-conclusion3}
\forall(t,r,y)\in\{t>0\}\cap\Omega^l_+,~~~
\tilde{u}(t,r,y)\gg\tilde{v}^l(t,r,y)~\text{ for each }l\in\mathbb{N}.
\end{equation}

To do so, fix $\epsilon>0$ and define
\begin{equation}\label{w-epsilon-def}
w^\epsilon(t,r,y)=\tilde{u}(t,r,y)-\tilde{v}^l(t,r,y)+\epsilon\bm{1}e^{Mt},~~
(t,r,y)\in\mathbb{R}_+\times\mathbb{R}\times\mathbb{T}^{N-1}_{\tau^\prime},
\end{equation}
where the constant $M>0$ is chosen such that
$$
M-\sum_{j=1}^{d}\tilde{h}^l_{ij}(t,r,y)>0~\text{ for all $i=1,\ldots,d$ and $l\in\mathbb{N}$}.
$$
From \eqref{left-stable-problem1} and \eqref{w-epsilon-def}, each component of $w^\epsilon$ then satisfies
\begin{equation}\label{w-epsilon-inequ}
\begin{dcases}
\partial_{t}w^\epsilon_i+\mathcal{L}^iw^\epsilon_i-\sum_{j=1}^{d}
\tilde{h}^l_{ij}w^\epsilon_i>0~~\text{ in }\{t>0\}\cap\Omega^l_+,\\
w^\epsilon_i\geq\epsilon~
~\text{ on }\partial(\{t>0\}\cap\Omega^l_+),~~i=1,\ldots,d.
\end{dcases}
\end{equation}
We now claim that $w^\epsilon\gg\bm{0}$ in $\{t>0\}\cap\Omega^l_+$ for any $\epsilon>0$.
To prove this we argue by contradiction.
Set $\Omega_T:=\{0<t\leq T\}\cap\Omega^l_+$ for $T>0$.
Suppose there exist $\epsilon>0$, $(t,r,y)\in\overline{\Omega_T}$ and, without loss of generality, $i_0\in\{1,\ldots,d\}$ such that $w_{i_0}^\epsilon(t,r,y)=0$.
Notice first that, by \eqref{subsol-form-larger-speeds}, we can easily check that the function $\tilde{v}^l$ defined in \eqref{left-stable-v-tilde-def} satisfies
$\max_{1\leq i\leq d}\tilde{v}_i^l(t,r,\cdot)\to-\infty$ as $r\to-\infty$ locally uniformly for $(t,l)\in\mathbb{R}\times\mathbb{Z}$,
and $\max_{1\leq i\leq d}\tilde{v}_i^l(t,r,\cdot)\to0$ as $r\to+\infty$ locally uniformly in $t$ and uniformly in $l$.
From \eqref{left-stable-conclusion1} and \eqref{w-epsilon-def}, we find that the minimum of $w_{i_0}^\epsilon$ over $\overline{\Omega_T}$ cannot be attained at $|r|=\infty$.
Consequently, there exists $(t_0,r_0,y_0)\in\overline{\Omega_T}$ such that
\begin{equation}\label{w-epsilon-minimum-point}
w_{i_0}^\epsilon(t_0,r_0,y_0)=\min_{\overline{\Omega_T}}w_{i_0}^\epsilon=0~
\text{ and }~w^\epsilon_j\geq0\text{ for }j\neq i_0.
\end{equation}
Observe from \eqref{w-epsilon-inequ} that one necessarily has $(t_0,r_0,y_0)\in\Omega_T$.
Therefore, $\nabla_{(r,y)}w^\epsilon_{i_0}(t_0,r_0,y_0)=\bm{0}$ and
$D^2_{(r,y)}w^\epsilon_{i_0}(t_0,r_0,y_0)$ is nonnegative definite.
Using the uniform ellipticity and symmetry involved in the operator $\mathcal{L}^{i_0}$, one can further deduce that
$$
\left(\mathcal{L}^{i_0}w^\epsilon_{i_0}\right)(t_0,r_0,y_0)\leq0.
$$
Combining \eqref{w-epsilon-inequ}, \eqref{w-epsilon-minimum-point}
and $\tilde{h}_{ij}\geq0$ for $i\neq j$, we obtain
$$
\partial_tw^\epsilon_{i_0}(t_0,r_0,y_0)>0.
$$
However, the definition of the minimum in \eqref{w-epsilon-minimum-point} yields that for all $t\in[0,t_0]$, if $(t,r_0,y_0)\in\Omega^l_+$, then $w^\epsilon_{i_0}(t,r_0,y_0)\geq w^\epsilon_{i_0}(t_0,r_0,y_0)$ by continuity.
Since $t_0>0$, for $\sigma>0$ sufficiently small, one has $(t_0-\sigma,r_0,y_0)\in\{t>0\}\cap\Omega^l_+$.
Therefore, it is possible to differentiate the inequality to obtain $\partial_tw^\epsilon_{i_0}(t_0,r_0,y_0)\leq0$.
This leads to a contradiction.

As a result, we obtain
$$
\forall\,T>0,~\forall\epsilon>0,~~~
w^\epsilon=(\tilde{u}-\tilde{v}^l)+\epsilon\bm{1}e^{Mt}\gg\bm{0}~
\text{ in $\Omega_T$},
$$
whence it also holds in $\{t>0\}\cap\Omega^l_+$.
Letting $\epsilon\to0$, then $\tilde{u}-\tilde{v}^l\geq\bm{0}$ in $\{t>0\}\cap\Omega^l_+$.
Furthermore, the inequality is strict, that is, $\tilde{u}\gg\tilde{v}^l$ in $\{t>0\}\cap\Omega^l_+$.
Indeed, suppose not, then there exist $(t^\prime,r^\prime,y^\prime)\in\{t\geq0\}\cap\overline{\Omega_+}$ and  $i^\prime\in\{1,\ldots,d\}$ such that
$(\tilde{u}-\tilde{v}^l)_{i^\prime}(t^\prime,r^\prime,y^\prime)=0$.
Applying the strong maximum principle and Hopf lemma
\cite[Theorem 7 in Chap.3]{protter2012maximum} to the $i^\prime$-th differential inequality of \eqref{left-stable-problem1}, one gets $(\tilde{u}-\tilde{v})_{i^\prime}=0$ in the connected component of $\{0<t\leq t^\prime\}\cap\Omega_+$ containing $(t^\prime,r^\prime,y^\prime)$.
Since the matrix $(\tilde{h}_{ij})_{1\leq i,j\leq d}$ is fully coupled,
one also has $(\tilde{u}-\tilde{v})_j=0$ for all $j\in\{1,\ldots,d\}$
over the same domain and, by continuity, up to the boundary as well.
Recalling \eqref{left-stable-conclusion1} and the definition of $\Omega_+$ in \eqref{left-stable-v-tilde-def}, a contradiction has been reached.

Finally, we conclude from \eqref{left-stable-conclusion1}-\eqref{left-stable-conclusion2} and \eqref{left-stable-conclusion3} that
$$
\forall t\geq0,~\forall(r,y)\in[a+ct,\infty)\times\mathbb{T}^{N-1}_{\tau^\prime},~~
\tilde{u}(t,r,y)\geq\max\left(\bm{0},
\max_{0\leq n\leq\lfloor-a/\tau_1\rfloor}\tilde{\omega}(t,r+n\tau_1,y)\right).
$$
Claim \ref{invariance} then follows.
\end{proof}
\subsection{Completion of the proof of Theorem \ref{main-result-existence}}
In this subsection, we return to the original coordinate system with the aim of proving the existence of pulsating waves propagating in any given direction.

Proposition \ref{rational-direction-existence} equivalently says that for each $c\geq c^*(\zeta)$, system \eqref{monotone-system} admits a pulsating travelling wave $u=u(t,x;\zeta)$ propagating in the direction $\zeta\in\mathbb{Q}^N\cap\mathbb{S}^{N-1}$ with speed $c$.
Moreover, from the proof of Proposition \ref{rational-direction-existence}, one also has,
for all $(t,x)\in\mathbb{R}\times\mathbb{R}^{N}$,
\begin{align}
 \max\left(\bm{0},\sup_{n\in\mathbb{N}}\omega(t,x+n\tau_{1}\zeta)\right)\leq u(t,x;\zeta)\leq \min\big(\widehat{\eta}\bm{1},h(t,x)\big)~~~~&\text{ for $c>c^*(\zeta)$},
 \label{estimate-rational-c}\\
 \max\left(\bm{0},\sup_{n\in\mathbb{N}}\omega^*(t,x+n\tau_{1}\zeta)\right)\leq u^*(t,x;\zeta)\leq \min\big(\widehat{\eta}\bm{1},\widehat{\gamma}h^*(t,x)\big)~&\text{ for $c=c^*(\zeta)$},
 \label{estimate-rational-c*}
\end{align}
where $h$, $\omega$, $\omega^*$ and $h^*$ are defined by \eqref{supersol-form-any-speed}, \eqref{subsol-form-larger-speeds}, \eqref{subsol-form-critical} and \eqref{supersol-form-critical},
and the parameter $\widehat{\gamma}$ in \eqref{estimate-rational-c*} is the same as in \eqref{bound-rational-c*}.
Recall that the asymptotics of the waves at infinity
were directly obtained from \eqref{estimate-rational-c}-\eqref{estimate-rational-c*}.
However, such lower bounds no longer hold for a general direction $e\in\mathbb{S}^{N-1}$ (more on this later).
Furthermore, since $f$ is not assumed to be subhomogeneous in Theorem \ref{main-result-existence}, the monotonicity results derived in Proposition \ref{monotonocity-proof} may not be valid.

On this account, we prove a persistence property for the solution of the Cauchy problem associated with the monotone system \eqref{monotone-system} before proving the existence of pulsating waves propagating in any given direction.
This kind of property is known as the \emph{hair-trigger effect}.
\begin{proposition}[Asymptotic persistence]\label{local-uniform-persistence}
Let $\psi=(\psi_1,\ldots,\psi_d)^T\in BUC(\mathbb{R}^N)^d$
be such that $\psi\geq\bm{0}$ and $\psi_i\not\equiv0$ for some $i$.
Denote by $u=u(t,x;\psi)$ the solution to \eqref{monotone-system} with the initial datum $u(0,x)\equiv\psi(x)$.
Let $\lambda_1$ be the generalized principal eigenvalue defined by \eqref{generalized-PE}.

Assume $\lambda_1<0$.
Then there exists $\epsilon>0$ independent of $\psi$ such that
$$
\liminf_{t\to+\infty}u(t,x;\psi)\geq\epsilon\bm{1}~~
\text{locally uniformly for $x\in\mathbb{R}^N$}.
$$
\end{proposition}
\begin{proof}
Applying Proposition \ref{generalized-principal-eigenvalue-properties} to problem \eqref{Dirichlet-PE} yields $\lambda^R_{1}\to\lambda_1$ as $R\to+\infty$.
Therefore, under the assumption $\lambda_1<0$, there exists $R_0>0$ large enough such that
$$
\forall R\geq R_0,~~~\lambda^R_{1}<0.
$$
In the sequel, fix some $R>R_0$ and
let $\varphi^R$ be the corresponding principal eigenfunction of \eqref{Dirichlet-PE}, normalized such that $\|\varphi^R\|_{L^\infty(B_R)^d}=1$.
Recall from our assumption \ref{item:smoothness} that $f(x,\bm{0})\equiv\bm{0}$ and
$f\in C^1(\mathbb{T}^N\times B_\infty(\bm{0},\sigma))^d$ with $\sigma>0$.
Thus, one can choose $\kappa_0\in(0,\sigma]$ small enough such that
\begin{equation}\label{C1-regularity-kappa}
\forall(x,\kappa)\in\mathbb{T}^N\times(0,\kappa_0],~~~
\left|D_uf(x,\bm{0})\kappa\bm{1}-f(x,\kappa\bm{1})\right|_\infty
\leq-\kappa\lambda^R_1.
\end{equation}
Due to $\bm{0}\ll\varphi^R\leq\bm{1}$, we then have that
\begin{align}\label{sub-stationary-equ}
\forall\kappa\in(0,\kappa_0],~~L(\kappa\varphi^R)-f(x,\kappa\varphi^R)
=D_uf(x,\bm{0})\kappa\varphi^R+\lambda^R_1\kappa\varphi^R
-f(x,\kappa\varphi^R)\leq\bm{0}~\text{ in }B_R,
\end{align}
where $L$ denotes a diagonal matrix of the operators given by \eqref{nondivergence}.

On the other hand, since the initial datum $\psi$ is assumed to be nonnegative and not identically equal to $\bm{0}$, we infer from the strong maximum principle coupled with assumption \ref{item:irreducibility} that
$u(1,x;\psi)\gg\bm{0}$ for all $x\in\mathbb{R}^N$.
Consequently, possibly by further decreasing $\kappa_0$ in \eqref{sub-stationary-equ},
$u(1,x;\psi)\gg\kappa_0\varphi^R(x)$ for all $x\in\overline{B_R}$.
For any $\kappa\in(0,\kappa_0]$, let $\psi^\kappa$ be the function defined by
$$
\psi^\kappa(x)=
\begin{dcases}
\kappa\varphi^R(x)&\text{ if }x\in B_R,\\
\bm{0}&\text{ otherwise}.
\end{dcases}
$$
From \eqref{sub-stationary-equ}, the function $\psi^\kappa$ then becomes a generalized subsolution of the stationary equation associated with \eqref{monotone-system} in $\mathbb{R}^N$.
Moreover, one has $u(1,x;\psi)\geq\psi^\kappa(x)$ for all $x\in\mathbb{R}^N$.
Consider now the solution $u=\underline{u}(t,x;\psi^\kappa)$ to \eqref{monotone-system} with the initial datum $\psi^\kappa$.
By the parabolic comparison principle, we obtain that
\begin{equation}\label{sub-Cauchy-problem}
\forall(t,x)\in(0,\infty)\times\mathbb{R}^N,~~
\bm{0}\ll\underline{u}(t,x;\psi^\kappa)\leq u(t+1,x;\psi),
\end{equation}
and that $\underline{u}$ is nondecreasing in $t$.
The classical parabolic estimates then imply that $\underline{u}$ converges in $C^2_{\rm loc}(\mathbb{R}^N)^d$ as $t\to+\infty$ to a bounded stationary solution to \eqref{monotone-system}, denoted as $p^\kappa$ satisfying
$\bm{0}\leq p^\kappa(x)\leq\widehat{\eta}\bm{1}$ for all $x\in\mathbb{R}^N$,
where $\widehat{\eta}$ is given in assumption \ref{item:upper-bound}.
Furthermore, due to
$$
p^\kappa(0)\geq\underline{u}(0,0;\psi^\kappa)=\kappa\varphi^R(0)\gg\bm{0},
$$
it follows from the strong maximum principle that $p^\kappa\gg\bm{0}$ in $\mathbb{R}^N$ for any $\kappa\in(0,\kappa_0]$,
and thus $\min_{x\in\overline{B_R}}p^\kappa(x)\gg\bm{0}$.
Our next aim is to prove that
\begin{equation}\label{sub-inital-datum}
p^\kappa(x)\geq\kappa_0\varphi^R(x)~~\text{ in }B_R.
\end{equation}
Note that one can define
$$
\gamma^*:=\sup\left\{\gamma>0\mid p^\kappa>\gamma\varphi^R
\,\text{ in }\overline{B_R}\right\}.
$$
We claim that $\gamma^*\geq\kappa_0$.
Indeed, if $\gamma^*<\kappa_0$, then there exist $x_0\in\overline{B_R}$ and $i_0\in\{1,\ldots,d\}$ such that
$$
\gamma^*\varphi^R_{i_0}(x_0)=p^\kappa_{i_0}(x_0)~\text{ and }~
\gamma^*\varphi^R(x)\leq p^\kappa(x)~~\text{in $\overline{B_R}$}.
$$
One necessarily has $x_0\in B_R$ since $\varphi^R\equiv\bm{0}$ on $\partial B_R$ while $\min_{\overline{B_R}}p^\kappa\gg\bm{0}$.
Observe also that due to $\gamma^*<\kappa_0$ and \eqref{sub-stationary-equ},
$\gamma^*\varphi^R$ is still a subsolution of the stationary equations associated with \eqref{monotone-system} in $B_R$.
From the strong elliptic maximum principle coupled with assumption \ref{item:irreducibility},
we deduce that $\gamma^*\varphi^R\equiv p^\kappa$ in $\overline{B_R}$,
which contradicts the values of both on $\partial B_R$.
We thereby conclude from \eqref{sub-Cauchy-problem} and \eqref{sub-inital-datum} that
$$
\liminf_{t\to+\infty}u(t,x;\psi)\geq\kappa_0\varphi^R(x)
~~\text{ in $\overline{B_{R/2}}$}.
$$

Now, recall that $R$ was chosen sufficiently large. Then, $[0,1]^N\subset B_{R/2}$.
Furthermore, since all coefficients of the operator $L$ and $f(\cdot,u)$ are $\mathbb{Z}^N$-periodic, it follows that for all $l\in\mathbb{Z}^N$ and any $\kappa\in(0,\kappa_0]$, $\underline{u}(t,x-l;\psi^\kappa(\cdot-l))$ satisfies the same system of equations as $\underline{u}(t,x;\psi^\kappa)$.
Note also that the Dirichlet principal eigenvalue $\lambda_1^R$ is invariant under periodic translations of the coefficients.
With the same arguments as above, we can obtain that
$$
\forall l\in\mathbb{Z}^N,~~~\liminf_{t\to+\infty}u(t,x;\psi)\geq\kappa_0\varphi^R(x-l)
~\text{ in }\overline{B_{R/2}(l)},
$$
where $B_{R}(l)$ denotes the ball of radius $R$ centered at $l\in\mathbb{Z}^N$.
Finally, since $\inf_{B_R}\varphi^R\gg\bm{0}$ for all $R>0$, then there exists a constant $\epsilon>0$ independent of $\psi$ such that
$$
\forall x\in\mathbb{R}^N,~~~\liminf_{t\to+\infty}u(t,x;\psi)\geq\epsilon\bm{1},
$$
which concludes the proof of Proposition \ref{local-uniform-persistence}.
\end{proof}
\begin{remark}
Recall from Remark \ref{sufficient-condition} that in general only $\lambda^p_1<\lambda_1$ holds,
where both quantities are given by \eqref{periodic-threshold} and \eqref{generalized-PE}, respectively.
It is worthy pointing out that in the case of $\lambda^p_1<0\leq\lambda_1$,
the persistence property relies strongly on the kind of additional conditions to be imposed on the initial data.
We refer to \cite{girardin2024persistence,nadin2010existence} for more details in a space-time periodic setting.
In particular, if the initial datum is periodic, a proof similar to the one above yields the uniform persistence of solutions associated with \eqref{monotone-system}.
\end{remark}

Following \cite{deng2023existence,SI-system} for the rational approximation to any direction of propagation, we are now in a position to complete the:
\begin{proof}
[Proof of Theorem \ref{main-result-existence}]
Fix any $e\in\mathbb{S}^{N-1}$.
Since $\mathbb{Q}^{N}\cap\mathbb{S}^{N-1}$ is a dense subset of $\mathbb{S}^{N-1}$ (see \cite[Lemma 4.1]{ding2021admissible}), there exists a sequence $\{\zeta^m\}_{m\in\mathbb{N}}\subset\mathbb{Q}^{N}\cap\mathbb{S}^{N-1}$ such that $\zeta^m\to e$ as $m\to+\infty$.
Moreover, for each given $c_e>c^*(e)$, we can choose a sequence $\{c^m\}_{m\in\mathbb{N}}\subset(c^*(\zeta^m),\infty)$ such that $c^m\to c_e$ as $m\to+\infty$ since the map $\mathbb{S}^{N-1}\ni e\mapsto c^*(e)$ is continuous.
Consider now the sequence of pulsating waves propagating in the direction $\zeta^m$ with speed $c^m$ for \eqref{monotone-system}:
$$
c^m>c^*(\zeta^m),~~u^m(t,x):=u(t,x;\zeta^m),~~~m\in\mathbb{N}.
$$
For each $m\in\mathbb{N}$, the vector-valued function $u^m$ satisfies
$$
u^m\left(t+\frac{k\cdot\zeta^m}{c^m},x\right)=u^m(t,x-k)
~\text{ for all }(t,x,k)\in\mathbb{R}\times\mathbb{R}^N\times\mathbb{Z}^N,
$$
and by \eqref{estimate-rational-c}, there holds
\begin{equation}\label{estimate-rational-sequence}
 \max\left(\bm{0},\sup_{n\in\mathbb{N}}\omega(t,x+n\tau^m_{1}\zeta^m;\zeta^m)\right)
 \leq u^m(t,x)\leq \min(\widehat{\eta}\bm{1},h(t,x;\zeta^m)),
\end{equation}
where the parameter $\tau^m_{1}$ is given by \eqref{tau1} with $\zeta=\zeta^m$.
From standard parabolic estimates, $\{u^m\}$ converges (up to extracting a subsequence) in
$C_{\rm loc}^{1,2}(\mathbb{R}\times\mathbb{R}^N)^d$  as $m\to+\infty$ to a bounded entire solution $u^\infty$ to \eqref{monotone-system}, which is nonnegative and verifies
\begin{equation}\label{general-diredtion-pulsating}
\forall(t,x,k)\in\mathbb{R}\times\mathbb{R}^N\times\mathbb{Z}^N,~~
u^\infty\left(t+\frac{k\cdot e}{c_e},x\right)=u^\infty(t,x).
\end{equation}
From statement \ref{item:continuity} of Proposition \ref{periodic-principal-eigenvalue-properties} and Proposition \ref{minimal-speed},
we find that the functions $h$ and $\omega$ given by \eqref{supersol-form-any-speed} and \eqref{subsol-form-larger-speeds} are continuous in both $e\in\mathbb{S}^{N-1}$ and $c\in(c^*(e),\infty)$ with respect to the uniform topology.
Note also that $u^\infty$ is bounded and $c_e>c^*(e)$.
We then infer from \eqref{estimate-rational-sequence} and the periodicity of equations that
\begin{equation}\label{general-diredtion-estimate}
\forall(t,x)\in\mathbb{R}\times\mathbb{R}^N,~~
\max(\bm{0},\omega(t,x;e))\leq u^\infty(t,x)\leq
\min(\widehat{\eta}\bm{1},h(t,x;e)),
\end{equation}
where
\begin{align*}
&h(t,x;e)=e^{-\lambda_{c_e}(e)(x\cdot e-c_et)}\phi_{\lambda_{c_e}(e)}(x),\\
&\omega(t,x;e)=h(t,x;e)-Ke^{-(\lambda_{c_e}(e)+\delta)
(x\cdot e-{c_e}t)}\phi_{\lambda_{c_e}(e)+\delta}(x),
\end{align*}
with $\lambda_{c_e}(e)>0$ given by \eqref{decay-rate}.
Here only such a lower bound stated in \eqref{general-diredtion-estimate} holds since $\tau^m_1\to0$ as $m\to+\infty$ by the definition of $\tau_1$ in \eqref{tau1}.
From the strong maximum principle coupled with assumption \ref{item:irreducibility}, one further gets $u^\infty\gg\bm{0}$.

It remains to verify the limiting condition \eqref{limiting-condition}.
Using the pulsating property of $u^\infty$ in \eqref{general-diredtion-pulsating} and the positivity of speed $c_e$ (remember that $c^*(e)>0$ under assumption $\lambda_1<0$),
it is equivalent to prove that
$$
\lim_{t\to-\infty}u^{\infty}(t,x)=\bm{0}~\text{ and }~
\liminf_{t\to+\infty}u^{\infty}(t,x)\gg\bm{0}
$$
locally uniformly for $x\in\mathbb{R}^N$.
The former follows immediately from \eqref{general-diredtion-estimate} and the fact that the function $h$ is componentwise decreasing in $t$.
To prove the latter limit, we consider a solution $\underline{u}:\mathbb{R}_+\times\mathbb{R}^N\to\mathbb{R}^d$ to \eqref{monotone-system} starting from $\underline{u}(0,x)\equiv\max(\bm{0},\omega(0,x;e))$.
Since $u^\infty$ is a solution to \eqref{monotone-system}, the parabolic comparison principle then ensures that
$$
\forall(t,x)\in\mathbb{R}_+\times\mathbb{R}^N,~~~
\underline{u}(t,x)\leq u^\infty(t,x)\leq\widehat{\eta}\bm{1}.
$$
Applying Proposition \ref{local-uniform-persistence} to $\underline{u}$, we obtain that there exists $T>0$ large enough such that
$$
\forall t\geq T,~\forall R>0,~~\inf_{|x|\leq R}u^\infty(t,x)\geq
\inf_{|x|\leq R}\underline{u}(t,x)\geq\epsilon\bm{1},
$$
whence the latter limit follows.
Thus, the function $u(t,x;e)\equiv u^\infty(t,x)$ is the desired pulsating travelling wave for \eqref{monotone-system} propagating along direction $e$ with speed $c_e$.

For the critical case $c_e=c^*(e)$, let us set
$$
c^m:=c^*(\zeta^m),~~u^{m}(t,x):=u^*(t,x;\zeta^m),~~~m\in\mathbb{N}.
$$
Passing to the limit $m\to+\infty$ as before, we can obtain the existence of a bounded entire solution of \eqref{monotone-system} verifying pulsating conditions as in \eqref{general-diredtion-pulsating} with $c_e=c^*(e)$,
and similarly, by \eqref{estimate-rational-c*} and Lemma \ref{supersol-construction-critical}, there also holds
$$
\forall(t,x)\in\mathbb{R}\times\mathbb{R}^N,~~
\max(\bm{0},\omega^*(t,x;e))\leq u^*(t,x;e)\leq
\min(\widehat{\eta}\bm{1},h^*(t,x;e)),
$$
where $\omega^*$ and $h^*$ are given by \eqref{subsol-form-critical} and \eqref{supersol-form-critical}.
The strong maximum principle coupled with assumption \ref{item:irreducibility} further yields $u^*(t,x;e)\gg\bm{0}$ in $\mathbb{R}\times\mathbb{R}^N$.
Furthermore, by Lemma \ref{supersol-construction-critical} and $c^*(e)>0$, it is straightforward to check that
$u^*(t,x)\to\bm{0}$ as $t\to-\infty$ locally uniformly for $x\in\mathbb{R}^N$.
With the same argument as before, by Proposition \ref{local-uniform-persistence}, we can also prove the other limit of $u^*$, that is,
$\liminf_{t\to+\infty}u^*(t,x;e)\gg\bm{0}$  locally uniformly for $x\in\mathbb{R}^N$.

This ends the proof of Theorem \ref{main-result-existence}.
\end{proof}
\section{Nonexistence of pulsating waves}\label{sect:nonexistence}
This section is devoted to the proof of Theorem \ref{main-result-nonexistence}.
The non-existence results will rest on the following proposition.
\begin{proposition}
[Extinction]\label{extinction}
Let $\psi=(\psi_1,\ldots,\psi_d)^T\in BUC(\mathbb{R}^N)^d$
be such that $\psi\geq\bm{0}$.
Denote by $u=u(t,x;\psi)$ the solution to \eqref{monotone-system} with the initial datum $u(0,x)\equiv\psi(x)$.
Let $\lambda_1^p$ be the periodic principal eigenvalue given by \eqref{periodic-threshold}.

Assume that \eqref{strict-sublinear-condition} holds.
If $\lambda_1^p\geq0$, then
$$
\lim_{t\to+\infty}u(t,x;\psi)=\bm{0}~~
\text{uniformly for $x\in\mathbb{R}^N$}.
$$
\end{proposition}
\begin{proof}
Consider the solution $\overline{u}=\overline{u}(t,x)$ of \eqref{monotone-system} with the initial datum given by
$$
\overline{u}(0,x)\equiv\bm{1}\max(\widehat{\eta},\|\psi\|_\infty),
$$
where the constant $\widehat{\eta}$ is given in \ref{item:upper-bound}.
It then follows from the comparison principle that
\begin{equation}\label{super-Cauchy-problem}
\forall(t,x)\in\mathbb{R}_+\times\mathbb{R}^N,~~
\bm{0}\leq u(t,x;\psi)\leq\overline{u}(t,x),
\end{equation}
and that $\overline{u}$ is nonincreasing in $t$.
Moreover, for all $t>0$, $\overline{u}(t,x)$ is periodic in $x$
because $\overline{u}(0,x)$ is constant.
As a consequence, the classical parabolic estimates imply that $\overline{u}$ converges as $t\to+\infty$ in $C^2(\mathbb{R}^N)^d$ to a $\mathbb{Z}^N$-periodic stationary solution to \eqref{monotone-system}, denoted as $U$,  that is,
\begin{equation}\label{stationary-U}
\begin{dcases}
LU(x)=f\left(x,U(x)\right)~~\text{for $x\in\mathbb{R}^N$},\\
\forall(x,k)\in\mathbb{R}^N\times\mathbb{Z}^N,~~U(x+k)=U(x),\\
\forall x\in\mathbb{R}^N,~~\bm{0}\leq U(x)\leq\overline{u}(0,x),
\end{dcases}
\end{equation}
where $L$ denotes a diagonal matrix of the operators given by \eqref{nondivergence}.

Next, we prove that $U=\bm{0}$.
To do so, let $\phi_0\in C^2(\mathbb{T}^N)^d$ denote the periodic principal eigenfunction of \eqref{periodic-PE} associated with $\lambda_1^p$.
Set
$$
\gamma^*:=\inf\left\{\gamma>0\mid\gamma\phi_0\gg U\text{ in }\mathbb{T}^N\right\}\geq0.
$$
This quantity is well-defined since $U$ is bounded and $\phi_0\gg\bm{0}$.
We claim that $\gamma^*=0$.
Suppose $\gamma^*>0$.
Then, $\gamma^*\phi_0\geq U$ and there exists $\{x_n\}_{n\in\mathbb{N}}\subset\mathbb{T}^N$ such that, for at least one index $i\in\{1,\ldots,d\}$, $(\gamma^*\phi_0-U)_i(x_n)\to0$ as $n\to\infty$.
By the compactness of $\mathbb{T}^N$, one may assume that $x_n\to x_\infty$ as $n\to\infty$, whence there holds $(\gamma^*\phi_0-U)_i(x_\infty)=0$ by continuity.
As $\lambda_1^p\geq0$, we infer from \eqref{sublinearity} and  \eqref{stationary-U} that
$$
L(\gamma^*\phi_0-U)-D_uf(x,\bm{0})(\gamma^*\phi_0-U)\geq
\lambda_1^p\gamma^*\phi_0\geq\bm{0}~~
\text{in }\mathbb{T}^N.
$$
Hence the strong maximum principle coupled with assumption \ref{item:irreducibility} yields $\gamma^*\phi_0\equiv U$.
However,
since $\phi_0\gg\bm{0}$ and $\gamma^*>0$, and by assumption \eqref{strict-sublinear-condition}, one has
$$
L(\gamma^*\phi_0)-f(x,\gamma^*\phi_0)=
\left(D_uf(x,\bm{0})+\lambda_1^p\right)
\gamma^*\phi_0-f(x,\gamma^*\phi_0)\geq(\not\equiv)\,\bm{0}~~
\text{in }\mathbb{T}^N.
$$
Recalling \eqref{stationary-U}, a contradiction has been achieved.
One gets $\gamma^*=0$ and thereby $U=\bm{0}$.

As a consequence, we infer from \eqref{super-Cauchy-problem} that
$$
\bm{0}\leq\lim_{t\to+\infty}u(t,x;\psi)\leq\lim_{t\to+\infty}\overline{u}(t,x)=\bm{0}
~~\text{ uniformly for }x\in\mathbb{R}^N.
$$
The proof of Proposition \ref{extinction} is complete.
\end{proof}
\begin{proof}[Proof of Theorem \ref{main-result-nonexistence}]
  Assume that $u=u(t,x)$ is a pulsating travelling wave of \eqref{monotone-system} according to Definition \ref{pulsating-wave-def}.
  Consider the solution $u=u(t,x;\widehat{\eta}\bm{1})$ to \eqref{monotone-system} starting from $u(0,x;\widehat{\eta}\bm{1})\equiv\widehat{\eta}\bm{1}$.
  Since $u(t,x)$ is bounded from above by $\widehat{\eta}\bm{1}$, the parabolic comparison principle ensures that for each $t\in\mathbb{R}$ and $s\in\mathbb{R}$, there holds
$$
\forall t\geq s,~\forall x\in\mathbb{R}^N,~~
\bm{0}\leq u(t,x)\leq u(t-s,x;\widehat{\eta}\bm{1}).
$$
By Proposition \ref{extinction}, $u(t-s,x;\widehat{\eta}\bm{1})$ converges to $\bm{0}$ as $s\to-\infty$ locally uniformly for $t\in\mathbb{R}$ and uniformly for $x\in\mathbb{R}^N$,
which implies that $u(t,x)\equiv\bm{0}$. This leads to a contradiction.
\end{proof}

\section*{Data availability statement}
\noindent
No new data were created or analysed in this study.

\section*{Acknowledgements}
\noindent
L. Deng's research is supported by the China Postdoctoral Science Foundation
(No. 2025M773053) and in part by the NSF of China (No.12471203).
He would like to thank the last two authors for generous hospitality during his visit to the
Universit\'{e} Le Havre Normandie,
where this work was initiated.
He would also like to express his sincere gratitude to Professors Xuefeng Wang and Xing Liang
for supporting his postdoctoral position.

\noindent
The authors are very grateful to the referees for their valuable comments and suggestions which helped to improve our original manuscript, as well as for pointing out some inspiring problems for further study.

\section*{Appendix}
\appendix
\section{Proof of Proposition
\ref{periodic-principal-eigenvalue-properties}}
\label{appendix:strict-concavity}
Statement \ref{item:existence-uniqueness} follows from the Krein-Rutman theory.
The \emph{max-min} characterization in \ref{item:max-min} is quite similar to the one by Sweers \cite[Remark 3.1]{sweers1992strong} (see also \cite{birindelli1999existence}).
The analyticity of $\lambda\mapsto k(\lambda,e)$ ($\phi_{\lambda e}$ as well) is  standard owing to Kato's perturbation theory (see \cite[Section 7.2]{Kato-book}).
Estimate \eqref{upper-estimate-principal-eigenvalue} is easy to derive by virtue of formula \eqref{max-min}.
Statement \ref{item:continuity} is a direct consequence of the foregoing properties.
Let us now prove the rest of the results.
\begin{proof}
[Proof of Proposition
\ref{periodic-principal-eigenvalue-properties} \ref{item:concavity}]
Only the concavity remains to be proved.
Following
\cite[Lemma 3.1]{berestycki2005analysis-2},
let us first show that for any given $e\in\mathbb{S}^{N-1}$ the function $\lambda\mapsto k(\lambda,e)$ is concave in $\mathbb{R}$.
One has to check that
\begin{equation}\label{k-concavity}
\forall\hat{\lambda}, \tilde{\lambda}\in\mathbb{R},~\forall r\in(0,1),~~ k\big(r\hat{\lambda}+(1-r)\tilde{\lambda},e\big)
\geq
rk(\hat{\lambda},e)+(1-r)k(\tilde{\lambda},e).
\end{equation}
Let $E_\lambda$ be the set defined by
\begin{equation}\label{E-lambda-def}
E_\lambda=\left\{\psi\in C^2(\mathbb{R}^N)^d\mid\psi\gg\bm{0}\text{ and }
\psi e^{\lambda e\cdot x}\text{ is $\mathbb{Z}^N$-periodic}\right\}.
\end{equation}
Then, formula \eqref{max-min} can be rewritten as
\begin{equation}\label{k-formula}
  k(\lambda,e)=\max_{\psi\in E_\lambda}
    \min_{1\leq i\leq d}\inf_{x\in\mathbb{T}^N}
    \frac{\big(e^i(L-H)\psi\big)(x)}{\psi_i(x)},
\end{equation}
where $L={\rm diag}(L^1,\ldots,L^d)$ denotes the operator given by \eqref{nondivergence}.
Let now $\hat{\psi}\in E_{\hat{\lambda}}$ and $\tilde{\psi}\in E_{\tilde{\lambda}}$ be  arbitrarily chosen, respectively.
Define
\begin{align*}
\hat{z}=(\hat{z}_1,\ldots,\hat{z}_d)^T=
\big(\ln\hat{\psi}_1,\ldots,\ln\hat{\psi}_d\big)^T~\text{ and }~
\tilde{z}=(\tilde{z}_1,\ldots,\tilde{z}_d)^T=
\big(\ln\tilde{\psi}_1,\ldots,\ln\tilde{\psi}_d\big)^T.
\end{align*}
Set $z:=r\hat{z}+(1-r)\tilde{z}$, $\lambda:=r\hat{\lambda}+(1-r)\tilde{\lambda}$, and $\psi(x):=(e^{z_1(x)},\ldots,e^{z_d(x)})^T$.
One can easily verify that $\psi\in E_\lambda$.
Hence, it can be used as a test function in \eqref{k-formula} so that
$$
k(\lambda,e)\geq\min_{1\leq i\leq d}\inf_{x\in\mathbb{T}^N}
\frac{\big(e^i(L-H)\psi\big)(x)}{\psi_i(x)}.
$$
Direct computations show that for each $i\in\{1,\ldots,d\}$,
\begin{align*}
&\frac{L^i\psi_i}{\psi_i}
=\frac{1}{\psi_i}\left[-\tr(A^iD^2\psi_i)+q^i\cdot\nabla\psi_i\right]
=-\tr(A^iD^2z_i)-\nabla z_iA^i\nabla z_i+q^i\cdot\nabla z_i,\\
&\frac{e^iB\psi}{\psi_i}=\frac{1}{\psi_i}\sum^d_{j=1}b_{ij}\psi_j
=\sum^d_{j=1}b_{ij}e^{r(\hat{z}_j-\hat{z}_i)+(1-r)(\tilde{z}_j-\tilde{z}_i)}.
\end{align*}
Moreover, due to $0<r<1$ and by \eqref{elliptic-condition}, one has
\begin{equation}\label{scalling}
\begin{split}
\nabla z_iA^i\nabla z_i
=&r\nabla\hat{z}_iA^i\nabla\hat{z}_i
-r(1-r)(\nabla\hat{z}_i-\nabla\tilde{z}_i)A^i(\nabla\hat{z}_i-\nabla\tilde{z}_i)
+(1-r)\nabla\tilde{z}_iA^i\nabla\tilde{z}_i\\
\leq& r\nabla\hat{z}_iA^i\nabla\hat{z}_i+
(1-r)\nabla\tilde{z}_iA^i\nabla\tilde{z}_i,
\end{split}
\end{equation}
and also
\begin{align*}
e^{r(\hat{z}_j-\hat{z}_i)+(1-r)(\tilde{z}_j-\tilde{z}_i)}
\leq re^{\hat{z}_j-\hat{z}_i}+(1-r)e^{\tilde{z}_j-\tilde{z}_i}
\end{align*}
since the function $x\mapsto e^x$ is convex.
As a consequence, we obtain
\begin{align*}
\frac{e^i(L-H)\psi}{\psi_i}
\geq&r\left(-\tr\left(A^i D^2\hat{z}_i\right)-\nabla\hat{z}_iA^i\nabla\hat{z}_i
+q^i\cdot\nabla\hat{z}_i-\sum^d_{j=1}h_{ij}e^{\hat{z}_j-\hat{z}_i}\right)\\
&+(1-r)\left(-\tr\left(A^iD^2\tilde{z}_i\right)-\nabla\tilde{z}_iA^i\nabla\tilde{z}_i
+q^i\cdot\nabla\tilde{z}_i-\sum^d_{j=1}h_{ij}e^{\tilde{z}_j-\tilde{z}_i}\right)
\\
\geq&r\left(\frac{-\tr\big(A^iD^2\hat{\psi}_i\big)
+q^i\cdot\nabla\hat{\psi}_i}{\hat{\psi}_i}
-\sum^d_{j=1}h_{ij}\frac{\hat{\psi}_j}{\hat{\psi}_i}\right)\\
&+(1-r)\left(\frac{-\tr\big(A^iD^2\tilde{\psi}_i\big)
+q^i\cdot\nabla\tilde{\psi}_i}{\tilde{\psi}_i}
-\sum^d_{j=1}h_{ij}\frac{\tilde{\psi}_j}{\tilde{\psi}_i}\right)
\\
=&r\frac{e^i(L-H)\hat{\psi}}{\hat{\psi}_i}+
(1-r)\frac{e^i(L-H)\tilde{\psi}}{\tilde{\psi}_i},
~~~~i=1,\ldots,d.
\end{align*}
This implies that
$$
k(\lambda,e)\geq\min_{1\leq i\leq d}\inf_{\mathbb{T}^N}
\frac{e^i(L-H)\psi}{\psi_i}
\geq r\min_{1\leq i\leq d}\inf_{\mathbb{T}^N}\frac{e^i(L-H)\hat{\psi}}{\hat{\psi}_i}+
(1-r)\min_{1\leq i\leq d}\inf_{\mathbb{T}^N}\frac{e^i(L-H)\tilde{\psi}}{\tilde{\psi}_i}.
$$
Eventually, since $\hat{\psi}$ and $\tilde{\psi}$ were arbitrarily chosen, \eqref{k-concavity} then follows.

Next, inspired by \cite[Proposition 2.2]{griette2021propagation}, we prove the strict concavity of $\lambda\mapsto k(\lambda,e)$.
Assume by contradiction that there exist two real numbers $\hat{\lambda}<\tilde{\lambda}$ such that \eqref{k-concavity} for $r=1/2$ (without loss of generality) is an equality.
Let $\phi_{\hat{\lambda}e}$ and $\phi_{\tilde{\lambda}e}$ be the unique principal eigenfunctions of \eqref{nonsymmetric-periodic-PE} in the sense of \eqref{normalization-principal-eigenfunction} corresponding to $\lambda=\hat{\lambda}$ and $\tilde{\lambda}$, respectively.
Set
\begin{equation}\label{strict-concavity-proof}
\hat{\psi}:=\phi_{\hat{\lambda}e}e^{-\hat{\lambda} e\cdot x}\in E_{\hat{\lambda}},~~~
\tilde{\psi}:=\phi_{\tilde{\lambda}e}e^{-\tilde{\lambda} e\cdot x}\in  E_{\tilde{\lambda}},
\end{equation}
and
\begin{equation}\label{test-function}
\phi=(\phi_1,\ldots,\phi_d)^T=
\left(\sqrt{(\phi_{\hat{\lambda}e})_1(\phi_{\tilde{\lambda}e})_1},
\ldots,\sqrt{(\phi_{\hat{\lambda}e})_d(\phi_{\tilde{\lambda}e})_d}\,\right)^T.
\end{equation}
Such a choice then yields
$$
\psi:=\phi e^{-\frac{\hat{\lambda}+\tilde{\lambda}}{2}e\cdot x}\in E_{(\hat{\lambda}+\tilde{\lambda})/2},
$$
where $E_{(\hat{\lambda}+\tilde{\lambda})/2}$ is given by \eqref{E-lambda-def}.
Thus,
$$
k\big(({\hat{\lambda}+\tilde{\lambda}})/{2},e\big)\geq
\min_{1\leq i\leq d}\inf_{\mathbb{T}^N}\frac{e^i(L-H)\psi}{\psi_i}\geq
\frac{1}{2}\big(k(\hat{\lambda},e)+k(\tilde{\lambda},e)\big).
$$
Combining our contradiction hypothesis with
formula \eqref{max-min}, one gets
\begin{align*}
\min_{1\leq i\leq d}\inf_{\mathbb{T}^N}
\frac{e^i\Big(L_{({\hat{\lambda}+\tilde{\lambda}})e/{2}}-H\Big)\phi}{\phi_i}
=
\frac{1}{2}\min_{1\leq i\leq d}\inf_{\mathbb{T}^N}
\frac{e^i(L_{\hat{\lambda}e}-H)\phi_{\hat{\lambda}e}}{(\phi_{\hat{\lambda}e})_i}+
\frac{1}{2}\min_{1\leq i\leq d}\inf_{\mathbb{T}^N}
\frac{e^i(L_{\tilde{\lambda}e}-H)\phi_{\tilde{\lambda}e}}{(\phi_{\tilde{\lambda}e})_i}.
\end{align*}
This means that the function $\phi$, defined by \eqref{test-function}, is a principal eigenfunction corresponding to $k((\hat{\lambda}+\tilde{\lambda})/2,e)$.
From our preceding computations, and in particular \eqref{scalling}, we then have that
${\nabla\hat{\psi}_i}/{\hat{\psi}_i}={\nabla\tilde{\psi}_i}/{\tilde{\psi}_i}$
for all $i\in\{1,\ldots,d\}$.
By \eqref{strict-concavity-proof}, we further obtain
$$
\frac{\nabla(\phi_{\hat{\lambda}e})_i}{(\phi_{\hat{\lambda}e})_i}-
\frac{\nabla(\phi_{\tilde{\lambda}e})_i}{(\phi_{\tilde{\lambda}e})_i}
=(\hat{\lambda}-\tilde{\lambda})e,~~i=1,\ldots,d.
$$
By integrating these equations over $\mathbb{T}^N$, one can deduce that $\hat{\lambda}=\tilde{\lambda}$.
This leads to a contradiction, thereby completing the proof.
\end{proof}
\begin{proof}[Proof of Proposition \ref{periodic-principal-eigenvalue-properties} \ref{item:maximum}]
The proof is inspired by \cite[Proposition 4.4]{berestycki2008asymptotic} (see also \cite[Theorem 2.12]{nadin2009principal}).
Firstly, by Proposition \ref{generalized-principal-eigenvalue-properties}, there  always exists a generalized principal eigenfunction $\varphi\in C^2(\mathbb{R}^N)^d$ associated with $\lambda_1$, that is,
\begin{equation}\label{generalized-principal-eigenfunction}
(L-H)\varphi=\lambda_1\varphi,~~~\varphi\gg\bm{0}~~\text{in $\mathbb{R}^N$}.
\end{equation}
Here $\varphi=(\varphi_1,\ldots,\varphi_d)^T$ is not necessarily bounded, nor uniformly positive (that is, it is possible that $\inf_{\mathbb{R}^N}\min_{1\leq i\leq d}\varphi_i=0$).
Moreover, if $\phi_{\lambda e}\in C^2(\mathbb{T}^N)^d$ is a principal eigenfunction of \eqref{nonsymmetric-periodic-PE}, then the function $\mathbb{R}^N\ni x\mapsto e^{\lambda e\cdot x}\phi_{\lambda e}$ satisfies \eqref{generalized-principal-eigenfunction}.
By \eqref{generalized-principal-eigenvalue}, one has $\lambda_1\geq k(\lambda,e)$  for all  $\lambda\in\mathbb{R}$ and $e\in\mathbb{S}^{N-1}$ and thus  $\lambda_1\geq\sup_{\lambda e\in\mathbb{R}^N}k(\lambda,e)$.
Note also that for any given $e\in\mathbb{S}^{N-1}$, $\lambda\mapsto k(\lambda,e)$ is strictly concave.
Therefore, to conclude it suffices to show that there exists $\alpha\in\mathbb{R}^N$ such that $\lambda_1=k(\alpha)$.

Let $e^1=(1,0,\ldots,0)^T\in\mathbb{R}^d$ and define a function $\psi=(\psi_1,\ldots,\psi_d)^T(x)$ for $x\in\mathbb{R}^N$ as
\begin{equation}\label{psi-def-shift}
\psi_i(x)=\frac{\varphi_i(x+e^1)}{\varphi_i(x)},~~i=1,\ldots,d.
\end{equation}
Applying the Harnack inequality from Theorem \ref{elliptic-Har-ineq} to \eqref{generalized-principal-eigenfunction},
we deduce that for any fixed $R>0$, there exists a constant $C_R>0$ such that for each $x\in\mathbb{R}^N$,
$$
\forall y\in B_{R+1}(x),~~~
\max_{1\leq i\leq d}\varphi_i(y)\leq C_R
 \min_{1\leq i\leq d}\varphi_i(x).
$$
We then infer from the periodicity of equations that the function $\psi$ defined by \eqref{psi-def-shift} is globally bounded.
Set
\begin{equation}\label{m-def}
M_1:=\sup_{x\in\mathbb{R}^N}\max_{1\leq i\leq d}\psi_i(x).
\end{equation}
Then there exist $\{x_n\}_{n\in\mathbb{N}}\subset\mathbb{R}^N$ and $i_0\in\{1,\ldots,d\}$ such that $\psi_{i_0}(x_n)\to M_1$ as $n\to\infty$.
For each $n\in\mathbb{N}$, let $\bar{x}_n\in[0,1]^N$ be such that $x_n-\bar{x}_n\in\mathbb{Z}^N$.
Up to extracting a subsequence,  $\bar{x}_n\to\bar{x}_\infty\in[0,1]^N$ as $n\to\infty$.
Now, we define two sequences $\{\widetilde{\varphi}^{\,n}\}_{n\in\mathbb{N}}$ and $\{\widehat{\varphi}^{\,n}\}_{n\in\mathbb{N}}$ of vector-valued functions as follows:
\begin{equation}\label{modified-two-shift-sequences}
\begin{split}
&\widetilde{\varphi}^{\,n}_i(x):=\frac{\varphi_i(x+x_n)}{\varphi_{i_0}(x_n)}
~~\text{ and}\\
&\widehat{\varphi}_i^{\,n}(x):=\frac{\varphi_i(x+e^1+x_n)-M_1\varphi_i(x+x_n)}
{\varphi_{i_0}(x_n)},~~i=1,\ldots,d.
\end{split}
\end{equation}
For each $n\in\mathbb{N}$, $\widetilde{\varphi}^{\,n}$ satisfies  \eqref{generalized-principal-eigenfunction} with a spatial shift of $\bar{x}_n$.
Applying the Harnack inequality again, we deduce that $\widetilde{\varphi}^{\,n}(x)$ is locally bounded in $\mathbb{R}^N$.
Indeed, due to $\widetilde{\varphi}^{\,n}_{i_0}(0)=1$, for any given $R>0$, there exists a constant $C_R>0$ such that
$$
\sup_{x\in B_{R}}\max_{1\leq i\leq d}\widetilde{\varphi}^{\,n}_i(x)\leq C_R
 \inf_{x\in B_{R}}\min_{1\leq i\leq d}\widetilde{\varphi}^{\,n}_i(x)\leq C_R,
 ~~~n\in\mathbb{N}.
$$
From Schauder interior estimates and Sobolev's injections, there exists a subsequence of $\{\widetilde{\varphi}^{\,n}\}$ that converges in
$C^{2+\alpha^\prime}(\overline{B_{R/2}})^d$ for all $\alpha^\prime\in[0,\alpha)$.
Since the convergence holds for any $R>0$, a diagonal argument allows us to extract a particular subsequence (not relabeled) such that $\widetilde{\varphi}^{\,n}$ converges in
$C_{\rm loc}^{2+\alpha^\prime}(\mathbb{R}^N)^d$ to a function $\widetilde{\varphi}^{\,\infty}$ satisfying
$$
\begin{dcases}
L[x+\bar{x}_\infty,\widetilde{\varphi}^{\,\infty}]-H(x+\bar{x}_\infty)
\widetilde{\varphi}^{\,\infty}=\lambda_1\widetilde{\varphi}^{\,\infty},\\
\widetilde{\varphi}^{\,\infty}\geq\bm{0}~~\text{in $\mathbb{R}^N$},~~~
\widetilde{\varphi}_{i_0}^{\,\infty}(0)=1.
\end{dcases}
$$
It then follows from the strong maximum principle, together with the fact that $H$ is cooperative and fully coupled, that $\widetilde{\varphi}^{\,\infty}(x)\gg\bm{0}$ for all $x\in\mathbb{R}^N$.

On the other hand, due to the periodicity of equations, the function $\widehat{\varphi}^{\,n}$ defined in \eqref{modified-two-shift-sequences} also satisfies \eqref{generalized-principal-eigenfunction} with a spatial shift of $\bar{x}_n$.
Moreover, it follows from \eqref{psi-def-shift} and \eqref{m-def} that $\widehat{\varphi}^{\,n}(x)\leq\bm{0}$ for all $x\in\mathbb{R}^N$ and
$\widehat{\varphi}^{\,n}_{i_0}(0)=\psi_{i_0}(x_n)-M_1$.
From standard elliptic estimates, we have that $\widehat{\varphi}^{\,n}(x)\to\widehat{\varphi}^{\,\infty}(x)$ as $n\to\infty$, after passing to a subsequence, locally uniformly for $x\in\mathbb{R}^N$, where $\widehat{\varphi}^{\,\infty}$ satisfies
$$
\begin{dcases}
L[x+\bar{x}_\infty,\widehat{\varphi}^{\,\infty}]-
H(x+\bar{x}_\infty)\widehat{\varphi}^{\,\infty}=
\lambda_1\widehat{\varphi}^{\,\infty},\\
\widehat{\varphi}^{\,\infty}\leq\bm{0}~~\text{in $\mathbb{R}^N$},~~
\widehat{\varphi}^{\,\infty}_{i_0}(0)=0.
\end{dcases}
$$
By the strong maximum principle, we deduce that $\widehat{\varphi}^{\,\infty}=\bm{0}$.
Recalling the definition of $\widehat{\varphi}^{\,n}$ in \eqref{modified-two-shift-sequences}, this implies in particular that
$$
\forall x\in\mathbb{R}^N,~~~
\widetilde{\varphi}^{\,\infty}(x+e^1)=
M_1\widetilde{\varphi}^{\,\infty}(x)\gg\bm{0}.
$$

As a consequence, if we set $\alpha_1:=\ln M_1$, then
$\phi_{\alpha_1}(x):=\widetilde{\varphi}^{\,\infty}(x)e^{\alpha_1x_1}$
is $1$-periodic in $x_1$.
Since $\widetilde{\varphi}^{\,\infty}$ is a generalized principal eigenfunction of the operator $L[\cdot+\bar{x}_\infty]-H(\cdot+\bar{x}_\infty)$ in $\mathbb{R}^N$ and $\widetilde{\varphi}_{i_0}^{\,\infty}(0)=1$, we then obtain
$$
\begin{dcases}
e^{\alpha_1x_1}L[x+\bar{x}_\infty,\phi_{\alpha_1}e^{-\alpha_1x_1}]=
H(x+\bar{x}_\infty)\phi_{\alpha_1}
+\lambda_1\phi_{\alpha_1},\\
\phi_{\alpha_1}\gg\bm{0}~\text{ in $\mathbb{R}^N$},~~~(\phi_{\alpha_1})_{i_0}(0)=1.
\end{dcases}
$$
Now let $e^2=(0,1,0,\ldots,0)^T\in\mathbb{R}^d$ and modify the definition \eqref{psi-def-shift} of $\psi$ as follows:
$$
\psi_i(x)=\frac{(\phi_{\alpha_1})_i(x+e^2)}{(\phi_{\alpha_1})_i(x)},~~~i=1,\ldots,d,
$$
which is also globally bounded by the Harnack inequalities once again.
Set
$$
M_2:=\sup_{x\in\mathbb{R}^N}\max_{1\leq i\leq d}\psi_i(x).
$$
Repeating the above proof, we obtain a function $\widetilde{\phi}^\infty_{\alpha_1}$ such that the following function
$$
\phi_{\alpha_1e^1+\alpha_2e^2}(x):=\widetilde{\phi}^\infty_{\alpha_1}(x)e^{\alpha_1x_1+\alpha_2x_2}
~~\text{ with $\alpha_2:=\ln M_2$}
$$
is $\mathbb{Z}^2$-periodic in $(x_1,x_2)$.
Going on the above construction until the $N$-th time, one can find $\alpha_i=\ln M_i$ for all $i\in\{3,\ldots,N\}$ and then get a function $\varphi$ verifying
$$
\begin{dcases}
L[x+\bar{z}_\infty,\varphi]-H(x+\bar{z}_\infty)\varphi=\lambda_1\varphi~~
\text{ for $(x,\bar{z}_\infty)\in\mathbb{R}^N\times\mathbb{R}^N$},\\
\varphi\gg\bm{0},~~\varphi_{i_0}(0)=1~\text{ and }
\varphi(x)e^{\alpha\cdot x}
\text{ is $\mathbb{Z}^N$-periodic in $x$},
\end{dcases}
$$
where one has set $\alpha:=(\alpha_1,\ldots,\alpha_N)^T\in\mathbb{R}^N$.
Define now $\phi_{\alpha}(x):=\varphi(x)e^{\alpha\cdot x}$ and recall that
$$
e^{\alpha\cdot x}L^i(\phi_{\alpha}e^{-\alpha\cdot x})=L^i_\alpha\phi_{\alpha}~~
\text{for }(\alpha,x)\in\mathbb{R}^N\times\mathbb{R}^N,
$$
where $L^i_\alpha$ denotes the operator given by \eqref{modified-operator} with $\lambda e=\alpha$.
Set $L_\alpha:={\rm diag}\left(L^1_{\alpha},\ldots,L^d_{\alpha}\right)$. Then,
$$
\begin{dcases}
L_\alpha[x+\bar{z}_\infty,\phi_{\alpha}]-
H(x+\bar{z}_\infty)\phi_{\alpha}=\lambda_1\phi_{\alpha}~~
\text{ for $(x,\bar{z}_\infty)\in\mathbb{R}^N\times\mathbb{R}^N$},\\
\forall(x,k)\in\mathbb{R}^N\times\mathbb{Z}^N,~~\phi_{\alpha}(x+k)=\phi_{\alpha}(x),~~
\phi_{\alpha}(x)\gg\bm{0}.
\end{dcases}
$$
Since the periodic principal eigenvalue of \eqref{periodic-principal-eigenvalue-properties} is invariant under a translation of the coefficients, it follows from the uniqueness of the principal eigenpairs that
$C\phi_{\alpha}$ for some constant $C>0$ is the principal eigenfunction of \eqref{periodic-principal-eigenvalue-properties}, and $\lambda_1=k(\alpha)$.
This concludes the proof.
\end{proof}

\section{Harnack-type estimates for linear cooperative and fully coupled systems}
\label{appendix:Harnack-inequa}
For the sake of convenience, we present the Harnack inequalities for linear cooperative and fully coupled systems in the context of the present work.
We refer to \cite[Theorem 1.2]{Chen1997Harnack}, \cite[Theorem 2.2]{Arapostathis1999Harnack} and \cite[Corollary 8.1]{Busca2004Harnack} for results on elliptic systems.
\begin{theorem}[Elliptic case]\label{elliptic-Har-ineq}
Let $L={\rm diag}(L^1,\ldots,L^d)$ denote a diagonal matrix of the linear elliptic operators given by \eqref{nondivergence}.
Let $H\in C^\alpha(\mathbb{T}^N,\mathcal{M}_{d}(\mathbb{R}))$ be cooperative and fully coupled in the sense of Definition \ref{cooperative-irreducible}.
  Consider the system of linear elliptic equations
  \begin{equation}\label{elliptic-Har-equation}
  (L-H)u=\bm{0}~~\text{in }\mathbb{R}^N.
  \end{equation}
Denote by $\bar{h}_0$ the smallest positive entry of $\overline{H}=(\overline{h}_{ij})_{1\leq i, j\leq d}$ where
$\overline{h}_{ij}:=\max_{x\in\mathbb{T}^N}h_{ij}(x)$.

 There exists a constant $C_R>0$ determined only by $N$, $d$, the coefficients of $L$ and $\bar{h}_0$ such that
  if
  $u=(u_1,\ldots,u_d)^T\in C(\overline{B_{R+1}})^d$
  is a nonnegative solution to \eqref{elliptic-Har-equation},
  then
  $$
  \max_{1\leq i\leq d}\sup_{x\in B_{R}}u_i(x)\leq C_R
  \min_{1\leq i\leq d}\inf_{x\in B_{R}}u_i(x).
  $$
\end{theorem}
In view of parabolic systems, we refer to \cite[Theorem 3.9]{foldes2009cooperative}, which assumes \emph{pointwise irreducibility} that is slightly stronger than the fully coupled property.
Girardin and Mazari have recently revisited and refined such estimates in the space-time periodic setting, see \cite[Proposition 2.4]{girardin2023generalized} for details.
\begin{theorem}[Parabolic case]\label{Har-ineq}
Let $H(t,x)=(h_{ij}(t,x))_{1\leq i, j\leq d}$ be a cooperative and fully coupled matrix field according to Definition \ref{cooperative-irreducible}.
Assume that each entry $h_{ij}:\mathbb{R}\times\mathbb{R}^N\to\mathbb{R}$ is bounded and continuous.
  Consider the system of linear parabolic equations
  \begin{equation}\label{parabolic-Har-equation}
  \partial_{t}u_i-L^iu_i=\sum_{j=1}^{d}h_{ij}(t,x)u_j,
  ~~~(t,x)\in\mathbb{R}\times\mathbb{R}^N,~~i=1,\ldots,d,
  \end{equation}
where $L^i$ is defined by \eqref{nondivergence}.
Denote by $\bar{h}_0$ be the smallest positive entry of $\overline{H}=(\overline{h}_{ij})_{1\leq i, j\leq d}$ where
$\overline{h}_{ij}:=\max_{(t,x)\in\mathbb{R}\times\mathbb{R}^N}h_{ij}(t,x)$.

Let $\theta>1$ be given.  There exists a constant $\widehat\kappa_{\theta}>0$ determined only by $N$, $d$, the coefficients of $L^i$ and $\bar{h}_0$ such that
  if
  $u=(u_1,\ldots,u_d)^T\in C([-2\theta,2\theta]
  \times[-\theta,\theta]^N)^d$
  is a nonnegative solution to \eqref{parabolic-Har-equation},
  then
  $$
  \min_{1\leq i\leq d}\min_{(t,x)\in\left[\frac{3\theta}{2},2\theta\right]
  \times[-\theta,\theta]^N}u_i(t,x)\geq
  \widehat\kappa_{\theta}\max_{1\leq i\leq d}
  \max_{(t,x)\in[-\theta,0]\times[-\theta,\theta]^N}u_i(t,x).
  $$
\end{theorem}
Note that in \cite[Proposition 2.4]{girardin2023generalized}, only the diagonal entries of the zero-order term matrix are not required to be periodic.
But going their proof, the Harnack inequality from Theorem \ref{Har-ineq} still holds as long as $H$ is globally bounded.

\bibliographystyle{amsplain}
\begin{footnotesize}
\bibliography{bibref_DDG}
\end{footnotesize}
\end{document}